\documentclass[12pt,reqno]{amsart}

\usepackage{mathdots}

\usepackage{color}

\usepackage{pdfsync}

\usepackage{enumitem}

\usepackage{wasysym}

\usepackage{amssymb}

\usepackage{mathrsfs}

\DeclareMathAlphabet{\mathpzc}{OT1}{pzc}{m}{it}

\usepackage[all]{xy}

\usepackage{tikz}
\usetikzlibrary{arrows,decorations.pathmorphing,decorations.pathreplacing,positioning,shapes.geometric,shapes.misc,decorations.markings,decorations.fractals,calc,patterns}

\usepackage{float}
\usepackage[bottom]{footmisc}
\usepackage{mathtools}

\newcommand{\vertgeq}{\rotatebox{90}{$\,<\;\:$}}
\newcommand{\vertleq}{\rotatebox{90}{$\,>\;\:$}}

\entrymodifiers={+!!<0pt,\fontdimen22\textfont2>}

\def\cC{\mathscr{C}}

\def\cS{\mathscr{S}}
\def\cT{\mathscr{T}}

\def\BQ{\mathbb{Q}}

\def\BZ{\mathbb{Z}}

\def\fS{\mathfrak{S}}
\def\fT{\mathfrak{T}}

\def\fa{\mathfrak{a}}
\def\fb{\mathfrak{b}}

\def\fv{\mathfrak{v}}

\def\adots{\mathinner{\mkern1mu\raise1.0pt\vbox{\kern7.0pt\hbox{.}}\mkern2mu\raise4.0pt\hbox{.}\mkern2mu\raise7.0pt\hbox{.}\mkern1mu}}

\def\dddots{\mathinner{\mkern1mu\raise10.0pt\vbox{\kern7.0pt\hbox{.}}\mkern2mu\raise5.3pt\hbox{.}\mkern2mu\raise1.0pt\hbox{.}\mkern1mu}}
\def\dddotssmall{\mathinner{\mkern1mu\raise7.0pt\vbox{\kern7.0pt\hbox{.}}\mkern-1mu\raise4pt\hbox{.}\mkern-1mu\raise1.0pt\hbox{.}\mkern1mu}}
\def\defect{\operatorname{def}}

\def\I{\mathrm{I}}
\def\II{\mathrm{II}}
\def\III{\mathrm{III}}
\def\IV{\mathrm{IV}}

\def\J{\mathrm{J}}
\def\K{\mathrm{K}}

\def\SL2{\operatorname{SL}_2}

\numberwithin{equation}{section}

\renewcommand{\labelenumi}{(\roman{enumi})}

\newtheorem{Lemma}{Lemma}[section]
\newtheorem{Theorem}[Lemma]{Theorem}
\newtheorem{Proposition}[Lemma]{Proposition}

\theoremstyle{definition}
\newtheorem{Definition}[Lemma]{Definition}
\newtheorem{Setup}[Lemma]{Setup}

\newtheorem{Construction}[Lemma]{Construction}
\newtheorem{Remark}[Lemma]{Remark}

\newtheorem{Notation}[Lemma]{Notation}

\newtheorem{Description}[Lemma]{Description}

\newtheorem*{bfhpg*}{}

\begin{document}

\setlength{\parindent}{0pt}
\setlength{\parskip}{7pt}
%The default \baselineskip is close to 4.8mm
%\setlength{\baselineskip}{5.8mm}

\title[All $\SL2$-tilings]{All $\SL2$-tilings come from infinite
  triangulations}

\author{Christine Bessenrodt}
\address{Institut f\"{u}r Algebra, Zahlentheorie und Diskrete
Mathematik, Fa\-kul\-t\"at f\"ur Mathematik und Physik, Leibniz
Universit\"{a}t Hannover, Welfengarten 1, 30167 Hannover, Germany}
\email{bessen@math.uni-hannover.de}
\urladdr{http://www2.iazd.uni-hannover.de/\~{ }bessen}

\author{Thorsten Holm}
\address{Institut f\"{u}r Algebra, Zahlentheorie und Diskrete
Mathematik, Fa\-kul\-t\"at f\"ur Mathematik und Physik, Leibniz
Universit\"{a}t Hannover, Welfengarten 1, 30167 Hannover, Germany}
\email{holm@math.uni-hannover.de}
\urladdr{http://www.iazd.uni-hannover.de/\~{ }tholm}

\author{Peter J\o rgensen}
\address{School of Mathematics and Statistics,
Newcastle University, Newcastle upon Tyne NE1 7RU, United Kingdom}
\email{peter.jorgensen@ncl.ac.uk}
\urladdr{http://www.staff.ncl.ac.uk/peter.jorgensen}

%\thanks{Date: \today. A thank you would go here}

\keywords{Arc, disc with accumulation points, Conway--Coxeter
  frieze, Igusa--Todorov cluster category, Ptolemy formula, tiling,
  triangulation} 

\subjclass[2010]{05E15, 13F60}
%05E10: Combinatorial aspects of representation theory
%05E15: Combinatorial aspects of groups and algebras
%05E40: Combinatorial aspects of commutative algebra
%05E45: Combinatorial aspects of simplicial complexes
%13D25: Complexes
%13F60: Cluster algebras
%16E10: Homological dimension
%16E45: Differential graded algebras and applications
%16G10: Representations of Artinian rings 
%16G60: Representation type (finite, tame, wild, etc.) 
%16G70: Auslander-Reiten sequences (almost split sequences) and
%       Auslander-Reiten quivers
%16S90: Torsion theories; radicals on module categories
%18E10: Exact categories, abelian categories
%18E30: Derived categories, triangulated categories
%18E35: Localization of categories
%18E40: Torsion theories, radicals
%18G05: Projectives and injectives
%18G35: Chain complexes
%18G99: Homological algebra: None of the above, but in this section 
%55P62: Rational homotopy theory

\begin{abstract} 

An $\SL2$-tiling is a bi-infinite matrix of positive integers such
that each adjacent $2 \times 2$-submatrix has determinant $1$.  Such
tilings are infinite analogues of Conway--Coxeter friezes, and they
have strong links to cluster algebras, combinatorics, mathematical
physics, and representation theory.

\medskip
\noindent
We show that, by means of so-called Conway--Coxeter counting, every
$\SL2$-tiling arises from a triangulation of the disc with two, three
or four accumulation points.

\medskip
\noindent
This improves earlier results which only discovered $\SL2$-tilings with
infinitely many entries equal to $1$.  Indeed, our methods show that
there are large classes of tilings with only finitely many entries
equal to $1$, including a class of tilings with no $1$'s at all.  In
the latter case, we show that the minimal entry of a tiling is unique.

\end{abstract}

\maketitle

\setcounter{section}{-1}
\section{Introduction}
\label{sec:introduction}

A {\em Conway--Coxeter frieze of Dynkin type $A_n$} is an infinite
strip of positive integers of the form shown in Figure
\ref{fig:frieze}.
\begin{figure}
\[
  \xymatrix @-0.5pc @!0 {
    \cdots & 1 && 1 && *+[black]{1} && *+[black]{1} && *+[black]{1} && *+[black]{1} && *+[black]{1} && *+[black]{1} && *+[black]{1} \\
    2 && 1 && *+[black]{3} && *+[black]{2} && *+[black]{2} && *+[black]{1} && *+[black]{4} && *+[black]{2} && *+[black]{1} &\cdots \\
    \cdots & 1 && *+[black]{2} && *+[black]{5} && *+[black]{3} && *+[black]{1} && *+[black]{3} && *+[black]{7} && *+[black]{1} && 2 \\
    3 && *+[black]{1} && *+[black]{3} && *+[black]{7} && *+[black]{1} && *+[black]{2} && *+[black]{5} && *+[black]{3} && 1 &\cdots \\
    \cdots & *+[black]{2} && *+[black]{1} && *+[black]{4} && *+[black]{2} && *+[black]{1} && *+[black]{3} && *+[black]{2} && 2 && 1 \\
    *+[black]{1} && *+[black]{1} && *+[black]{1} && *+[black]{1} && *+[black]{1} && *+[black]{1} && *+[black]{1} && 1 && 1 &\cdots \\
                        }
\]
\caption{A Conway-Coxeter frieze of Dynkin type $A_4$.}
\label{fig:frieze}
\end{figure}
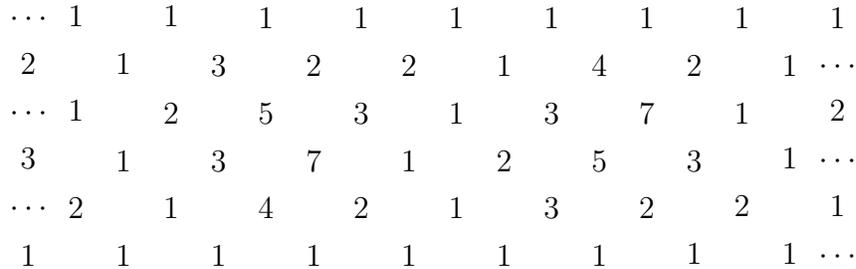
It consists of $n+2$ horizontal rows with an offset between odd
and even rows.  It is bordered by rows of ones and satisfies the
condition $ad - bc = 1$ for each ``diamond''
$\raisebox{13.3pt}{\xymatrix @-2.2pc { & b & \\ a & & d \\ & c & }}$.

Conway--Coxeter friezes were introduced in \cite{CC1} and \cite{CC2}
and inhabit a rich combinatorial theory.  For instance, each frieze
can be obtained by so-called {\em Conway--Coxeter counting} on a
triangulation of the $( n+3 )$-gon, see (28) and (29) in \cite{CC1}
and \cite{CC2} or Definition \ref{def:CC} below.

{\em $\SL2$-tilings} are infinite analogues of Conway--Coxeter
friezes.  They are bi-infinite matrices of positive integers such that
each adjacent $2 \times 2$-submatrix has determinant $1$, see Figure
\ref{fig:intro_tiling_new}.  They were introduced by Assem,
Reutenauer, and Smith in \cite{ARS} and have turned out to be
important objects with a wealth of connections to cluster algebras,
combinatorics, mathematical physics, and representation theory.

Some classes of $\SL2$-tilings were discovered in \cite{ARS} and
\cite{HJ}, but there were examples not belonging to the classes, see
\cite[exa.\ 2.9]{HJ}, and there was no insight into the structure of
the set of all $\SL2$-tilings.

We improve the results from \cite{ARS} and \cite{HJ} significantly by
showing that every $\SL2$-tiling can be obtained by Conway--Coxeter
counting on an {\em infinite triangulation of the disc with two,
  three, or four accumulation points}.  We also show that the
$\SL2$-tilings found in \cite{ARS} and \cite{HJ} are rather special,
because they have infinitely many entries equal to $1$.  Our methods
reveal that there are large classes of $\SL2$-tilings with only
finitely many $1$'s, and even a class of tilings with no $1$'s
at all, see Remark \ref{rmk:zig-zag}.

In the latter case, we show that the minimal entry of a tiling is
unique, see Lemma \ref{lem:minimum}.

{\bf Motivations for studying $\SL2$-tilings. } The introduction of
$\SL2$-tilings in \cite{ARS} was motivated by applications to linear
recurrence relations for certain friezes, and to formulae for cluster
variables in Euclidean type, see \cite[secs.\ 7 and 8]{ARS}.  There is
an application by Assem and Reutenauer in \cite{AR} to formulae for
cluster seeds in types $A$ and $\tilde{A}$.

$\SL2$-tilings were applied to the theory of cluster characters by
Assem, Dupont, Schiffler, and Smith in \cite{ADSS} and J\o rgensen and
Palu in \cite{JP}.  Cluster characters were introduced by Palu in
\cite{P} to formalise cluster categorification.

Di Francesco in \cite{DF2}, \cite{DF1} and Di Francesco and Kedem in
\cite{DFK1}, \cite{DFK2} showed how $\SL2$-tilings are linked to
mathematical physics, where a so-called T-system of type $A_1$ is
simply a pair of $\SL2$-tilings, albeit with Laurent polynomial
values.

$\SL2$-tilings were generalised by Bergeron and Reutenauer in
\cite{BR} to $\operatorname{SL}_k$-tilings.  Other types of
$\SL2$-tilings, relaxing parts of the definition, were obtained by
Baur, Parsons, and Tschabold in \cite{BPT}, Morier-Genoud, Ovsienko,
and Tabachnikov in \cite{MGOT}, Tschabold in \cite{T}, and also in
\cite{HJ} and \cite{JP}.

We continue with a more detailed explanation of this paper.

{\bf Primer on Conway--Coxeter counting. } Figure
\ref{fig:intro_triangulation} shows a triangulation $\fT$ of the disc
with two accumulation points, $D_2$.  The notches indicate marked
points on the boundary of the disc, also called {\em vertices}.  There
are countably many vertices in each of two {\em intervals} given by
the upper and lower half circles.  The vertices converge clockwise
and anticlockwise to two {\em accumulation points} marked with small
circles.  A numbering of the vertices is shown in black; the
superscripts $\I$ and $\III$ are not powers but distinguish between
the vertices on the two intervals.  The triangulation $\fT$ is a set
of arcs between non-neighbouring vertices which divides the disc into
triangular regions.  The figure shows only a finite part of the
infinite set $\fT$.

Conway--Coxeter counting on $\fT$ is the following procedure: Start
with a fixed vertex $\mu$ and label it $0$.  If vertex $\nu$ is a
neighbour of $\mu$, or linked to $\mu$ by an arc in $\fT$, then $\nu$
is labelled $1$.  Now iterate the following: If a triangular region
defined by $\fT$ has precisely two labelled vertices $\pi$ and $\rho$
with labels $i$ and $j$, then its third vertex $\sigma$ is labelled
$i+j$.  The label which eventually appears at $\sigma$ is denoted
$\fT( \mu,\sigma )$.  The green numbers in Figure
\ref{fig:intro_triangulation} show $\fT( \mu,\sigma )$ for $\mu =
-3^{\I}$.

It follows from results by Conway and Coxeter that
\begin{equation}
\label{equ:intro_t}
  t( b,v ) = \fT( b^{\I},v^{\III} )
\end{equation}
with $b,v \in \BZ$ defines an $\SL2$-tiling $t$, said to {\em arise from
$\fT$ by Conway--Coxeter counting}.  Part of $t$ is shown on the left
in Figure \ref{fig:intro_tiling_new}.  Note that we use matrix
notation so $b$ increases when going down, $v$ increases when going
right.

{\bf $\SL2$-tilings without $1$'s and the main result. } Not every
$\SL2$-tiling can be obtained as above.  To see so, observe that if
$\fT$ contains an arc between a vertex $b^{\I}$ on the top half circle
and a vertex $v^{\III}$ on the bottom half circle, then $\fT(
b^{\I},v^{\III} ) = 1$ so $t$ has at least one entry equal to $1$.
But the right half of Figure \ref{fig:intro_tiling_new} shows part of
an $\SL2$-tiling $t'$ with no entry equal to $1$.  One could try
to obtain $t'$ by letting $\fT$ have no arcs between the top half
circle and the bottom half circle, but this will not work: If there
are no such connecting arcs, then Conway--Coxeter counting does not
terminate.  Indeed, the procedure never reaches the bottom half circle
at all, so no labels are defined there.
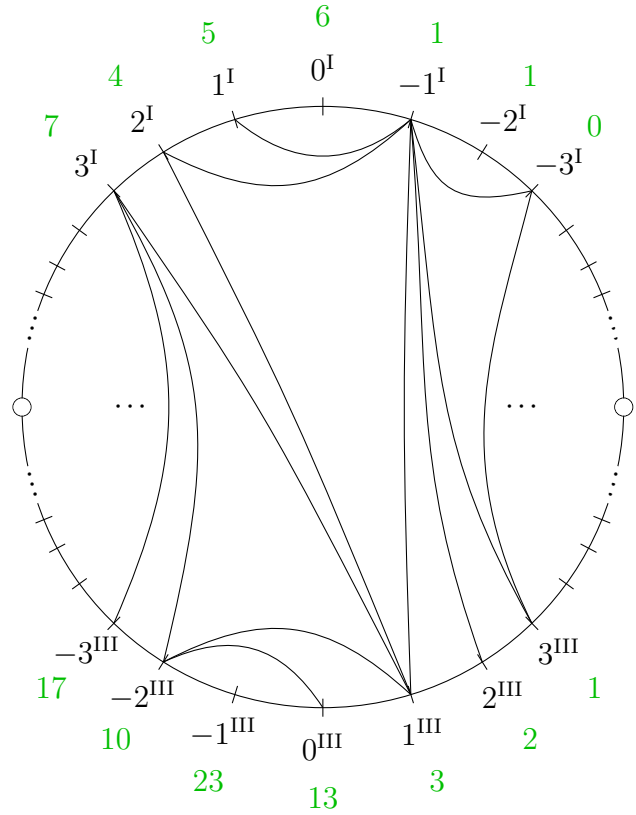
\begin{figure}
  \centering
    \begin{tikzpicture}[scale=4]
      \draw (0,0) circle (1cm);

      \draw (0:1cm) node[fill=white,circle,inner sep=0.080cm] {} circle (0.03cm);
      \draw (180:1cm) node[fill=white,circle,inner sep=0.080cm] {} circle (0.03cm);

      \draw (13:1.00cm) node[fill=white,circle,inner sep=0.004cm] {$\cdot$};
      \draw (15:1.00cm) node[fill=white,circle,inner sep=0.004cm] {$\cdot$};
      \draw (17:1.00cm) node[fill=white,circle,inner sep=0.004cm] {$\cdot$};
      \draw (22:0.97cm) -- (22:1.03cm);
      \draw (28:0.97cm) -- (28:1.03cm);
      \draw (36:0.97cm) -- (36:1.03cm);
      \draw (46:0.97cm) -- (46:1.03cm);
      \draw (46:1.13cm) node{$-3^{\I}$};
      \draw[color=green!75!black] (46:1.30cm) node{$0$};
      \draw (58:0.97cm) -- (58:1.03cm);
      \draw (58:1.13cm) node{$-2^{\I}$};
      \draw[color=green!75!black] (58:1.30cm) node{$1$};
      \draw (73:0.97cm) -- (73:1.03cm);
      \draw (73:1.13cm) node{$-1^{\I}$};
      \draw[color=green!75!black] (73:1.30cm) node{$1$};
      \draw (90:0.97cm) -- (90:1.03cm);
      \draw (90:1.13cm) node{$0^{\I}$};
      \draw[color=green!75!black] (90:1.30cm) node{$6$};
      \draw (107:0.97cm) -- (107:1.03cm);
      \draw (107:1.13cm) node{$1^{\I}$};
      \draw[color=green!75!black] (107:1.30cm) node{$5$};
      \draw (122:0.97cm) -- (122:1.03cm);
      \draw (122:1.13cm) node{$2^{\I}$};
      \draw[color=green!75!black] (122:1.30cm) node{$4$};
      \draw (134:0.97cm) -- (134:1.03cm);
      \draw (134:1.13cm) node{$3^{\I}$};
      \draw[color=green!75!black] (134:1.30cm) node{$7$};
      \draw (144:0.97cm) -- (144:1.03cm);
      \draw (152:0.97cm) -- (152:1.03cm);
      \draw (158:0.97cm) -- (158:1.03cm);
      \draw (163:1.00cm) node[fill=white,circle,inner sep=0.004cm] {$\cdot$};
      \draw (165:1.00cm) node[fill=white,circle,inner sep=0.004cm] {$\cdot$};
      \draw (167:1.00cm) node[fill=white,circle,inner sep=0.004cm] {$\cdot$};

      \draw (193:1.00cm) node[fill=white,circle,inner sep=0.004cm] {$\cdot$};
      \draw (195:1.00cm) node[fill=white,circle,inner sep=0.004cm] {$\cdot$};
      \draw (197:1.00cm) node[fill=white,circle,inner sep=0.004cm] {$\cdot$};
      \draw (202:0.97cm) -- (202:1.03cm);
      \draw (208:0.97cm) -- (208:1.03cm);
      \draw (216:0.97cm) -- (216:1.03cm);
      \draw (226:0.97cm) -- (226:1.03cm);
      \draw (226:1.13cm) node{$-3^{\III}$};
      \draw[color=green!75!black] (226:1.30cm) node{$17$};
      \draw (238:0.97cm) -- (238:1.03cm);
      \draw (238:1.13cm) node{$-2^{\III}$};
      \draw[color=green!75!black] (238:1.30cm) node{$10$};
      \draw (253:0.97cm) -- (253:1.03cm);
      \draw (253:1.13cm) node{$-1^{\III}$};
      \draw[color=green!75!black] (253:1.30cm) node{$23$};
      \draw (270:0.97cm) -- (270:1.03cm);
      \draw (270:1.13cm) node{$0^{\III}$};
      \draw[color=green!75!black] (270:1.30cm) node{$13$};
      \draw (287:0.97cm) -- (287:1.03cm);
      \draw (287:1.13cm) node{$1^{\III}$};
      \draw[color=green!75!black] (287:1.30cm) node{$3$};
      \draw (302:0.97cm) -- (302:1.03cm);
      \draw (302:1.13cm) node{$2^{\III}$};
      \draw[color=green!75!black] (302:1.30cm) node{$2$};
      \draw (314:0.97cm) -- (314:1.03cm);
      \draw (314:1.13cm) node{$3^{\III}$};
      \draw[color=green!75!black] (314:1.30cm) node{$1$};
      \draw (324:0.97cm) -- (324:1.03cm);
      \draw (332:0.97cm) -- (332:1.03cm);
      \draw (338:0.97cm) -- (338:1.03cm);
      \draw (343:1.00cm) node[fill=white,circle,inner sep=0.004cm] {$\cdot$};
      \draw (345:1.00cm) node[fill=white,circle,inner sep=0.004cm] {$\cdot$};
      \draw (347:1.00cm) node[fill=white,circle,inner sep=0.004cm] {$\cdot$};

      \draw (73:1cm) .. controls (20:0.41cm) and (-20:0.39cm) .. (314:1cm);
      \draw (46:1cm) .. controls (20:0.60cm) and (-20:0.44cm) .. (314:1cm);
      \draw (46:1cm) .. controls (56:0.80cm) and (63:0.80cm) .. (73:1cm);
      \draw (73:1cm) .. controls (20:0.38cm) and (-30:0.33cm) .. (302:1cm);
      \draw (73:1cm) .. controls (20:0.28cm) and (-20:0.28cm) .. (287:1cm);
      \draw (73:1cm) .. controls (89:0.70cm) and (106:0.70cm) .. (122:1cm);
      \draw (73:1cm) .. controls (83:0.80cm) and (97:0.80cm) .. (107:1cm);
      \draw (238:1cm) .. controls (254:0.70cm) and (271:0.70cm) .. (287:1cm);
      \draw (238:1cm) .. controls (248:0.80cm) and (260:0.80cm) .. (270:1cm);
      \draw (122:1cm) .. controls (180:0.10cm) .. (287:1cm);
      \draw (134:1cm) .. controls (180:0.20cm) .. (287:1cm);
      \draw (134:1cm) .. controls (160:0.38cm) and (200:0.38cm) .. (238:1cm);
      \draw (134:1cm) .. controls (160:0.48cm) and (200:0.48cm) .. (226:1cm);

      \draw (0:0.62cm) node{$\cdot$};
      \draw (0:0.66cm) node{$\cdot$};
      \draw (0:0.70cm) node{$\cdot$};

      \draw (180:0.60cm) node{$\cdot$};
      \draw (180:0.64cm) node{$\cdot$};
      \draw (180:0.68cm) node{$\cdot$};

    \end{tikzpicture} 
  \caption{A triangulation $\fT$ of the disc with two accumulation
    points, $D_2$.  Black numbers label the vertices, green numbers show an
    example of Conway--Coxeter counting starting at vertex $-3^{\I}$.} 
\label{fig:intro_triangulation}
\end{figure}

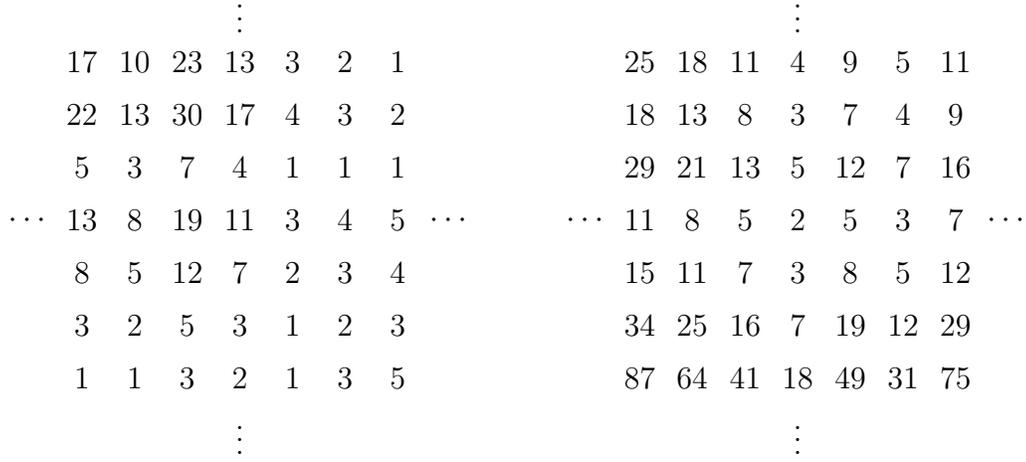
\begin{figure}
\begin{tabular}{ccc}
\begin{tikzpicture}

  \matrix [column sep={0.7cm,between origins},row sep={0.7cm,between origins}]
  {
    &&&& \node {$\vdots$}; &&&& \\
    &\node {17}; & \node {10}; & \node {23}; & \node {13}; & \node {3}; & \node {2}; & \node{1}; &\\
    &\node {22}; & \node {13}; & \node {30}; & \node {17}; & \node {4}; & \node {3}; & \node{2}; &\\
    &\node {5}; & \node {3}; & \node {7}; & \node {4}; & \node {1}; & \node {1}; & \node{1}; &\\
    \node {$\cdots$}; &\node {13}; & \node {8}; & \node {19}; & \node {11}; & \node {3}; & \node {4}; & \node{5}; &\node {$\cdots$};\\
    &\node {8}; & \node {5}; & \node {12}; & \node {7}; & \node {2}; & \node {3}; & \node{4}; &\\
    &\node {3}; & \node {2}; & \node {5}; & \node {3}; & \node {1}; & \node {2}; & \node{3}; &\\
    &\node {1}; & \node {1}; & \node {3}; & \node {2}; & \node {1}; & \node {3}; & \node{5}; &\\
    &&&& \node {$\vdots$}; &&&& \\
  };

\end{tikzpicture}
&&
\begin{tikzpicture}
  \matrix [column sep={0.7cm,between origins},row sep={0.7cm,between origins}]
  {
    &&&& \node {$\vdots$}; &&&& \\
    &\node {25}; & \node {18}; & \node {11}; & \node {4}; & \node {9}; & \node {5}; & \node{11}; &\\
    &\node {18}; & \node {13}; & \node {8}; & \node {3}; & \node {7}; & \node {4}; & \node{9}; &\\
    &\node {29}; & \node {21}; & \node {13}; & \node {5}; & \node {12}; & \node {7}; & \node{16}; &\\
    \node {$\cdots$}; &\node {11}; & \node {8}; & \node {5}; & \node {2}; & \node {5}; & \node {3}; & \node{7}; &\node {$\cdots$};\\
    &\node {15}; & \node {11}; & \node {7}; & \node {3}; & \node {8}; & \node {5}; & \node{12}; &\\
    &\node {34}; & \node {25}; & \node {16}; & \node {7}; & \node {19}; & \node {12}; & \node{29}; &\\
    &\node {87}; & \node {64}; & \node {41}; & \node {18}; & \node {49}; & \node {31}; & \node{75}; &\\
    &&&& \node {$\vdots$}; &&&& \\
  };

\end{tikzpicture}
\end{tabular}
  \caption{Left: The $\SL2$-tiling $t$ obtained by Conway--Coxeter counting
    on $\fT$ from Figure \ref{fig:intro_triangulation}.  Right:
    An $\SL2$-tiling $t'$ with no entry equal to $1$.}
\label{fig:intro_tiling_new}
\end{figure}

\begin{figure}
  \centering
    \begin{tikzpicture}[scale=4]
      \draw (0,0) circle (1cm);

      \draw (45:1cm) node[fill=white,circle,inner sep=0.080cm] {} circle (0.03cm);
      \draw (135:1cm) node[fill=white,circle,inner sep=0.080cm] {} circle (0.03cm);
      \draw (225:1cm) node[fill=white,circle,inner sep=0.080cm] {} circle (0.03cm);
      \draw (315:1cm) node[fill=white,circle,inner sep=0.080cm] {} circle (0.03cm);

      \draw (-34:1.00cm) node[fill=white,circle,inner sep=0.004cm] {$\cdot$};
      \draw (-32:1.00cm) node[fill=white,circle,inner sep=0.004cm] {$\cdot$};
      \draw (-30:1.00cm) node[fill=white,circle,inner sep=0.004cm] {$\cdot$};
      \draw (-24:0.97cm) -- (-24:1.03cm);
%      \draw (-24:1.13cm) node{$-3$};
      \draw (-17:0.97cm) -- (-17:1.03cm);
%      \draw (-17:1.13cm) node{$-2$};
      \draw (-9:0.97cm) -- (-9:1.03cm);
      \draw[color=green!75!black] (-9:1.13cm) node{$\scriptstyle 6$};
      \draw (0:0.97cm) -- (0:1.03cm);
      \draw[color=green!75!black] (0:1.13cm) node{$\scriptstyle 1$};
      \draw (9:0.97cm) -- (9:1.03cm);
      \draw[color=green!75!black] (9:1.13cm) node{$\scriptstyle 1$};
      \draw (17:0.97cm) -- (17:1.03cm);
%      \draw (17:1.13cm) node{$2$};
      \draw (24:0.97cm) -- (24:1.03cm);
%      \draw (24:1.13cm) node{$3$};
      \draw (30:1.00cm) node[fill=white,circle,inner sep=0.004cm] {$\cdot$};
      \draw (32:1.00cm) node[fill=white,circle,inner sep=0.004cm] {$\cdot$};
      \draw (34:1.00cm) node[fill=white,circle,inner sep=0.004cm] {$\cdot$};

      \draw (56:1.00cm) node[fill=white,circle,inner sep=0.004cm] {$\cdot$};
      \draw (58:1.00cm) node[fill=white,circle,inner sep=0.004cm] {$\cdot$};
      \draw (60:1.00cm) node[fill=white,circle,inner sep=0.004cm] {$\cdot$};
      \draw (66:0.97cm) -- (66:1.03cm);
      \draw (66:1.13cm) node{$\scriptstyle -3^{\I}$};
      \draw[color=green!75!black] (66:1.25cm) node{$\scriptstyle 0$};
      \draw (73:0.97cm) -- (73:1.03cm);
      \draw (73:1.13cm) node{$\scriptstyle -2^{\I}$};
      \draw[color=green!75!black] (73:1.25cm) node{$\scriptstyle 1$};
      \draw (81:0.97cm) -- (81:1.03cm);
      \draw (81:1.13cm) node{$\scriptstyle -1^{\I}$};
      \draw[color=green!75!black] (81:1.25cm) node{$\scriptstyle 3$};
      \draw (90:0.97cm) -- (90:1.03cm);
      \draw (90:1.13cm) node{$\scriptstyle 0^{\I}$};
      \draw[color=green!75!black] (90:1.25cm) node{$\scriptstyle 2$};
      \draw (99:0.97cm) -- (99:1.03cm);
      \draw (99:1.13cm) node{$\scriptstyle 1^{\I}$};
      \draw[color=green!75!black] (99:1.25cm) node{$\scriptstyle 5$};
      \draw (107:0.97cm) -- (107:1.03cm);
      \draw (107:1.13cm) node{$\scriptstyle 2^{\I}$};
      \draw[color=green!75!black] (107:1.25cm) node{$\scriptstyle 13$};
      \draw (114:0.97cm) -- (114:1.03cm);
      \draw (114:1.13cm) node{$\scriptstyle 3^{\I}$};
      \draw[color=green!75!black] (114:1.25cm) node{$\scriptstyle 34$};
      \draw (120:1.00cm) node[fill=white,circle,inner sep=0.004cm] {$\cdot$};
      \draw (122:1.00cm) node[fill=white,circle,inner sep=0.004cm] {$\cdot$};
      \draw (124:1.00cm) node[fill=white,circle,inner sep=0.004cm] {$\cdot$};

      \draw (146:1.00cm) node[fill=white,circle,inner sep=0.004cm] {$\cdot$};
      \draw (148:1.00cm) node[fill=white,circle,inner sep=0.004cm] {$\cdot$};
      \draw (150:1.00cm) node[fill=white,circle,inner sep=0.004cm] {$\cdot$};
      \draw (156:0.97cm) -- (156:1.03cm);
%      \draw (156:1.13cm) node{$-3$};
      \draw (163:0.97cm) -- (163:1.03cm);
      \draw[color=green!75!black] (163:1.13cm) node{$\scriptstyle 21$};
      \draw (171:0.97cm) -- (171:1.03cm);
      \draw[color=green!75!black] (171:1.13cm) node{$\scriptstyle 8$};
      \draw (180:0.97cm) -- (180:1.03cm);
      \draw[color=green!75!black] (180:1.13cm) node{$\scriptstyle 3$};
      \draw (189:0.97cm) -- (189:1.03cm);
      \draw[color=green!75!black] (189:1.13cm) node{$\scriptstyle 7$};
      \draw (197:0.97cm) -- (197:1.03cm);
%      \draw (197:1.13cm) node{$2$};
      \draw (204:0.97cm) -- (204:1.03cm);
%      \draw (204:1.13cm) node{$3$};
      \draw (210:1.00cm) node[fill=white,circle,inner sep=0.004cm] {$\cdot$};
      \draw (212:1.00cm) node[fill=white,circle,inner sep=0.004cm] {$\cdot$};
      \draw (214:1.00cm) node[fill=white,circle,inner sep=0.004cm] {$\cdot$};

      \draw (236:1.00cm) node[fill=white,circle,inner sep=0.004cm] {$\cdot$};
      \draw (238:1.00cm) node[fill=white,circle,inner sep=0.004cm] {$\cdot$};
      \draw (240:1.00cm) node[fill=white,circle,inner sep=0.004cm] {$\cdot$};
      \draw (246:0.97cm) -- (246:1.03cm);
      \draw (246:1.13cm) node{$\scriptstyle -3^{\III}$};
      \draw[color=green!75!black] (246:1.25cm) node{$\scriptstyle 25$};
      \draw (253:0.97cm) -- (253:1.03cm);
      \draw (253:1.13cm) node{$\scriptstyle -2^{\III}$};
      \draw[color=green!75!black] (253:1.25cm) node{$\scriptstyle 18$};
      \draw (261:0.97cm) -- (261:1.03cm);
      \draw (261:1.13cm) node{$\scriptstyle -1^{\III}$};
      \draw[color=green!75!black] (261:1.25cm) node{$\scriptstyle 11$};
      \draw (270:0.97cm) -- (270:1.03cm);
      \draw (270:1.13cm) node{$\scriptstyle 0^{\III}$};
      \draw[color=green!75!black] (270:1.25cm) node{$\scriptstyle 4$};
      \draw (279:0.97cm) -- (279:1.03cm);
      \draw (279:1.13cm) node{$\scriptstyle 1^{\III}$};
      \draw[color=green!75!black] (279:1.25cm) node{$\scriptstyle 9$};
      \draw (287:0.97cm) -- (287:1.03cm);
      \draw (287:1.13cm) node{$\scriptstyle 2^{\III}$};
      \draw[color=green!75!black] (287:1.25cm) node{$\scriptstyle 5$};
      \draw (294:0.97cm) -- (294:1.03cm);
      \draw (294:1.13cm) node{$\scriptstyle 3^{\III}$};
      \draw[color=green!75!black] (294:1.25cm) node{$\scriptstyle 11$};
      \draw (300:1.00cm) node[fill=white,circle,inner sep=0.004cm] {$\cdot$};
      \draw (302:1.00cm) node[fill=white,circle,inner sep=0.004cm] {$\cdot$};
      \draw (304:1.00cm) node[fill=white,circle,inner sep=0.004cm] {$\cdot$};

      \draw (66:1cm) .. controls (56:0.60cm) and (19:0.60cm) .. (9:1cm);
      \draw (66:1cm) .. controls (56:0.55cm) and (10:0.60cm) .. (0:1cm);
      \draw (73:1cm) .. controls (63:0.60cm) and (10:0.50cm) .. (0:1cm);
      \draw (73:1cm) .. controls (78:0.80cm) and (85:0.80cm) .. (90:1cm);
      \draw (90:1cm) .. controls (70:0.50cm) and (10:0.30cm) .. (0:1cm);
      \draw (90:1cm) .. controls (110:0.40cm) and (170:0.30cm) .. (180:1cm);
      \draw (99:1cm) .. controls (119:0.40cm) and (170:0.35cm) .. (180:1cm);
      \draw (99:1cm) .. controls (119:0.45cm) and (161:0.35cm) .. (171:1cm);
      \draw (107:1cm) .. controls (127:0.45cm) and (161:0.40cm) .. (171:1cm);
      \draw (107:1cm) .. controls (127:0.50cm) and (143:0.40cm) .. (163:1cm);
      \draw (114:1cm) .. controls (134:0.50cm) and (143:0.45cm) .. (163:1cm);
      \draw (0:1cm) .. controls (90:0.05cm) .. (180:1cm);
      \draw (180:1cm) .. controls (190:0.40cm) and (250:0.30cm) .. (270:1cm);
      \draw (360:1cm) .. controls (340:0.40cm) and (290:0.30cm) .. (270:1cm);
      \draw (189:1cm) .. controls (199:0.40cm) and (250:0.35cm) .. (270:1cm);
      \draw (189:1cm) .. controls (199:0.45cm) and (240:0.35cm) .. (261:1cm);
      \draw (189:1cm) .. controls (199:0.50cm) and (233:0.35cm) .. (253:1cm);
      \draw (189:1cm) .. controls (199:0.55cm) and (226:0.35cm) .. (246:1cm);
      \draw (270:1cm) .. controls (275:0.80cm) and (282:0.80cm) .. (287:1cm);
      \draw (287:1cm) .. controls (297:0.35cm) and (350:0.55cm) .. (360:1cm);
      \draw (287:1cm) .. controls (297:0.45cm) and (341:0.55cm) .. (351:1cm);
      \draw (294:1cm) .. controls (304:0.45cm) and (341:0.60cm) .. (351:1cm);

    \end{tikzpicture} 
  \caption{A triangulation $\fT'$ of the disc with four accumulation
    points, $D_4$.  Black numbers label the vertices, green numbers show an
    example of Conway--Coxeter counting starting at vertex $-3^{\I}$.}
\label{fig:intro_triangulation2}
\end{figure}
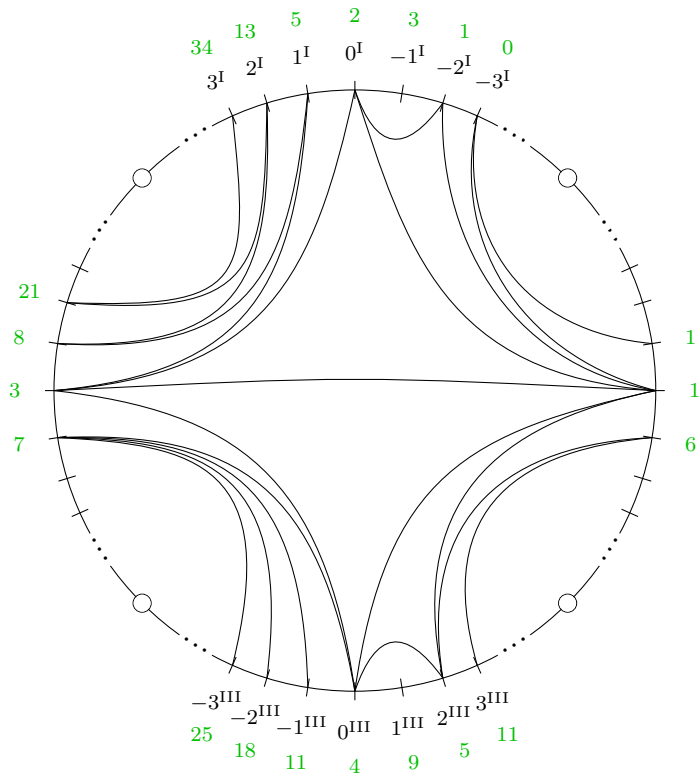
Figure \ref{fig:intro_triangulation2} shows a more sophisticated
triangulation $\fT'$ of the disc with four accumulation points.  There
are now vertices in four intervals, converging clockwise and
anticlockwise to four accumulation points marked with small circles.
The top and bottom intervals are numbered $\I$ and $\III$ as above;
indeed, two of the intervals on a disc will always be numbered $\I$ and
$\III$ regardless of how many accumulation points there are.  A
numbering of the vertices in the top and bottom intervals is shown in
black, and green numbers show the $\fT'( \mu,\sigma )$ for $\mu =
-3^{\I}$.  The $\SL2$-tiling arising from $\fT'$ by Conway--Coxeter
counting is defined as above: $t'( b,v ) = \fT'( b^{\I},v^{\III} )$,
and this is in fact the $t'$ in the right half of Figure
\ref{fig:intro_tiling_new}.

The extra accumulation points mean that there is room in $\fT'$ for a
horizontal arc which blocks $\fT'$ from having arcs between the top
and bottom intervals.  This means that $\fT'( b^{\I},v^{\III} )$ is
never equal to $1$, so $t'$ has no entry equal to $1$.  Note that in
this example, Conway--Coxeter counting does indeed terminate with
labels on the bottom interval because it can progress through the side
intervals.

Our main result is that four accumulation points are sufficient for 
every $\SL2$-tiling to arise:

{\bf Theorem A. }
{\em
Let $t$ be an $\SL2$-tiling.  There exists a good triangulation
$\fT$ of the disc with two, three, or four accumulation points,
such that $t$ arises from $\fT$ by Conway--Coxeter counting between
two of the intervals which go from one accumulation point to the next. 
\hfill $\Box\!\!\!$
}

The notion of a {\em good triangulation} is made precise in Definition
\ref{def:triangulation}.  The point is that Conway--Coxeter counting
always terminates for these.  Theorem A is a portmanteau of Theorems
\ref{thm:Case1}, \ref{thm:Case2}, \ref{thm:Case3}, \ref{thm:Case4},
\ref{thm:Case5}, and \ref{thm:Case6}, each of which starts with an
$\SL2$-tiling $t$ of a certain type and constructs a good
triangulation $\fT$.

{\bf On the proof of Theorem A. }
The construction of $\fT$ is split across six theorems because the
details depend strongly on $t$; specifically, on the pattern of
entries equal to $1$.  However, the philosophy is the same in all
cases as we now explain.  

Let $t$ be an $\SL2$-tiling.  On the one hand, $t$ gives rise to two
{\em infinite friezes} in the sense of Tschabold, see \cite[def.\
1.1]{T} or Definition \ref{def:tpq} and Figure
\ref{fig:intro_tiling3}.  They are defined by
\[
  p( a,d ) =
  \begin{vmatrix}
    t( a,w ) & t( a,w+1 ) \\
    t( d,w ) & t( d,w+1 )
  \end{vmatrix}
  \;\;,\;\;
  q( u,x ) =
  \begin{vmatrix}
    t( c,u ) & t( c,x ) \\
    t( c+1,u ) & t( c+1,x )
  \end{vmatrix}
\]
for integers $a \leqslant d$, $u \leqslant x$.  Note that the integers
$w$ and $c$ can be chosen freely; $p( a,d )$ and $q( u,x )$ do not
depend on them.  To say that $p$ is an infinite frieze means that $p(
a,a ) = 0$, $p( a,a+1 ) = 1$, $p( a,d ) \geqslant 1$ for $a < d$, and,
when writing $p$ as a matrix, each $2 \times 2$-submatrix which makes
sense has determinant $1$.  Note that to improve the compatibility
with $\SL2$-tilings, our convention for indexing an infinite frieze
differs from \cite[def.\ 1.1]{T}.
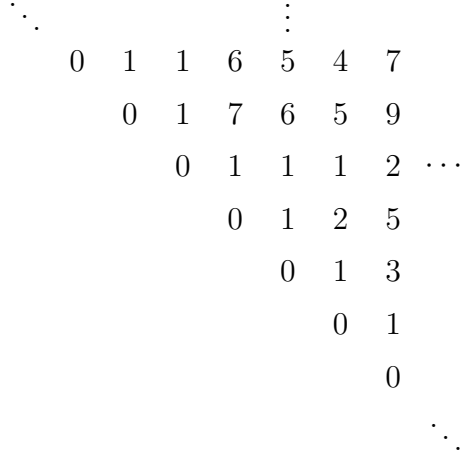
\begin{figure}
\begin{tikzpicture}
  \matrix [column sep={0.7cm,between origins},row sep={0.7cm,between origins}]
  {
    \node {$\dddots$}; &&&&& \node {$\vdots$}; &&& \\
    &\node {0}; & \node {1}; & \node {1}; & \node {6}; & \node {5}; & \node {4}; & \node{7}; &\\
    && \node {0}; & \node {1}; & \node {7}; & \node {6}; & \node {5}; & \node{9}; &\\
    &&& \node {0}; & \node {1}; & \node {1}; & \node {1}; & \node{2}; &\node {$\cdots$};\\
    &&&& \node {0}; & \node {1}; & \node {2}; & \node{5}; & \\
    &&&&& \node {0}; & \node {1}; & \node{3}; &\\
    &&&&&& \node {0}; & \node{1}; &\\
    &&&&&&& \node{0}; &\\
    &&&&&&&& \node {$\dddots$}; \\
  };

\end{tikzpicture}
  \caption{An infinite frieze.}
\label{fig:intro_tiling3}
\end{figure}

On the other hand, a putative good triangulation $\fT$ gives rise not
merely to the $\SL2$-tiling of Equation \eqref{equ:intro_t}, but also
to two infinite friezes defined by
\[
  ( a,d ) \mapsto \fT( a^{\I},d^{\I} )
  \;\;,\;\;
  ( u,x ) \mapsto \fT( u^{\III},x^{\III} )
\]
for integers $a \leqslant d$, $u \leqslant x$; this again follows from
results by Conway and Coxeter.

To prove Theorem A we must show that when $t$ is an $\SL2$-tiling,
there is a good triangulation $\fT$ satisfying Equation
\eqref{equ:intro_t}.  However, we will tackle the seemingly harder
problem of also asking for
\begin{align}
\label{equ:intro_p}
  p( a,d ) & = \fT( a^{\I},d^{\I} ), \\
\label{equ:intro_q}
  q( u,x ) & = \fT( u^{\III},x^{\III} )
\end{align}
for $a \leqslant d$, $u \leqslant x$.  This actually turns out to be
easier because the entries in the triple $( t,p,q )$ and the numbers
$\fT( \mu,\nu )$ satisfy two strong sets of equations called {\em
  Ptolemy relations} which we do not list here, but see Lemmas
\ref{lem:CC}(v) and \ref{lem:Ptolemy}.  They mean that, when $\fT$ has
been constructed, in order to prove Equations \eqref{equ:intro_t}
through \eqref{equ:intro_q} in general, it is sufficient to do so in a
relatively small set of special cases.

For example, suppose that $t$ has infinitely many entries equal to $1$
in both the first and the third quadrant; this is the case considered
in Theorem \ref{thm:Case1}.  For such a $t$, we will show that the set
of arcs
\begin{align}
\nonumber
  \Theta( t ) =
  \big\{ \{ b^{\I},v^{\III} \} \,\big|\, t( b,v ) = 1 \big\}
  & \cup
  \big\{ \{ a^{\I},d^{\I} \} \,\big|\, a+2 \leqslant d, \, p( a,d ) = 1 \big\} \\
\label{equ:pre_Theta}
  & \cup
  \big\{ \{ u^{\III},x^{\III} \} \,\big|\, u+2 \leqslant x, \, q( u,x ) = 1 \big\}
\end{align}
is a good triangulation of $D_2$ (observe that we think of an arc as
a purely combinatorial object specified by giving the end vertices).
Moreover, if we set $\fT = \Theta ( t )$ then Equations
\eqref{equ:intro_t} through \eqref{equ:intro_q} hold in some special
cases: If $t( b,v ) = 1$ then $\{ b^{\I},v^{\III} \} \in \fT$ whence
$\fT( b^{\I},v^{\III} ) = 1$, so Equation \eqref{equ:intro_t} holds.
Likewise, if $p( a,d ) = 1$ then Equation \eqref{equ:intro_p} holds,
and if $q( u,x ) = 1$ then Equation \eqref{equ:intro_q} holds.  Using
only this, the Ptolemy relations turn out to imply the three equations
in general.  In particular, Equation \eqref{equ:intro_t} holds in
general, so $t$ arises from $\fT$ by Conway--Coxeter counting.

Before ending this discussion, let us highlight another useful
phenomenon: The special cases $d = a+2$ of Equation
\eqref{equ:intro_p} and (symmetrically) $x = u+2$ of Equation
\eqref{equ:intro_q} imply the two equations in general.  Indeed, this
is just the easy fact that the second diagonal, or {\em quiddity
  sequence}, of an infinite frieze determines the whole frieze, see
\cite[rmk.\ 1.3]{T}.  When $t$ is given, it is hence important to be
able to construct a good triangulation $\fT$ which satisfies Equations
\eqref{equ:intro_p} and \eqref{equ:intro_q} in these special cases.
We will use the following approach: The vertices $a^{\I}$, $( a+1
)^{\I}$, $( a+2 )^{\I}$ are consecutive on the disc.  It is known that
hence, if $\fT$ can be constructed, then
\[
  \fT\big( a^{\I},( a+2 )^{\I} \big)
  = 1 + \mbox{\big(the number of arcs in $\fT$ which end at $( a+1 )^{\I}$\big)}.
\]
To get Equation \eqref{equ:intro_p} for $d = a+2$, we must construct
$\fT$ such that
\[
  p( a,a+2 )
  = 1 + \mbox{\big(the number of arcs in $\fT$ which end at $( a+1 )^{\I}$\big)}.
\]
In Theorems \ref{thm:Case2}, \ref{thm:Case3}, \ref{thm:Case4},
\ref{thm:Case5}, and \ref{thm:Case6}, this is accomplished by starting
with the set of arcs $\Theta( t )$ from Equation \eqref{equ:pre_Theta}
and adding arcs so that, eventually, there are $p( a,a+2 ) - 1$ arcs
ending at $( a+1 )^{\I}$ for each $a$.  See for instance Figure
\ref{fig:ft_in_Case2} where the arcs in $\Theta( t )$ are black and
the additional arcs are red.  The figure also illustrates that the
additional arcs need somewhere to end.  This is the reason we need
more intervals than $\I$ and $\III$.  The number of arcs to be added
at $( a+1 )^{\I}$ is given by the {\em defect} $\defect_p( a+1 )$
introduced in Definition \ref{def:defects}; this is the rationale for
defining and manipulating defects in Section \ref{sec:defects}.
See also Figure \ref{fig:ft_in_Case2} and its caption.

{\bf Link to the cluster categories of Igusa and Todorov. } Let $n$
be $2$, $3$, or $4$, and let $D_n$ be the disc with $n$ accumulation
points.  The set of vertices of $D_n$ is an example of a cyclic poset
in the sense of Igusa and Todorov, see \cite[def.\ 1.1.12]{IT}.  There
is an associated cluster category $\cC$ with infinite clusters, see
\cite[thm.\ 2.4.1]{IT}.  It categorifies $D_n$ in the sense that there
is a bijection between arcs in $D_n$ and indecomposable objects in
$\cC$, such that crossing of arcs corresponds to existence of
non-split extensions.  Moreover, if $\fT$ is a good triangulation of
$D_n$, then the arcs in $\fT$ correspond to a set of indecomposable
objects whose finite direct sums form a cluster tilting subcategory
$\cT$ of $\cC$.  See \cite{GHJ} for more details.

There is an arithmetic Caldero--Chapoton map $\rho_{\cT}$ associated
to $\cC$ and $\cT$.  As indicated by the name, the map is due to
Caldero and Chapoton, but the specific version we have in mind is the
one from \cite[def.\ 3.1]{HJ}.  It is a map
\[
  \varphi_{\cT} : \operatorname{obj}\,\cC \rightarrow \BZ
\]
which can be computed by Conway--Coxeter counting; this follows from
\cite[prop.\ 1.10]{JP} by the method used to prove \cite[thm.\
5.4]{HJ2}.  Hence if $a_{ \mu\nu }$ in $\cC$ is the indecomposable
object corresponding to the arc $\{ \mu,\nu \}$, then
\[
  \varphi_{ \cT }( a_{ \mu\nu } ) = \fT( \mu,\nu ).
\]
This means that we can view $\cC$ and $\varphi_{ \cT }$ as
categorifying the $\SL2$-tiling arising from $\fT$ by Conway--Coxeter
counting. 

This is of interest because there is a more general Caldero--Chapoton
map 
\[
  \rho_{\cT} : \operatorname{obj}\,\cC \rightarrow
  \BQ( x_t \,|\, t \mbox{ indecomposable in } \cT )
\]
with Laurent polynomial values whose image generates a cluster algebra
with infinite clusters, see \cite[thm.\ 2.3 and cor.\ 2.5]{JP}.  The
$\SL2$-tiling arising from $\fT$ by Conway--Coxeter counting can be
recovered by specialising the initial cluster variables $x_t$ to $1$.
Such cluster algebras have so far only been studied carefully for the
disc with one accumulation point.  They have several interesting
properties different from cluster algebras with finite clusters, and
seem likely to be of interest also for larger numbers of accumulation
points.  See \cite{GG} by Grabowski and Gratz.

{\bf Structure of the paper. }
Theorem A will not be proved in one go, but sums up Theorems
\ref{thm:Case1}, \ref{thm:Case2}, \ref{thm:Case3}, \ref{thm:Case4},
\ref{thm:Case5}, and \ref{thm:Case6}.  Each of these starts with an
$\SL2$-tiling $t$ with a certain pattern of entries equal to $1$ and
constructs a triangulation $\fT$ of the disc with two, three, or four
accumulation points.

Reading the theorems in order will make it clear that they cover every
possible $\SL2$-tiling $t$: They progress through $\SL2$-tilings $t$
with fewer and fewer entries equal to $1$, ending with no $1$'s at all
in Theorem \ref{thm:Case6}.

Conversely, $\SL2$-tilings of the types described in the theorems do
exist: In each case, they can be obtained as the $\SL2$-tilings
arising by Conway--Coxeter counting from triangulations of the type
constructed in the theorem, see Remark \ref{rmk:zig-zag}.

Section \ref{sec:triangulations} gives formal definitions relating to
triangulations of the disc with accumulation points.  Section
\ref{sec:CC} recalls some properties of Conway--Coxeter counting.
Section \ref{sec:SL2} shows some results on $\SL2$-tilings and their
associated infinite friezes.  Section \ref{sec:basic} starts with an
$\SL2$-tiling $t$ and constructs a partial triangulation $\Theta( t
)$.  In each subsequent case, the (full) triangulation $\fT$ is
obtained either as $\Theta( t )$ itself (Theorem \ref{thm:Case1}), or
is constructed by adding arcs to $\Theta( t )$ (Theorems
\ref{thm:Case2}, \ref{thm:Case3}, \ref{thm:Case4}, \ref{thm:Case5},
and \ref{thm:Case6}).  Section \ref{sec:defects} introduces what we
call defects and shows some properties.  The defects provide
information about how many arcs we must add to $\Theta( t )$ to get
$\fT$.

Sections \ref{sec:Case1} through \ref{sec:Case5} prove Theorems
\ref{thm:Case1}, \ref{thm:Case2}, \ref{thm:Case3}, \ref{thm:Case4},
\ref{thm:Case5}.  Section \ref{sec:CC2} shows a technical result on
Conway--Coxeter friezes, Section \ref{sec:minimum} shows that an
$\SL2$-tiling with no entry equal to $1$ has a unique minimum, and
Section \ref{sec:Case6} proves Theorem \ref{thm:Case6}, thereby
completing the proof of Theorem A.

\section{Triangulations of the disc with accumulation points and other
basic definitions}
\label{sec:triangulations}

\begin{Setup}
Throughout, $C$ is a circle with anticlockwise orientation, $D$ is a disc with boundary
$C$, and $n$ is $2$, $3$, or $4$.
\end{Setup}

\begin{Notation}
Let $\mu_1$, $\ldots$, $\mu_m$ be points on $C$.

The string of inequalities $\mu_1 < \cdots < \mu_m$ will mean that
each $\mu_i$ is different from its predecessor, and that if we start
from $\mu_1$ and move anticlockwise on $C$ by one full turn, then we
encounter the points in precisely the order $\mu_1$, $\ldots$,
$\mu_m$.

It is straightforward to modify this to permit the inequality sign
$\leqslant$ as well as infinite strings of inequalities.
\end{Notation}

\begin{Definition}
[The disc with four accumulation points]
\label{def:C4}
Let $D_4$, the {\em disc with four accumulation points}, be 
the object sketched in Figure \ref{fig:C4}.
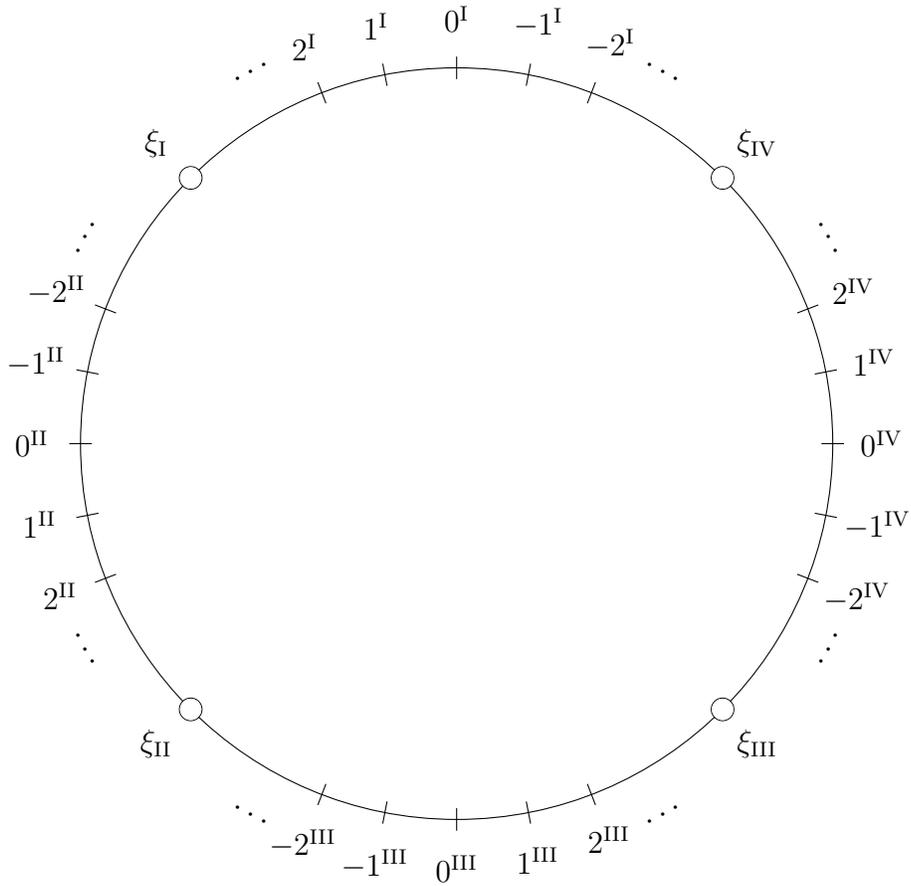
\begin{figure}
  \centering
    \begin{tikzpicture}[scale=5]
      \draw (0,0) circle (1cm);

      \draw (45:1cm) node[fill=white,circle,inner sep=0.101cm] {} circle (0.03cm);
      \draw (135:1cm) node[fill=white,circle,inner sep=0.101cm] {} circle (0.03cm);
      \draw (225:1cm) node[fill=white,circle,inner sep=0.101cm] {} circle (0.03cm);
      \draw (315:1cm) node[fill=white,circle,inner sep=0.101cm] {} circle (0.03cm);

      \draw (45:1.13cm) node{$\xi_{\IV}$};
      \draw (135:1.13cm) node{$\xi_{\I}$};
      \draw (225:1.13cm) node{$\xi_{\II}$};
      \draw (315:1.13cm) node{$\xi_{\III}$};

      \draw (-31:1.13cm) node{$\cdot$};
      \draw (-29:1.13cm) node{$\cdot$};
      \draw (-27:1.13cm) node{$\cdot$};
      \draw (-21:0.97cm) -- (-21:1.03cm);
      \draw (-21:1.14cm) node{$-2^{\IV}$};
      \draw (-11:0.97cm) -- (-11:1.03cm);
      \draw (-11:1.14cm) node{$-1^{\IV}$};
      \draw (0:0.97cm) -- (0:1.03cm);
      \draw (0:1.13cm) node{$0^{\IV}$};
      \draw (11:0.97cm) -- (11:1.03cm);
      \draw (11:1.13cm) node{$1^{\IV}$};
      \draw (21:0.97cm) -- (21:1.03cm);
      \draw (21:1.13cm) node{$2^{\IV}$};
      \draw (27:1.13cm) node{$\cdot$};
      \draw (29:1.13cm) node{$\cdot$};
      \draw (31:1.13cm) node{$\cdot$};

      \draw (59:1.13cm) node{$\cdot$};
      \draw (61:1.13cm) node{$\cdot$};
      \draw (63:1.13cm) node{$\cdot$};
      \draw (69:0.97cm) -- (69:1.03cm);
      \draw (69:1.14cm) node{$-2^{\I}$};
      \draw (79:0.97cm) -- (79:1.03cm);
      \draw (79:1.14cm) node{$-1^{\I}$};
      \draw (90:0.97cm) -- (90:1.03cm);
      \draw (90:1.13cm) node{$0^{\I}$};
      \draw (101:0.97cm) -- (101:1.03cm);
      \draw (101:1.13cm) node{$1^{\I}$};
      \draw (111:0.97cm) -- (111:1.03cm);
      \draw (111:1.13cm) node{$2^{\I}$};
      \draw (117:1.13cm) node{$\cdot$};
      \draw (119:1.13cm) node{$\cdot$};
      \draw (121:1.13cm) node{$\cdot$};

      \draw (149:1.13cm) node{$\cdot$};
      \draw (151:1.13cm) node{$\cdot$};
      \draw (153:1.13cm) node{$\cdot$};
      \draw (159:0.97cm) -- (159:1.03cm);
      \draw (159:1.14cm) node{$-2^{\II}$};
      \draw (169:0.97cm) -- (169:1.03cm);
      \draw (169:1.14cm) node{$-1^{\II}$};
      \draw (180:0.97cm) -- (180:1.03cm);
      \draw (180:1.13cm) node{$0^{\II}$};
      \draw (191:0.97cm) -- (191:1.03cm);
      \draw (191:1.13cm) node{$1^{\II}$};
      \draw (201:0.97cm) -- (201:1.03cm);
      \draw (201:1.13cm) node{$2^{\II}$};
      \draw (207:1.13cm) node{$\cdot$};
      \draw (209:1.13cm) node{$\cdot$};
      \draw (211:1.13cm) node{$\cdot$};

      \draw (239:1.13cm) node{$\cdot$};
      \draw (241:1.13cm) node{$\cdot$};
      \draw (243:1.13cm) node{$\cdot$};
      \draw (249:0.97cm) -- (249:1.03cm);
      \draw (249:1.14cm) node{$-2^{\III}$};
      \draw (259:0.97cm) -- (259:1.03cm);
      \draw (259:1.14cm) node{$-1^{\III}$};
      \draw (270:0.97cm) -- (270:1.03cm);
      \draw (270:1.13cm) node{$0^{\III}$};
      \draw (281:0.97cm) -- (281:1.03cm);
      \draw (281:1.13cm) node{$1^{\III}$};
      \draw (291:0.97cm) -- (291:1.03cm);
      \draw (291:1.13cm) node{$2^{\III}$};
      \draw (297:1.13cm) node{$\cdot$};
      \draw (299:1.13cm) node{$\cdot$};
      \draw (301:1.13cm) node{$\cdot$};

    \end{tikzpicture} 
  \caption{This is $D_4$, the disc with four accumulation points,
    $\xi_{\I}$ through $\xi_{\IV}$.}
\label{fig:C4}
\end{figure}

More formally, $D_4$ is the disc $D$ along with four points $\xi_{\I}
< \xi_{\II} < \xi_{\III} < \xi_{\IV}$ on the boundary $C$ called {\em
  accumulation points of $D_4$}, and infinitely many points on $C$
called {\em vertices of $D_4$}, defined as follows:

For each $\J \in \{ \I,\II,\III,\IV \}$, let $\ldots$,
$-1^{\J}$, $0^{\J}$, $1^{\J}$, $\ldots$ be countably many points on
$C$ which satisfy:
\begin{itemize}
\setlength\itemsep{4pt}

  \item  $\xi_{ \J-1 } < \cdots < -1^{\J} < 0^{\J} < 1^{\J} <
    \cdots < \xi_{\J}$, 

  \item  the sequence $0^{\J}$, $1^{\J}$, $2^{\J}$, $\ldots$ converges to
    $\xi_{\J}$,

  \item  the sequence $0^{\J}$, $-1^{\J}$, $-2^{\J}$, $\ldots$ converges to
    $\xi_{\J-1}$. 
\end{itemize}
Here $\J-1$ stands for the Roman numeral one below $\J$, or $\IV$ if
$\J = \I$.  The vertices of $D_4$ are the points $\ldots$, $-1^{\J}$,
$0^{\J}$, $1^{\J}$, $\ldots$ for $\J \in \{ \I,\II,\III,\IV \}$.

The set
\[
  \{ \omega \in C \mid \xi_{\J-1} < \omega < \xi_{\J} \}
\]
will be called {\em interval $\J$ of the boundary of $D_4$}.  There is
an obvious notion of when two intervals are {\em neighbouring}.

Our convention for numbering the intervals of the boundary of $D_4$ is
shown in simplified form in Figure \ref{fig:C4_simple}.
\end{Definition}

\begin{Definition}
[The disc with two or three accumulation points]
\label{def:C2C3}
We can mimic Definition \ref{def:C4} in order to define $D_2$, the
disc with two accumulation points, and $D_3$, the disc with three
accumulation points.  For reasons which will be explained later, in
case of $D_2$ we will denote the intervals by Roman numerals $\I$ and
$\III$, and in case of $D_3$ by Roman numerals $\{ \I,\II,\III \}$ or
$\{ \I,\III,\IV \}$.

That is, intervals $\I$ and $\III$ are always present, but $\II$
and/or $\IV$ may be dropped; see Figures \ref{fig:C2} and
\ref{fig:C3}.
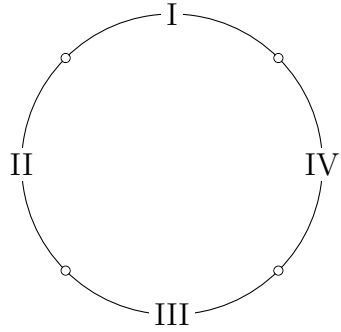
\begin{figure}
  \centering
    \begin{tikzpicture}[scale=2]
      \draw (0,0) circle (1cm);

      \draw (45:1cm) node[fill=white,circle,inner sep=0.037cm] {} circle (0.03cm);
      \draw (135:1cm) node[fill=white,circle,inner sep=0.037cm] {} circle (0.03cm);
      \draw (225:1cm) node[fill=white,circle,inner sep=0.037cm] {} circle (0.03cm);
      \draw (315:1cm) node[fill=white,circle,inner sep=0.037cm] {} circle (0.03cm);
%inner sep at scale=2 is 0.037cm

      \draw (90:1cm) node[fill=white,rectangle,inner sep=0.07cm] {$\I$};
      \draw (180:1cm) node[fill=white,rectangle,inner sep=0.07cm] {$\II$};
      \draw (270:1cm) node[fill=white,rectangle,inner sep=0.07cm] {$\III$};
      \draw (360:1cm) node[fill=white,rectangle,inner sep=0.07cm] {$\IV$};

    \end{tikzpicture} 
    \caption{A simpler view of the disc with four accumulation
      points, $D_4$, and our convention for numbering the intervals of
      the boundary.}
\label{fig:C4_simple}
\end{figure}
\begin{figure}
  \centering
    \begin{tikzpicture}[scale=2]
      \draw (0,0) circle (1cm);

      \draw (0:1cm) node[fill=white,circle,inner sep=0.037cm] {} circle (0.03cm);
      \draw (180:1cm) node[fill=white,circle,inner sep=0.037cm] {} circle (0.03cm);
%inner sep at scale=2 is 0.037cm

      \draw (90:1cm) node[fill=white,rectangle,inner sep=0.07cm] {$\I$};
      \draw (270:1cm) node[fill=white,rectangle,inner sep=0.07cm] {$\III$};

    \end{tikzpicture} 
  \caption{The disc with two accumulation points, $D_2$, and our
    convention for numbering the intervals of the boundary.}
\label{fig:C2}
\end{figure}
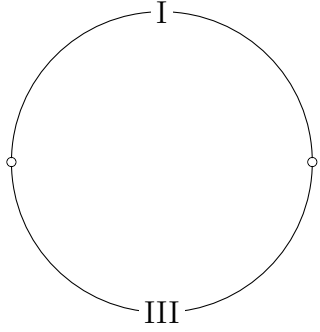
\begin{figure}
  \centering
    \begin{tabular}{cccc}
    \begin{tikzpicture}[scale=2]
      \draw (0,0) circle (1cm);

      \draw (0:1cm) node[fill=white,circle,inner sep=0.037cm] {} circle (0.03cm);
      \draw (120:1cm) node[fill=white,circle,inner sep=0.037cm] {} circle (0.03cm);
      \draw (240:1cm) node[fill=white,circle,inner sep=0.037cm] {} circle (0.03cm);

      \draw (60:1cm) node[fill=white,rectangle,inner sep=0.07cm] {$\I$};
      \draw (180:1cm) node[fill=white,rectangle,inner sep=0.07cm] {$\II$};
      \draw (300:1cm) node[fill=white,rectangle,inner sep=0.07cm] {$\III$};

    \end{tikzpicture} 
    &&&
    \begin{tikzpicture}[scale=2]
      \draw (0,0) circle (1cm);

      \draw (60:1cm) node[fill=white,circle,inner sep=0.037cm] {} circle (0.03cm);
      \draw (180:1cm) node[fill=white,circle,inner sep=0.037cm] {} circle (0.03cm);
      \draw (300:1cm) node[fill=white,circle,inner sep=0.037cm] {} circle (0.03cm);

      \draw (0:1cm) node[fill=white,rectangle,inner sep=0.07cm] {$\IV$};
      \draw (120:1cm) node[fill=white,rectangle,inner sep=0.07cm] {$\I$};
      \draw (240:1cm) node[fill=white,rectangle,inner sep=0.07cm] {$\III$};

    \end{tikzpicture} 
    \end{tabular}
  \caption{The disc with three accumulation points, $D_3$, and our
    two possible conventions for numbering the intervals of the boundary.}
\label{fig:C3}
\end{figure}
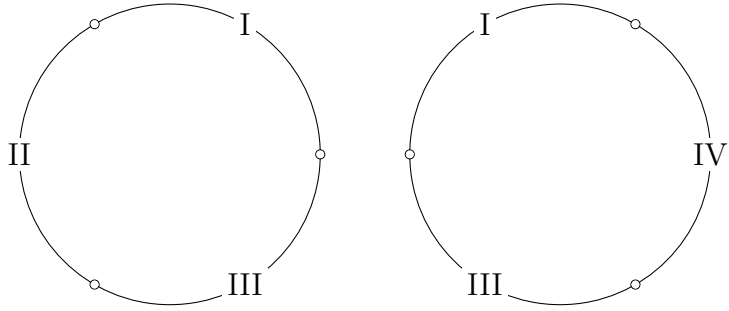
\end{Definition}

\begin{Notation}
Recall that $n$ is $2$, $3$ or $4$, so we may consider $D_n$, the disc
with $n$ accumulation points.

Generic integers will often be denoted by $i$, $j$, $k$, $\ell$, $m$
and generic vertices of $D_n$ often by $\iota$, $\kappa$, $\mu$,
$\nu$, $\pi$, $\rho$, $\sigma$.

If $\J$ is an interval of the boundary of $D_n$ and $m$ is an integer,
then the vertex $m^{\J}$ is in $\J$.  Depending on whether $\J$ is
$\I$, $\II$, $\III$, or $\IV$, we will typically replace $m$ by one of
the letters in Figure \ref{fig:letters}.
\begin{figure}
  \centering
    \begin{tikzpicture}[scale=4]
      \draw (0,0) circle (1cm);

      \draw (45:1cm) node[fill=white,circle,inner sep=0.080cm] {} circle (0.03cm);
      \draw (135:1cm) node[fill=white,circle,inner sep=0.080cm] {} circle (0.03cm);
      \draw (225:1cm) node[fill=white,circle,inner sep=0.080cm] {} circle (0.03cm);
      \draw (315:1cm) node[fill=white,circle,inner sep=0.080cm] {} circle (0.03cm);

      \draw (-21:0.97cm) -- (-21:1.03cm);
      \draw (-21:1.13cm) node{$\tau^{\IV}$};
      \draw (-7:0.97cm) -- (-7:1.03cm);
      \draw (-7:1.13cm) node{$\varphi^{\IV}$};
      \draw (7:0.97cm) -- (7:1.03cm);
      \draw (7:1.13cm) node{$\chi^{\IV}$};
      \draw (21:0.97cm) -- (21:1.03cm);
      \draw (21:1.13cm) node{$\psi^{\IV}$};

      \draw (69:0.97cm) -- (69:1.03cm);
      \draw (69:1.13cm) node{$a^{\I}$};
      \draw (83:0.97cm) -- (83:1.03cm);
      \draw (83:1.13cm) node{$b^{\I}$};
      \draw (97:0.97cm) -- (97:1.03cm);
      \draw (97:1.13cm) node{$c^{\I}$};
      \draw (111:0.97cm) -- (111:1.03cm);
      \draw (111:1.13cm) node{$d^{\I}$};

      \draw (159:0.97cm) -- (159:1.03cm);
      \draw (159:1.13cm) node{$\alpha^{\II}$};
      \draw (173:0.97cm) -- (173:1.03cm);
      \draw (173:1.13cm) node{$\beta^{\II}$};
      \draw (187:0.97cm) -- (187:1.03cm);
      \draw (187:1.13cm) node{$\gamma^{\II}$};
      \draw (201:0.97cm) -- (201:1.03cm);
      \draw (201:1.13cm) node{$\delta^{\II}$};

      \draw (249:0.97cm) -- (249:1.03cm);
      \draw (249:1.13cm) node{$u^{\III}$};
      \draw (263:0.97cm) -- (263:1.03cm);
      \draw (263:1.13cm) node{$v^{\III}$};
      \draw (277:0.97cm) -- (277:1.03cm);
      \draw (277:1.13cm) node{$w^{\III}$};
      \draw (291:0.97cm) -- (291:1.03cm);
      \draw (291:1.13cm) node{$x^{\III}$};

    \end{tikzpicture} 
  \caption{Depending on the interval, we typically use these
    labels for the vertices.}
\label{fig:letters}
\end{figure}
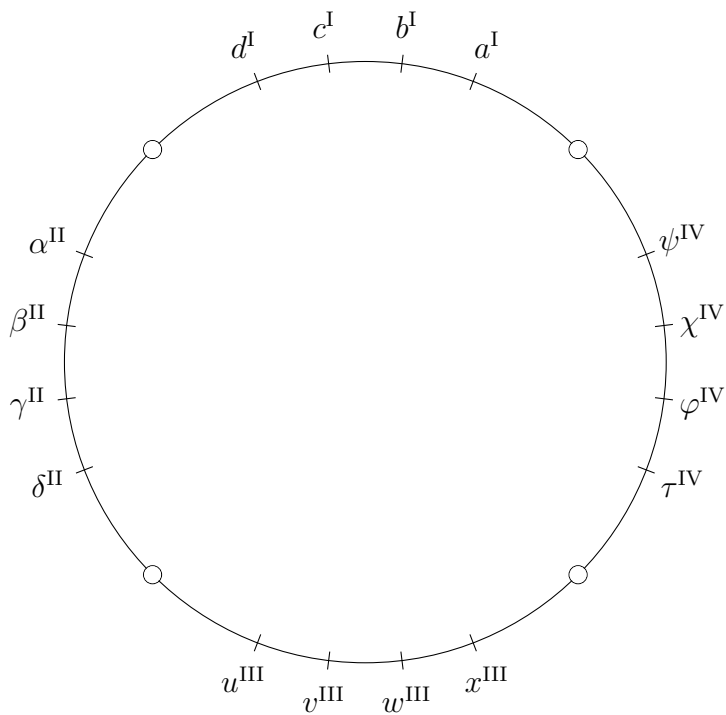

In subsequent figures, the superscripts of vertices will be omitted since
it is clear from a figure when two vertices belong to different
intervals.  Superscripts will, however, be used in the main text.
\end{Notation}

\begin{Definition}
[Edges, arcs, and crossing]
Let $\mu$ be a vertex of $D_n$.  There are evident notions of the
{\em previous} vertex $\mu^-$ and the {\em next} vertex $\mu^+$.
These are called the {\em neighbouring} vertices of $\mu$.

When $\mu$ and $\nu$ are different vertices of $D_n$, we can consider
the set $\{ \mu,\nu \}$.  If $\mu$ and $\nu$ are neighbouring vertices
then $\{ \mu,\nu \}$ is called the {\em edge between $\mu$ and $\nu$
  in $D_n$}, and if $\mu$ and $\nu$ are non-neighbouring vertices then
$\{ \mu,\nu \}$ is called the {\em arc between $\mu$ and $\nu$ in
  $D_n$}.  In either case, we say that {\em $\{ \mu,\nu \}$ ends at
  $\mu$ and $\nu$} and {\em links these two vertices}.

This is a combinatorial definition, but we keep in mind the
geometrical intuition to think of an edge as part of the circle $C$
bounding the disc $D$, and of an arc as an actual arc inside $D$.

The arcs $\{ \mu,\nu \}$ and $\{ \pi,\rho \}$ are said to {\em cross}
if $\mu < \pi < \nu < \rho$ or $\pi < \mu < \rho < \nu$.  This is
compatible in an evident way with the geometrical intuition of the
previous paragraph, see Figure \ref{fig:Ptolemy}.
\end{Definition}

\begin{Definition}
[Internal, connecting, clockwise, and anticlockwise arcs]
An arc $\{ \mu,\nu \}$ is called {\em internal} if $\mu$ and $\nu$
belong to the same interval.  Otherwise it is called {\em
  connecting} (because it connects two different intervals).  Note
that the words {\em peripheral} and {\em bridging} are used in
essentially the same sense in \cite{BPT} and \cite{T}.

If $\{ \mu,\nu \}$ is an internal arc or an edge, then either $\nu =
\mu^{++ \cdots +}$ or $\nu = \mu^{-- \cdots -}$.  In the former case,
we say that {\em $\{ \mu,\nu \}$ goes anticlockwise from $\mu$}, in
the latter case that {\em $\{ \mu,\nu \}$ goes clockwise from $\mu$}.
\end{Definition}

\begin{Definition}
[Blocking an accumulation point]
\label{def:block}
Let $\J$ and $\K$ be neighbouring intervals of the boundary of $D_n$
separated by the accumulation point $\xi$, such that if $\iota \in
\J$ and $\kappa \in \K$ are vertices then $\iota < \xi < \kappa$.

Let $\fT$ be a set of arcs in $D_n$.  We say that $\fT$ {\em blocks
  the accumulation point} $\xi$ if it contains the configuration shown
in Figure \ref{fig:block}.
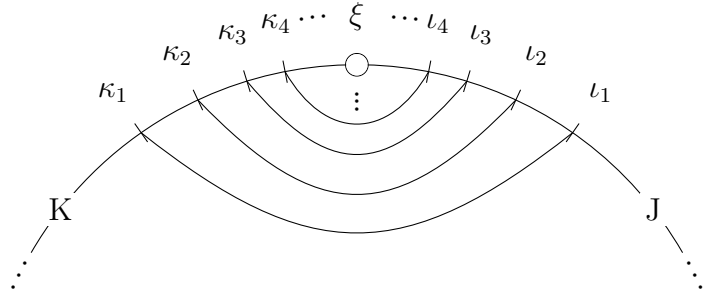
\begin{figure}
  \centering
    \begin{tikzpicture}[scale=5]
      \draw (30:1cm) arc (30:150:1cm);
%      \draw (0,0) circle (1cm);

%      \draw (0:1cm) node[fill=white,circle,inner sep=0.101cm] {} circle (0.03cm);
      \draw (90:1cm) node[fill=white,circle,inner sep=0.101cm] {} circle (0.03cm);
%      \draw (180:1cm) node[fill=white,circle,inner sep=0.101cm] {} circle (0.03cm);
%      \draw (270:1cm) node[fill=white,circle,inner sep=0.101cm] {} circle (0.03cm);

      \draw (38:1cm) node[fill=white,rectangle,inner sep=0.07cm] {$\J$};
      \draw (142:1cm) node[fill=white,rectangle,inner sep=0.07cm] {$\K$};

      \draw (24:1cm) node{$\cdot$};
      \draw (26:1cm) node{$\cdot$};
      \draw (28:1cm) node{$\cdot$};

      \draw (152:1cm) node{$\cdot$};
      \draw (154:1cm) node{$\cdot$};
      \draw (156:1cm) node{$\cdot$};

      \draw (55:0.97cm) -- (55:1.03cm);
      \draw (55:1.13cm) node{$\iota_1$};
      \draw (65:0.97cm) -- (65:1.03cm);
      \draw (65:1.13cm) node{$\iota_2$};
      \draw (73:0.97cm) -- (73:1.03cm);
      \draw (73:1.13cm) node{$\iota_3$};
      \draw (79:0.97cm) -- (79:1.03cm);
      \draw (79:1.13cm) node{$\iota_4$};
      \draw (82:1.13cm) node{$\cdot$};
      \draw (83.5:1.13cm) node{$\cdot$};
      \draw (85:1.13cm) node{$\cdot$};

      \draw (90:1.13cm) node{$\xi$};

      \draw (94.5:1.13cm) node{$\cdot$};
      \draw (96:1.13cm) node{$\cdot$};
      \draw (97.5:1.13cm) node{$\cdot$};
      \draw (101:0.97cm) -- (101:1.03cm);
      \draw (101:1.13cm) node{$\kappa_4$};
      \draw (107:0.97cm) -- (107:1.03cm);
      \draw (107:1.13cm) node{$\kappa_3$};
      \draw (115:0.97cm) -- (115:1.03cm);
      \draw (115:1.13cm) node{$\kappa_2$};
      \draw (125:0.97cm) -- (125:1.03cm);
      \draw (125:1.13cm) node{$\kappa_1$};

      \draw (55:1cm) .. controls (75:0.48cm) and (105:0.48cm) .. (125:1cm);
      \draw (65:1cm) .. controls (80:0.58cm) and (100:0.58cm) .. (115:1cm);
      \draw (73:1cm) .. controls (84:0.70cm) and (96:0.70cm) .. (107:1cm);
      \draw (79:1cm) .. controls (84:0.8cm) and (96:0.8cm) .. (101:1cm);

      \draw (90:0.88cm) node{$\cdot$};
      \draw (90:0.90cm) node{$\cdot$};
      \draw (90:0.92cm) node{$\cdot$};

    \end{tikzpicture} 
  \caption{The arcs block the accumulation point $\xi$.}
\label{fig:block}
\end{figure}

More formally, $\fT$ must contain arcs $\{ \iota_i,\kappa_i \}$ for
$i \geqslant 1$ where the vertices $\iota_i \in \J$ and $\kappa_i \in
\K$ satisfy $\iota_1 < \iota_2 < \cdots < \xi < \cdots < \kappa_2 <
\kappa_1$.  Note that the $\iota_i$ and the $\kappa_i$ converge to
$\xi$ from opposite sides.
\end{Definition}

\begin{Definition}
[Triangulations]
\label{def:triangulation}
A set of pairwise non-crossing arcs in $D_n$ is called a {\em partial
  triangulation of $D_n$}, and a maximal set of pairwise non-crossing
arcs in $D_n$ is called a {\em triangulation of $D_n$}.

A partial triangulation $\fT$ of $D_n$ is called {\em good} if it
blocks each accumulation point of $D_n$.  It is called {\em locally
  finite} if, for each vertex $\mu$ of $D_n$, only finitely many arcs
in $\fT$ end at $\mu$.
\end{Definition}

\begin{Definition}
[Vertex sets compatible with a partial triangulation]
\label{def:restriction}
Let $\fT$ be a partial triangulation of $D_n$.  A finite set of $m
\geqslant 2$ vertices $\mu_1 < \mu_2 < \cdots < \mu_m$ of $D_n$ is
said to be {\em compatible with $\fT$} if each pair $\{ \mu_1,\mu_2
\}$, $\{ \mu_2,\mu_3 \}$, $\ldots$, $\{ \mu_m,\mu_1 \}$ is either an
edge or an arc in $\fT$.

The pairs $\{ \mu_1,\mu_2 \}$, $\{ \mu_2,\mu_3 \}$, $\ldots$, $\{
\mu_m,\mu_1 \}$ can be viewed as the edges of a finite polygon $P$
with vertices equal to the $\mu_i$, and we say that the set $M = \{
\mu_1, \ldots, \mu_m \}$ {\em spans} $P$.

The remaining arcs in $\fT$ between the $\mu_i$ form a partial
triangulation $\fT_P$ of $P$, and we say that {\em $\fT$ restricts to
$\fT_P$}.  If $\fT$ is a triangulation of $D_n$ then $\fT_P$ is a
triangulation of $P$.

The following special cases will play a prominent role.
%We are permitting the two-gon as $P$!
\begin{enumerate}
\setlength\itemsep{4pt}

\item Let $\J$ be an interval of the boundary of $D_n$.  If $a < d$
  are such that $\{ a^{\J},d^{\J} \} \in \fT$ or $\{ a^{\J},d^{\J} \}$
  is an edge, then the set of vertices $\{ a^{\J}, \ldots, d^{\J} \}$
  is compatible with $\fT$ and spans a finite polygon $P$ called {\em
    the polygon below $\{ a^{\J},d^{\J} \}$}.  See the left part of
  Figure \ref{fig:PP} where $n = 4$ and $\J = \I$.

\item Let $\J$ and $\K$ be distinct intervals of the boundary of
  $D_n$.  If $a \leqslant d$ and $u \leqslant x$ are such that $\{ a^{
    \J },x^{ \K } \}, \{ d^{ \J },u^{ \K } \} \in \fT$ then the set of
  vertices $a^{ \J }, \ldots, d^{ \J }, u^{ \K }, \ldots, x^{ \K}$ is
  compatible with $\fT$ and spans a finite polygon $R$ called {\em the
    polygon between $\{ a^{ \J },x^{ \K } \}$ and $\{ d^{ \J },u^{ \K
    } \}$}.  See the right part of Figure \ref{fig:PP} where $n = 4$,
  $\J = \I$, $\K = \III$.

\end{enumerate}
\begin{figure}
  \centering
  \begin{tabular}[h]{ccc}
      \begin{tikzpicture}[scale=3]
  %      \draw[step=.25cm,gray,very thin] (-1.4,-1.4) grid (1.4,1.4);

        \draw (0,0) circle (1cm);

        \draw (45:1cm) node[fill=white,circle,inner sep=0.059cm] {} circle (0.03cm);
        \draw (135:1cm) node[fill=white,circle,inner sep=0.059cm] {} circle (0.03cm);
        \draw (225:1cm) node[fill=white,circle,inner sep=0.059cm] {} circle (0.03cm);
        \draw (315:1cm) node[fill=white,circle,inner sep=0.059cm] {} circle (0.03cm);

        \draw (64:0.97cm) -- (64:1.03cm);
        \draw (64:1.13cm) node{$\scriptstyle a$};
        \draw (78:0.97cm) -- (78:1.03cm);
        \draw (78:1.13cm) node{$\scriptstyle a+1$};
        \draw (88:1.12cm) node{$\cdot$};
        \draw (90:1.12cm) node{$\cdot$};
        \draw (92:1.12cm) node{$\cdot$};
        \draw (102:0.97cm) -- (102:1.03cm);
        \draw (102:1.13cm) node{$\scriptstyle d-1$};
        \draw (116:0.97cm) -- (116:1.03cm);
        \draw (116:1.13cm) node{$\scriptstyle d$};

        \draw (52.5:1cm) node[fill=white,rectangle,inner sep=0.07cm] {$\scriptstyle \I$};
        \draw (180:1cm) node[fill=white,rectangle,inner sep=0.07cm] {$\scriptstyle \II$};
        \draw (270:1cm) node[fill=white,rectangle,inner sep=0.07cm] {$\scriptstyle \III$};
        \draw (360:1cm) node[fill=white,rectangle,inner sep=0.07cm] {$\scriptstyle \IV$};

%The following bit is pretty silly: It typesets white (hence invisible)
%text below the circle on the left.  The purpose is to make the two
%circles appear to be top aligned.  Tabular bottom aligns by default
%and it's apparently very difficult to persuade it otherwise...
%        \draw[white] (-64:0.97cm) -- (-64:1.03cm);
        \draw[white] (-64:1.13cm) node{$\scriptstyle x$};
%        \draw[white] (-78:0.97cm) -- (-78:1.03cm);
        \draw[white] (-78:1.13cm) node{$\scriptstyle x-1$};
        \draw[white] (-88:1.12cm) node{$\cdot$};
        \draw[white] (-90:1.12cm) node{$\cdot$};
        \draw[white] (-92:1.12cm) node{$\cdot$};
%        \draw[white] (-102:0.97cm) -- (-102:1.03cm);
        \draw[white] (-102:1.13cm) node{$\scriptstyle u+1$};
%        \draw[white] (-116:0.97cm) -- (-116:1.03cm);
        \draw[white] (-116:1.13cm) node{$\scriptstyle u$};

        \draw (64:1cm) .. controls (80:0.65cm) and (100:0.65cm) .. (116:1cm);
%        \draw (-64:1cm) .. controls (-80:0.65cm) and (-100:0.65cm) .. (-116:1cm);

        \draw (90:0.85cm) node{$P$};
%        \draw (-90:0.85cm) node{$Q$};

      \end{tikzpicture} 
    & \;\; &
      \begin{tikzpicture}[scale=3]
  %      \draw[step=.25cm,gray,very thin] (-1.4,-1.4) grid (1.4,1.4);

        \draw (0,0) circle (1cm);

        \draw (45:1cm) node[fill=white,circle,inner sep=0.059cm] {} circle (0.03cm);
        \draw (135:1cm) node[fill=white,circle,inner sep=0.059cm] {} circle (0.03cm);
        \draw (225:1cm) node[fill=white,circle,inner sep=0.059cm] {} circle (0.03cm);
        \draw (315:1cm) node[fill=white,circle,inner sep=0.059cm] {} circle (0.03cm);

        \draw (64:0.97cm) -- (64:1.03cm);
        \draw (64:1.13cm) node{$\scriptstyle a$};
        \draw (78:0.97cm) -- (78:1.03cm);
        \draw (78:1.13cm) node{$\scriptstyle a+1$};
        \draw (88:1.12cm) node{$\cdot$};
        \draw (90:1.12cm) node{$\cdot$};
        \draw (92:1.12cm) node{$\cdot$};
        \draw (102:0.97cm) -- (102:1.03cm);
        \draw (102:1.13cm) node{$\scriptstyle d-1$};
        \draw (116:0.97cm) -- (116:1.03cm);
        \draw (116:1.13cm) node{$\scriptstyle d$};

        \draw (-64:0.97cm) -- (-64:1.03cm);
        \draw (-64:1.13cm) node{$\scriptstyle x$};
        \draw (-78:0.97cm) -- (-78:1.03cm);
        \draw (-78:1.13cm) node{$\scriptstyle x-1$};
        \draw (-88:1.12cm) node{$\cdot$};
        \draw (-90:1.12cm) node{$\cdot$};
        \draw (-92:1.12cm) node{$\cdot$};
        \draw (-102:0.97cm) -- (-102:1.03cm);
        \draw (-102:1.13cm) node{$\scriptstyle u+1$};
        \draw (-116:0.97cm) -- (-116:1.03cm);
        \draw (-116:1.13cm) node{$\scriptstyle u$};

        \draw (-64:1cm) .. controls (0.2,-0.3) and (0.2,0.3) .. (64:1cm);
        \draw (-116:1cm) .. controls (-0.2,-0.3) and (-0.2,0.3) .. (116:1cm);

        \draw (0,0) node{$R$};

        \draw (52.5:1cm) node[fill=white,rectangle,inner sep=0.07cm] {$\scriptstyle \I$};
        \draw (180:1cm) node[fill=white,rectangle,inner sep=0.07cm] {$\scriptstyle \II$};
        \draw (233.5:1cm) node[fill=white,rectangle,inner sep=0.07cm] {$\scriptstyle \III$};
        \draw (360:1cm) node[fill=white,rectangle,inner sep=0.07cm] {$\scriptstyle \IV$};

    \end{tikzpicture} 
  \end{tabular}
  \caption{There is a finite polygon $P$ below the arc $\{
    a^{\I},d^{\I} \}$.  The vertices of $P$ are $a^{\I}$, $( a+1
    )^{\I}$, $\ldots$, $( d-1 )^{\I}$, $d^{\I}$.  Among them, $( a+1
    )^{\I}$, $( a+2 )^{\I}$, $\ldots$, $( d-2 )^{\I}$, $( d-1 )^{\I}$
    are said to be strictly below $\{ a^{\I},d^{\I} \}$.  There is
    similar terminology for the finite polygon $R$, see Definitions
    \ref{def:restriction} and \ref{def:below}.}
\label{fig:PP}
\end{figure}
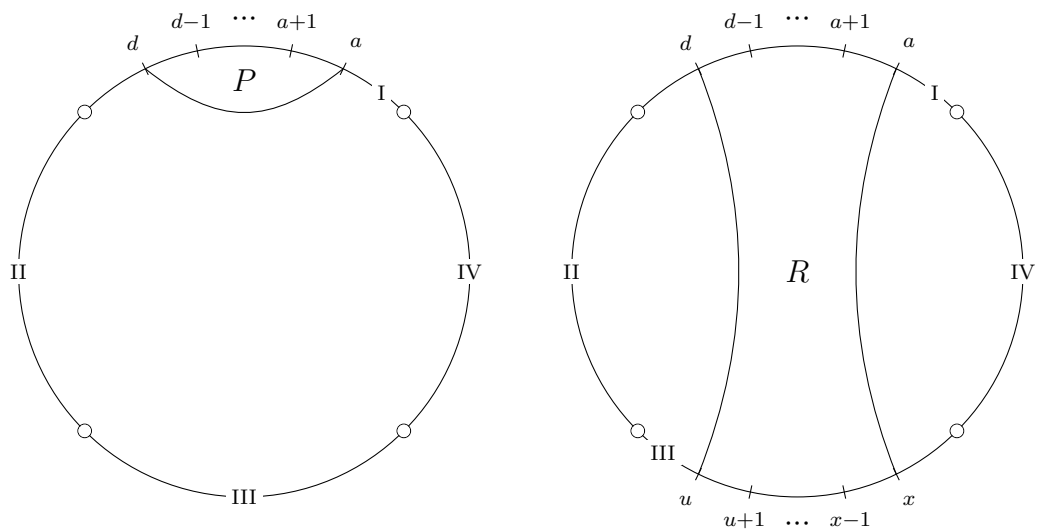
\end{Definition}

\begin{Lemma}
\label{lem:good}
Let $\fT$ be a good triangulation of $D_n$ and let $N$ be a finite set
of vertices of $D_n$.  Then there exists a finite set $M$ of vertices
such that $N \subseteq M$ and $M$ is compatible with $\fT$.

The set $M$ spans a finite polygon $P$, and $\fT$ restricts to
a triangulation $\fT_P$ of $P$.
\end{Lemma}

\begin{proof}
Consider the following construction of a set $M$ of vertices
compatible with $\fT$:

Start by including in $M$ a vertex in interval $\I$.  Move
anticlockwise around $D_n$ and include in $M$ the vertices in interval
$\I$ encountered.  End with a vertex linked to interval $\II$ by an
arc in $\fT$.

Continue by including in $M$ the vertex at the other end of this arc.
Move anticlockwise around $D_n$ and include in $M$ the vertices in
interval $\II$ encountered.  End with a vertex linked to interval
$\III$ by an arc in $\fT$.

Continue in the same fashion, thereby defining a set of vertices $M$
as shown in Figure \ref{fig:good}.
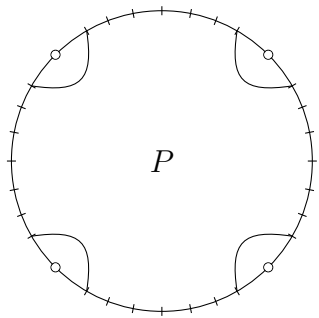
\begin{figure}
  \centering
    \begin{tikzpicture}[scale=2]
      \draw (0,0) circle (1cm);

      \draw (45:1cm) node[fill=white,circle,inner sep=0.037cm] {} circle (0.03cm);
      \draw (135:1cm) node[fill=white,circle,inner sep=0.037cm] {} circle (0.03cm);
      \draw (225:1cm) node[fill=white,circle,inner sep=0.037cm] {} circle (0.03cm);
      \draw (315:1cm) node[fill=white,circle,inner sep=0.037cm] {} circle (0.03cm);
%inner sep at scale=3 is 0.059

      \draw (-30:0.97cm) -- (-30:1.03cm);
      \draw (-21:0.97cm) -- (-21:1.03cm);
      \draw (-11:0.97cm) -- (-11:1.03cm);
      \draw (0:0.97cm) -- (0:1.03cm);
      \draw (11:0.97cm) -- (11:1.03cm);
      \draw (21:0.97cm) -- (21:1.03cm);
      \draw (30:0.97cm) -- (30:1.03cm);

      \draw (60:0.97cm) -- (60:1.03cm);
      \draw (69:0.97cm) -- (69:1.03cm);
      \draw (79:0.97cm) -- (79:1.03cm);
      \draw (90:0.97cm) -- (90:1.03cm);
      \draw (101:0.97cm) -- (101:1.03cm);
      \draw (111:0.97cm) -- (111:1.03cm);
      \draw (120:0.97cm) -- (120:1.03cm);

      \draw (150:0.97cm) -- (150:1.03cm);
      \draw (159:0.97cm) -- (159:1.03cm);
      \draw (169:0.97cm) -- (169:1.03cm);
      \draw (180:0.97cm) -- (180:1.03cm);
      \draw (191:0.97cm) -- (191:1.03cm);
      \draw (201:0.97cm) -- (201:1.03cm);
      \draw (210:0.97cm) -- (210:1.03cm);

      \draw (240:0.97cm) -- (240:1.03cm);
      \draw (249:0.97cm) -- (249:1.03cm);
      \draw (259:0.97cm) -- (259:1.03cm);
      \draw (270:0.97cm) -- (270:1.03cm);
      \draw (281:0.97cm) -- (281:1.03cm);
      \draw (291:0.97cm) -- (291:1.03cm);
      \draw (300:0.97cm) -- (300:1.03cm);

      \draw (300:1cm) .. controls (310:0.7cm) and (320:0.7cm) .. (-30:1cm);
      \draw (30:1cm) .. controls (40:0.7cm) and (50:0.7cm) .. (60:1cm);
      \draw (120:1cm) .. controls (130:0.7cm) and (140:0.7cm) .. (150:1cm);
      \draw (210:1cm) .. controls (220:0.7cm) and (230:0.7cm) .. (240:1cm);

      \draw (0,0) node{$P$};

    \end{tikzpicture} 
  \caption{The four arcs are elements of a good triangulation $\fT$
      of $D_4$, so the set $M$ of vertices shown in the figure is
      compatible with $\fT$.  The set $M$ spans a finite
      polygon $P$.  The four arcs can be viewed as four of the edges
      of $P$, and $\fT$ restricts to a triangulation $\fT_P$ of $P$.}
\label{fig:good}
\end{figure}
The set $M$ spans a finite polygon $P$ and $\fT$ restricts to a
triangulation $\fT_P$ of $P$; see Definition \ref{def:restriction}.

This proves the lemma since we can always accomplish $N \subseteq M$
by making $M$ sufficiently big.  Namely, $N$ is finite, and the arcs
which link different intervals in the construction of $M$ can be
chosen arbitrarily close to the accumulation points because $\fT$ is
good.
\end{proof}

\section{Conway--Coxeter counting}
\label{sec:CC}

\begin{Definition}
[Conway--Coxeter counting]
\label{def:CC}
Let $P$ be a finite polygon with a triangulation $\fS$ and fix a
vertex $\mu$ of $P$.  The following procedure is due to
\cite[(32)]{CC2}, see also \cite[sec.\ 2]{BCI}.  We will refer to it
as {\em Conway--Coxeter counting}:

Each vertex of $P$ is assigned a non-negative integer by the following
inductive procedure.  The vertex $\mu$ is assigned $0$.  If $\{
\mu,\nu \}$ is an edge or an arc in $\fS$, then the vertex $\nu$
is assigned $1$.  If there is a triangle in $\fS$ of which only two
vertices, say $\pi$ and $\rho$, have been assigned integers, say $a$
and $b$, then the third vertex is assigned $a+b$.

We let $\fS( \mu,\nu )$ denote the integer assigned to vertex $\nu$.
\end{Definition}

\begin{Remark}
\label{rmk:CC}
Let $\fT$ be a good triangulation of $D_n$ and let $\mu$, $\nu$ be
vertices of $D_n$.

By Lemma \ref{lem:good}, we can pick a finite set of vertices $M$ such
that $\mu, \nu \in M$ and such that $M$ is compatible with $\fT$ in
the sense of Definition \ref{def:restriction}.  The set $M$ spans a
finite polygon $P$, and $\fT$ restricts to a triangulation $\fT_P$ of
$P$, so Conway--Coxeter counting defines a non-negative integer
$\fT_P( \mu,\nu )$.

It is easy to see that $\fT_P( \mu,\nu )$ does not depend on the
choice of vertex set $M$.  Indeed, $\fT_P( \mu,\nu )$ can be computed
by following the inductive procedure of Definition \ref{def:CC} on
$\fT$ itself.  Accordingly, we drop the subscript $P$ and write
$\fT( \mu,\nu )$.
\end{Remark}

\begin{Lemma}
[Basic properties of Conway--Coxeter counting]
\label{lem:CC}
Let $\fT$ be a triangulation of a finite polygon $P$ or a good
triangulation of $D_n$.  Then Conway--Coxeter counting has the
following properties. 
\begin{enumerate}
\setlength\itemsep{4pt}

  \item  Each $\fT( \mu,\nu )$ is a well-defined non-negative
    integer. 

  \item  $\fT( \mu,\nu ) = 0$ if and only if $\mu = \nu$. 

  \item  $\fT( \mu,\nu ) = 1$ if and only if $\mu$ and $\nu$ are
    consecutive vertices or $\{ \mu,\nu \} \in \fT$.  

  \item  $\fT( \mu,\nu ) = \fT( \nu,\mu )$.

  \item If the arcs $\{ \mu,\nu \}$ and $\{
    \pi,\rho \}$ cross, then we have the following Ptolemy relation
    illustrated by Figure \ref{fig:Ptolemy}.
\[
  \fT( \mu,\nu )\fT( \pi,\rho )
  = \fT( \mu,\pi )\fT( \nu,\rho )
    + \fT( \mu,\rho )\fT( \nu,\pi ).
\]
\begin{figure}
  \centering
    \begin{tikzpicture}[scale=2]
      \draw (0,0) circle (1cm);

      \draw (-20:0.97cm) -- (-20:1.03cm);
      \draw (-20:1.13cm) node{$\rho$};
      \draw (69:0.97cm) -- (69:1.03cm);
      \draw (69:1.13cm) node{$\mu$};
      \draw (179:0.97cm) -- (179:1.03cm);
      \draw (179:1.13cm) node{$\pi$};
      \draw (291:0.97cm) -- (291:1.03cm);
      \draw (291:1.13cm) node{$\nu$};

      \draw (-20:1cm) .. controls (-15:0.5cm) and (174:0.5cm) .. (179:1cm);
      \draw (69:1cm) .. controls (74:0.5cm) and (286:0.5cm) .. (291:1cm);

      \draw[dashed] (-20:1cm) .. controls (0:0.6cm) and (49:0.6cm) .. (69:1cm);
      \draw[dashed] (69:1cm) .. controls (74:0.6cm) and (174:0.6cm) .. (179:1cm);
      \draw[dashed] (179:1cm) .. controls (184:0.6cm) and (286:0.6cm) .. (291:1cm);
      \draw[dashed] (291:1cm) .. controls (296:0.8cm) and (-25:0.8cm) .. (-20:1cm);

    \end{tikzpicture} 
  \caption{The arcs $\{ \mu,\nu \}$ and $\{ \pi,\rho \}$ cross since
      $\mu < \pi < \nu < \rho$.  The crossing gives the Ptolemy
      relation $\fT( \mu,\nu )\fT( \pi,\rho ) = \fT( \mu,\pi )\fT(
      \nu,\rho ) + \fT( \mu,\rho )\fT( \nu,\pi )$.}
\label{fig:Ptolemy}
\end{figure}
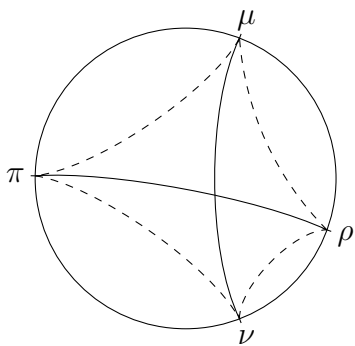

  \item  If $\mu^-, \mu, \mu^+$ are three consecutive vertices
    then  
\[
  \fT( \mu^-,\mu^+ ) = 1 + (\mbox{the number of arcs in $\fT$
    ending at $\mu$}).
\]

\end{enumerate}
\end{Lemma}

\begin{proof}
By Remark \ref{rmk:CC}, the case of a good triangulation of $D_n$
reduces to the case of a triangulation of a finite polygon.  In this
case, all the properties are well-known; indeed, (i) through (iii) are
clear from Definition \ref{def:CC} and Remark \ref{rmk:CC}.  For (iv)
see \cite[cor.\ 1]{BCI}, for (v) see \cite[sec.\ 4.2]{MG}, and for (vi)
see (27) in \cite{CC1} and \cite{CC2} or \cite[thm.\ 4.3]{MG}.
\end{proof}

\section{$\SL2$-tilings in the abstract and $\SL2$-tilings arising
  from triangulations of the disc}
\label{sec:SL2}

\begin{Definition}
\label{def:tpq}
Let $A \subseteq \BZ \times \BZ$ be given.  A {\em partial
  $\SL2$-tiling defined on $A$} is a map $t : A \rightarrow \{
1,2,3,\ldots \}$ such that
\[
  \begin{vmatrix}
    t( i,j ) & t( i,j+1 ) \\
    t( i+1,j ) & t( i+1,j+1 )
  \end{vmatrix}
  = 1
\]
whenever the determinant makes sense.

The values of $t$ are called {\em entries} of the partial
$\SL2$-tiling.  We always write the entries $t( i,j )$ in matrix
style, so $i$ increases when we move down, $j$ increases when we move
right.  Compass directions and words like {\em row}, {\em column},
{\em first quadrant}, and {\em third quadrant} are to be interpreted
in this context, see Figure \ref{fig:coordinates}.
\begin{figure}
  \centering
    \begin{tikzpicture}[scale=1]
%      \draw[step=1cm,gray,very thin] (-7,-7) grid (7,7);

      \draw [blue!30, line width=8, dashed] (-4,1.6) -- (-3,1.6);
      \draw [blue!30, line width=8] (-3,1.6) -- (3,1.6);
      \draw [blue!30, line width=8, dashed] (3,1.6) -- (4,1.6);
      \draw (1.5,1.6) node{row};

      \draw [blue!30, line width=8, dashed] (-0.7,-4) -- (-0.7,-3);
      \draw [blue!30, line width=8] (-0.7,-3) -- (-0.7,3);
      \draw [blue!30, line width=8, dashed] (-0.7,3) -- (-0.7,4);
      \node[rotate=270] at (-0.7,-1.0) {column};

      \draw (2,0.7) node{first quadrant};
      \draw (-2.6,-0.7) node{third quadrant};

      \draw[->] (-4,0) -- (4,0);
      \draw[->] (0,4) -- (0,-4);
      \draw (0,-4.5) node{$i$};
      \draw (4.5,0) node{$j$};

%      \draw (2,-1) node{$t( i,j )$};

      \draw[->] (2.5,2.5) -- (3.5,3.5);
      \draw (4,4) node{northeast};

      \draw[->] (2.5,-2.5) -- (3.5,-3.5);
      \draw (4,-4) node{southeast};

      \draw[->] (-2.5,-2.5) -- (-3.5,-3.5);
      \draw (-4,-4) node{southwest};

      \draw[->] (-2.5,2.5) -- (-3.5,3.5);
      \draw (-4,4) node{northwest};

    \end{tikzpicture} 
  \caption{When writing the entries $t( i,j )$ of an $\SL2$-tiling
    $t$, we use matrix style so $i$ increases downwards, $j$ increases
    to the right.}
\label{fig:coordinates}
\end{figure}
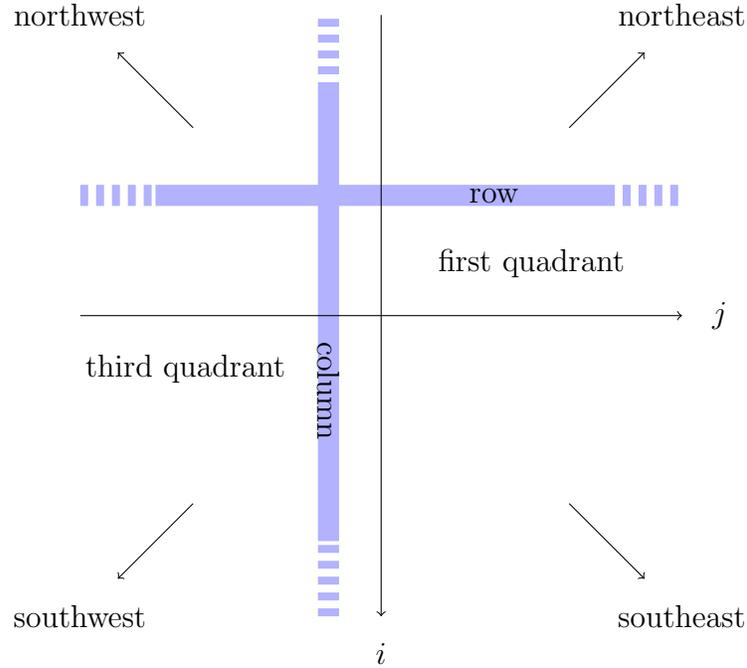

If $A = \BZ \times \BZ$ then $t$ is simply called an {\em
  $\SL2$-tiling}, see Figure \ref{fig:intro_tiling_new}.

If $A = \{ ( i,j ) \,\mid\, i < j \}$ and $t( i,i+1 ) = 1$ for each
$i$, then $t$ is called an {\em infinite frieze}, see Figure
\ref{fig:intro_tiling3}.  Infinite friezes were introduced in
\cite[def.\ 1.1]{T}.  Note that we index them differently from
\cite{T} to improve compatibility with $\SL2$-tilings.  When $t$ is an
infinite frieze, we set $t( i,i ) = 0$.  

If $A$ is a diagonal band running northwest to southeast and $t$ is
equal to $1$ on both edges of the band, then $t$ is called a {\em
  Conway--Coxeter frieze}.  These were introduced in \cite{CC1} and
\cite{CC2} where the band is typeset horizontally, that is, rotated by
$45$ degrees anticlockwise compared to our notation, see Figure
\ref{fig:frieze}.

By (21) in \cite{CC1} and \cite{CC2}, a Conway--Coxeter
frieze has a fundamental domain which is the restriction $t\!\mid_F$
of the frieze $t$ to a triangle $F$ as shown in Figure
\ref{fig:p_tiling}.  The lower edge of the frieze is the diagonal with
all entries equal to $1$.  The upper edge of the frieze contains the
$1$ at the upper right corner of the triangle.  The entries of the
whole frieze are obtained by tiling the diagonal band $A$ with translations of
$t\!\mid_F$ and its reflection in a line running northwest to southeast.
\end{Definition}

\begin{Setup}
\label{set:blanket}
Throughout, $t$ is an $\SL2$-tiling.  Recall from \cite[sec.\ 5]{HJ}
that there are associated infinite friezes $p$ and $q$ defined by 
\[
  p( a,d ) =
  \begin{vmatrix}
    t( a,w ) & t( a,w+1 ) \\
    t( d,w ) & t( d,w+1 )
  \end{vmatrix}
  \;\;,\;\;
  q( u,x ) =
  \begin{vmatrix}
    t( c,u ) & t( c,x ) \\
    t( c+1,u ) & t( c+1,x )
  \end{vmatrix}
\]
for integers $a \leqslant d$, $u \leqslant x$.  Note that the integers
$w$ and $c$ can be chosen freely by \cite[rmk.\ 5.2]{HJ}, and that $p(
a,d )$ and $q( u,x )$ are indeed positive for $a < d$ and $u < x$ by
\cite[prop.\ 5.6]{HJ}. 
\end{Setup}

\begin{Lemma}
\label{lem:Ptolemy}
We have the following Ptolemy relations.
\begin{enumerate}
\setlength\itemsep{4pt}

  \item  Let $a < b < c < d$ be integers.  Then
\[
  p( a,c )p( b,d ) = p( a,b )p( c,d ) + p( a,d )p( b,c )
  \; \mbox{ and } \;
  q( a,c )q( b,d ) = q( a,b )q( c,d ) + q( a,d )q( b,c ).
\]

  \item  Let $v$ and $a < b < c$ be integers.  Then
\[
  p( a,c )t( b,v ) = p( b,c )t( a,v ) + p( a,b )t( c,v )
  \; \mbox{ and } \;
  q( a,c )t( v,b ) = q( b,c )t( v,a ) + q( a,b )t( v,c ).
\]

  \item  Let $b < c$ and $v < w$ be integers.  Then
\[
  t( b,v )t( c,w ) = t( b,w )t( c,v ) + p( b,c )q( v,w ). 
\]

\end{enumerate}
\end{Lemma}

\begin{proof}
See \cite{HJ}, propositions 5.4, 5.5, and 5.7.
\end{proof}

\begin{Lemma}
\label{lem:finiteness_in_rows_and_colums}
Let $i$, $m$ be integers with $m > 0$.  The entry $m$ occurs only finitely many
times in each of the following (half-)rows and (half)-columns:
\[
  t( i,- )
  \;,\;
  t( -,i )
  \;,\;
  p( i,- )
  \;,\;
  p( -,i )
  \;,\;
  q( i,- )
  \;,\;
  q( -,i ).
\]
\end{Lemma}

\begin{proof}
The statements about $t( i,- )$ and $t( -,i )$ are \cite[prop.\
6.1]{HJ}.  The proof of that result can be modified as follows to show 
the remaining statements:

Suppose that $i < j < k$ are integers with $p( i,j ) = p(
i,k ) = m$.  The Ptolemy relation \ref{lem:Ptolemy}(i) gives
\[
  p( i-1,j )p( i,k )
  = p( i-1,i )p( j,k ) + p( i-1,k )p( i,j )
  = p( j,k ) + p( i-1,k )p( i,j )
\]
whence
\[
  p( j,k )
  =
  \begin{vmatrix}
    p( i-1,j ) & p( i-1,k ) \\
    p( i,j )   & p( i,k )
  \end{vmatrix}
  =
  \begin{vmatrix}
    p( i-1,j ) & p( i-1,k ) \\
    m          & m
  \end{vmatrix}
  = m \cdot \big( p( i-1,j ) - p( i-1,k ) \big).
\]
Since $p( j,k ) > 0$ we learn $p( i-1,j ) > p( i-1,k )$.

Hence if there is a sequence of integers $i < j < k < \cdots$ with
$p( i,j ) = p( i,k ) = \cdots = m$, then $p( i-1,j ) > p( i-1,k ) >
\cdots$.  Since the entries of $p$ are non-negative, this implies that the
sequence is finite, so the half-row $p( i,- )$ has only finitely
many entries equal to $m$.

The remaining claims are proved symmetrically.
\end{proof}

\begin{Proposition}
\label{pro:tiling_from_fT}
Let $\fT$ be a good triangulation of $D_n$, the disc with $n$
accumulation points where $n \in \{ 2,3,4 \}$. 
\begin{enumerate}
\setlength\itemsep{4pt}

  \item  The map $( b,v ) \mapsto \fT( b^{\I},v^{\III} )$ is an
    $\SL2$-tiling. 

  \item  The map $( a,d ) \mapsto \fT( a^{\I},d^{\I} )$ is an infinite
    frieze defined in the half plane $\{ ( a,d ) \mid a \leqslant d
    \}$.

  \item  The map $( u,x ) \mapsto \fT( u^{\III},x^{\III} )$ is an
    infinite frieze defined in the half plane $\{ ( u,x ) \mid u
    \leqslant x \}$.

\end{enumerate}
\end{Proposition}

\begin{proof}
(i)  We have
\begin{align*}
  \lefteqn{
    \begin{vmatrix}
       \fT( b^{\I},v^{\III} ) & \fT\big( b^{\I},( v+1 )^{\III} \big) \\
       \fT\big( ( b+1 )^{\I},v^{\III} \big) & \fT\big( ( b+1 )^{\I},( v+1 )^{\III} \big) \\
    \end{vmatrix}
          } & \\[1.5mm]
  & \;\;\;\;\; = \fT( b^{\I},v^{\III} )\fT\big( ( b+1 )^{\I},( v+1 )^{\III} \big)
    - \fT\big( b^{\I},( v+1 )^{\III} \big)\fT\big( ( b+1 )^{\I},v^{\III} \big)\\
  & \;\;\;\;\; = \fT\big( b^{\I},( b+1 )^{\I} \big)\fT\big( v^{\III},( v+1 )^{\III} \big)\\
  & \;\;\;\;\; = 1,
\end{align*}
where the second equality is by the Ptolemy relation in Lemma
\ref{lem:CC}(v) and the last equality is by Lemma \ref{lem:CC}(iii).

(ii) and (iii) are analogous to (i).
\end{proof}

\begin{Remark}
As explained, the point of the paper is to show that {\em every}
$\SL2$-tiling arises as in Proposition \ref{pro:tiling_from_fT}(i).
\end{Remark}

\begin{Definition}
The $\SL2$-tiling $( b,v ) \mapsto \fT( b^{\I},v^{\III} )$ from
Proposition \ref{pro:tiling_from_fT}(i) is said to {\em arise from
  $\fT$ by Conway--Coxeter counting}.  It will be denoted $\Phi( \fT
)$. 
\end{Definition}

\begin{Lemma}
\label{lem:agree}
Let $t$ be an $\SL2$-tiling with associated infinite friezes $p$, $q$.  Let $\fT$ be a good triangulation of $D_n$.

On the one hand, assume $t = \Phi( \fT )$, that is, $t( b,v ) = \fT(
b^{\I},v^{\III} )$ for all integers $b$, $v$.  Then
\begin{enumerate}
\setlength\itemsep{4pt}

  \item  $p( b-1,b+1 ) = \fT\big( ( b-1 )^{\I},( b+1 )^{\I} \big)$ for
    each $b$,

  \item  $q( v-1,v+1 ) = \fT\big( ( v-1 )^{\III},( v+1 )^{\III} \big)$
    for each $v$.  

\end{enumerate}

On the other hand, assume that (i) and (ii) hold along with at least
one of the following two conditions.
\begin{itemize}
\setlength\itemsep{4pt}

  \item[(iii)]  There are integers $e < f$ and $g < h$ such that
\[
  t( e,g ) = \fT( e^{\I},g^{\III} )
  \;,\;
  t( f,g ) = \fT( f^{\I},g^{\III} )
  \;,\;
  t( f,h ) = \fT( f^{\I},h^{\III} ).
\]

  \item[(iii)']  There are integers $e < f$ and $g < h$ such that
\[
  t( e,h ) = \fT( e^{\I},h^{\III} )
  \;,\;
  t( f,g ) = \fT( f^{\I},g^{\III} )
  \;,\;
  t( f,h ) = \fT( f^{\I},h^{\III} ).
\]

\end{itemize}
Then $t = \Phi( \fT )$, that is, $t( b,v ) = \fT( b^{\I},v^{\III} )$
for all $b$ and $v$.
\end{Lemma}

\begin{proof}
``On the one hand'':  Let $b$ be a given integer and pick an arbitrary
integer $v$.  Then 
\begin{align*}
  p( b-1,b+1 )t( b,v )
  & = p( b-1,b )t( b+1,v ) + p( b,b+1 )t( b-1,v ) \\
  & = t( b+1,v ) + t( b-1,v )
\end{align*}
where the first equality is by the Ptolemy relation in Lemma
\ref{lem:Ptolemy}(ii).  Moreover,
\begin{align*}
  \lefteqn{
    \fT\big( ( b-1 )^{\I},( b+1 )^{\I} \big)\fT( b^{\I},v^{\III} )
          }
  & \\
  & \;\;\;\; = \fT\big( ( b-1 )^{\I},b^{\I} \big)\fT\big( ( b+1 )^{\I},v^{\III} \big)
    + \fT\big( b^{\I},( b+1 )^{\I} \big)\fT\big( ( b-1 )^{\I},v^{\III} \big)\\
  & \;\;\;\; = \fT\big( ( b+1 )^{\I},v^{\III} \big)
    + \fT\big( ( b-1 )^{\I},v^{\III} \big)
\end{align*}
where the first equality is by the Ptolemy relation in Lemma \ref{lem:CC}(v)
and the second equality is by Lemma \ref{lem:CC}(iii).

By assumption, we have $t( c,v ) = \fT( c^{\I},v^{\III} )$ for each
$c$.  In particular this holds for $c$ equal to $b-1$, $b$, or
$b+1$.  The two displayed equations therefore combine to give $p(
b-1,b+1 ) = \fT\big( ( b-1 )^{\I},( b+1 )^{\I} \big)$.  This shows (i),
and (ii) follows by symmetry.

``On the other hand'': Consider the infinite friezes $p$ and $( a,d )
\mapsto \fT( a^{\I},d^{\I} )$, see Proposition
\ref{pro:tiling_from_fT}(ii).  Condition (i) says that they agree on
the diagonal $\{ (b-1,b+1 ) \mid b \in \BZ \}$.  However, it is easy
to see that an infinite frieze is determined entirely by its values on
this diagonal which is also known as its quiddity sequence; see
\cite[rmk.\ 1.3]{T}.  So the two infinite friezes agree:
\begin{equation}
\label{equ:agree_p}
  p( a,d ) = \fT( a^{\I},d^{\I} ) \mbox{ for } a \leqslant d.
\end{equation}
Similarly, condition (ii) implies
\begin{equation}
\label{equ:agree_q}
  q( u,x ) = \fT( u^{\III},x^{\III} ) \mbox{ for } u \leqslant x.
\end{equation}

Now suppose condition (iii) holds.  If $b < e$ then the Ptolemy
relation in \ref{lem:Ptolemy}(ii) says
\begin{equation}
\label{equ:Ptolemy1}
  p( b,f )t( e,g ) = p( b,e )t( f,g ) + p( e,f )t( b,g )
\end{equation}
while the Ptolemy relation in Lemma \ref{lem:CC}(v) says
\begin{equation}
\label{equ:Ptolemy2}
  \fT( b^{\I},f^{\I} )\fT( e^{\I},g^{\III} )
  = \fT( b^{\I},e^{\I} )\fT( f^{\I},g^{\III} )
    + \fT( e^{\I},f^{\I} )\fT( b^{\I},g^{\III} ).
\end{equation}
Equations \eqref{equ:agree_p} and \eqref{equ:agree_q} and condition
(iii) say that the first five of the six factors in Equation
\eqref{equ:Ptolemy1} are equal to the corresponding factors in Equation
\eqref{equ:Ptolemy2}.  Hence the last factors are also equal, so
\[
  t( b,g ) = \fT( b^{\I},g^{\III} ) \mbox{ for } b < e.
\]
Similar arguments can be applied for $e < b < f$ and $f < b$, and the
cases $b = e$ and $b = f$ are handled by condition (iii) itself, so we
get
\[
  t( b,g ) = \fT( b^{\I},g^{\III} ) \mbox{ for each } b.
\]

Similar arguments prove that
\[
  t( f,v ) = \fT( f^{\I},v^{\III} ) \mbox{ for each } v.
\]
Hence the $\SL2$-tilings $t$ and $( b,v ) \mapsto \fT( b^{\I},v^{\III}
)$ match on column number $g$ and row number $f$.  However, it is easy
to see that an $\SL2$-tiling is determined entirely by its values on a
column and a row, so the two $\SL2$-tilings must agree everywhere as
claimed. 

If condition (iii)' holds then we proceed symmetrically.
\end{proof}

\begin{Lemma}
\label{lem:zig-zag}
If $t$ has entries which are equal to $1$, then they occur on a
zig-zag as shown in Figure \ref{fig:zig-zag}.  The zig-zag is
bounded or unbounded to each side. 
\begin{figure}
  \centering
  \begin{tikzpicture}[xscale=0.4,yscale=0.4]

%    \tikzset{help lines/.style={color=gray,very thin}}%or color=black!15
%    \draw[help lines] (-8,-8) grid (8,8);

      \draw[->] (-8,0) -- (8,0);
      \draw[->] (0,8) -- (0,-8);
      \draw (0,-8.5) node{$i$};
      \draw (8.5,0) node{$j$};

    \draw [blue!30, line width=8, dashed] (-5,-7) -- (-5,-6);
    \draw [rounded corners, blue!30, line width=8] (-5,-6) -- (-5,-5)
    -- (-2,-5) -- (-2,-1) -- (-2,3) -- (0,3) -- (5,3) -- (5,5) -- (6,5);
    \draw [blue!30, line width=8, dashed] (6,5) -- (7,5);

    \draw (-5,-5) node {$\scriptstyle 1$};
    \draw (-2,-5) node {$\scriptstyle 1$};
    \draw (-2,-1) node {$\scriptstyle 1$};
    \draw (-2,3) node {$\scriptstyle1$};
    \draw (0,3) node {$\scriptstyle 1$};
    \draw (5,3) node {$\scriptstyle 1$};
    \draw (5,5) node {$\scriptstyle 1$};

  \end{tikzpicture} 
  \caption{The entries of $t$ which are equal to $1$ occur on a
    zig-zag which can be bounded or unbounded to each side.}
\label{fig:zig-zag}
\end{figure}

More formally, there is an interval $K \subseteq \BZ$ which is
bounded or unbounded to each side, along with coordinate pairs $(
b_k,v_k )$ for $k \in K$ such that the following hold.
\begin{enumerate}
\setlength\itemsep{4pt}

  \item  $t( b,v ) = 1$ if and only if $( b,v ) = ( b_k,v_k )$ for
    some $k \in K$. 

  \item  For each $k \neq \max(K)$, we have precisely one of the
    following two options:

\begin{enumerate}
\setlength\itemsep{4pt}

  \item  $b_{ k+1 } < b_{ k }$ and $v_{ k+1 } = v_{ k }$ or

  \item  $b_{ k+1 } = b_{ k }$ and $v_{ k+1 } > v_{ k }$.

\end{enumerate}

  \item  If $K$ is unbounded above, then there are infinitely many
    shifts between options (a) and (b) when $k$ increases.

  \item  If $K$ is unbounded below, then there are infinitely many
    shifts between options (a) and (b) when $k$ decreases.

\end{enumerate}
\end{Lemma}

\begin{proof}
The proof of \cite[prop.\ 8.2]{HJ} works.
\end{proof}

\begin{Remark}
\label{rmk:zig-zag}
Lemma \ref{lem:zig-zag} says that the zig-zag of $1$'s in $t$, shown
in Figure \ref{fig:zig-zag}, is bounded or unbounded to each side.
There are hence four possibilities which can all be realised, as one
can see in the $\SL2$-tilings $t = \Phi( \fT )$ for various choices of
the triangulation $\fT$.

Namely, by Lemma \ref{lem:CC}(iii), the $1$'s in $t$ correspond to the
arcs in $\fT$ which connect intervals $\I$ and $\III$.

It follows that if $\fT$ has the form in Figure \ref{fig:ft_in_Case1},
then the zig-zag of $1$'s in $t$ is unbounded to both sides.  That is,
$t$ has infinitely many $1$'s in both the first and the third
quadrant.

If $\fT$ has the form in Figure \ref{fig:ft_in_Case2}, then the
zig-zag is bounded to the left and unbounded to the right.  That is,
$t$ has infinitely many $1$'s in the first, but not the third
quadrant.  The opposite situation can be obtained by reflecting $\fT$
in a vertical line.

If $\fT$ has the form in Figure \ref{fig:ft_in_Case3}, then the
zig-zag is bounded to the left and to the right.  That is, $t$ has only
finitely many $1$'s.

Finally, note that if $\fT$ has the form in Figure
\ref{fig:ft_in_Case6}, then $t$ has no entries equal to $1$.  This
corresponds to $K = \emptyset$ in Lemma \ref{lem:zig-zag}.
\end{Remark}

\section{The partial triangulation $\Theta( t )$ of an $\SL2$-tiling $t$}
\label{sec:basic}

Recall that $t$ is a fixed $\SL2$-tiling with associated infinite
friezes $p$ and $q$, see Setup \ref{set:blanket}.

\begin{Definition}
[The partial triangulation $\Theta( t )$]
\label{def:Theta}
We define a set of arcs in $D_n$ as follows. 
\begin{align*}
  \Theta( t ) =
  \big\{ \{ b^{\I},v^{\III} \} \,\big|\, t( b,v ) = 1 \big\}
  & \cup
  \big\{ \{ a^{\I},d^{\I} \} \,\big|\, a+2 \leqslant d, \, p( a,d ) = 1 \big\} \\
  & \cup
  \big\{ \{ u^{\III},x^{\III} \} \,\big|\, u+2 \leqslant x, \, q( u,x ) = 1 \big\}.
\end{align*}
The definition makes sense since intervals $\I$ and $\III$ are always
present on the boundary of $D_n$; see Definitions \ref{def:C4} and
\ref{def:C2C3}. 
\end{Definition}

\begin{Remark}
We remind the reader that when $t$ is given and we seek to construct a
good triangulation $\fT$ of $D_n$ such that $t = \Phi( \fT )$, we will
do so by adding arcs to $\Theta( t )$.   
\end{Remark}

\begin{Lemma}
\label{lem:non-crossing}
The set $\Theta( t )$ is a partial triangulation of $D_n$.
\end{Lemma}

\begin{proof}
If the internal arcs $\{ a^{\I},c^{\I} \}$ and $\{ b^{\I},d^{\I} \}$
cross, then we can suppose $a < b < c < d$.  Then the Ptolemy relation
in Lemma \ref{lem:Ptolemy}(i) says
\[
  p( a,c )p( b,d )
  = p( a,b )p( c,d ) + p( a,d )p( b,c )
  \geqslant 2,
\]
where the inequality holds since each of $p( a,b )$, $p( c,d )$, $p( a,d
)$, $p( b,c )$ is a positive integer.  Hence $p( a,c )$ and $p( b,d )$
cannot both be equal to $1$, so $\{ a^{\I},c^{\I} \}$ and $\{
b^{\I},d^{\I} \}$ cannot both be in $\Theta ( t )$.

A crossing of an internal arc and a connecting arc is handled similarly by
means of Lemma \ref{lem:Ptolemy}(ii), and a crossing of two connecting
arcs is handled by means of Lemma \ref{lem:Ptolemy}(iii).
\end{proof}

\begin{Lemma}
\label{lem:locally_finite}
The partial triangulation $\Theta( t )$ is locally finite in the sense of
Definition \ref{def:triangulation}.
\end{Lemma}

\begin{proof}
Consider the arcs in $\Theta( t )$ which end at the vertex $\mu =
b^{\I}$ in interval $\I$.  By Definition \ref{def:Theta} they
correspond to the entries which are equal to $1$ in the row $t( b,- )$
and the half-rows $p( b,- )$ and $p( -,b )$.  By Lemma
\ref{lem:finiteness_in_rows_and_colums} there are only finitely many
such entries.

The arcs in $\Theta( t )$ which end at the vertex $\mu = v^{\III}$ in
interval $\III$ are handled by symmetry.
\end{proof}

\begin{Lemma}
\label{lem:finite_triangulation1}
Let $a < d$ be such that $\{ a^{\I},d^{\I} \} \in \Theta( t )$ or $\{
a^{\I},d^{\I} \}$ is an edge, and let $P$ denote the finite polygon
below $\{ a^{\I},d^{\I} \}$, see Definition
\ref{def:restriction} and the left half of Figure \ref{fig:PP}.
\begin{enumerate}
\setlength\itemsep{4pt}

  \item The restriction $\Theta( t )_P$ is a triangulation of $P$.

  \item Conway-Coxeter counting on $\Theta( t )_P$ agrees with a
    certain part of the infinite frieze $p$ in the following sense: If
    $a \leqslant b \leqslant c \leqslant d$, then
\[
  p( b,c ) = \Theta( t )_P( b^{ \I },c^{ \I } ).
\]
\end{enumerate}
\end{Lemma}

\begin{proof}
If $\{ a^{\I},d^{\I} \}$ is an edge then $P$ is a $2$-gon and the
lemma is trivial, so suppose $\{ a^{\I},d^{\I} \} \in \Theta( t )$.
In particular $\{ a^{\I},d^{\I} \}$ is an arc so $d \geqslant a+2$.

The infinite frieze $p$ is defined on the half plane $\{ ( b,c ) \in
\BZ \times \BZ \mid b < c \}$ and we have $p( b,b+1 ) = 1$ for each
$b$.  Recall that we set $p( b,b ) = 0$ so $p$ is as shown in Figure
\ref{fig:p_tiling}.
\begin{figure}
  \centering
    \begin{tikzpicture}[scale=0.7]
%      \draw[step=1cm,gray,very thin] (-7,-7) grid (7,7);

      \draw (-8,2) node{$0$};
      \draw (-7,2) node{$1$};
      \draw (-6,2) node{$\circ$};
      \draw (-5,2) node{$\circ$};
      \draw (-7,1) node{$0$};
      \draw (-6,1) node{$1$};
      \draw (-5,1) node{$\circ$};
      \draw (-4,1) node{$\circ$};
      \draw (-6,0) node{$0$};
      \draw (-5,0) node{$1$};
      \draw (-4,0) node{$\circ$};
      \draw (-3,0) node{$\circ$};
      \draw (-2.2,0) node{$\cdot$};
      \draw (-2,0) node{$\cdot$};
      \draw (-1.8,0) node{$\cdot$};
      \draw (-1,0) node{$\circ$};
      \draw (0,0) node {$1$};
      \draw (-5,-1) node{$0$};
      \draw (-4,-1) node{$1$};
      \draw (-3,-1) node{$\circ$};
      \draw (0,-1) node{$\circ$};
      \draw (-4,-2) node{$0$};
      \draw (-3,-2) node{$1$};
      \draw (-2,-2) node{$\circ$};
      \draw (0,-1.8) node{$\cdot$};
      \draw (0,-2) node{$\cdot$};
      \draw (0,-2.2) node{$\cdot$};
      \draw (-3,-3) node{$0$};
      \draw (-2,-3) node{$1$};
      \draw (-1,-3) node{$\circ$};
      \draw (0,-3) node{$\circ$};
      \draw (-2,-4) node{$0$};
      \draw (-1,-4) node{$1$};
      \draw (0,-4) node{$\circ$};
      \draw (1,-4) node{$\circ$};
      \draw (-1,-5) node{$0$};
      \draw (0,-5) node{$1$};
      \draw (1,-5) node{$\circ$};
      \draw (2,-5) node{$\circ$};
      \draw (0,-6) node{$0$};
      \draw (1,-6) node{$1$};
      \draw (2,-6) node{$\circ$};
      \draw (1,-7) node{$0$};
      \draw (2,-7) node{$1$};
      \draw (2,-8) node{$0$};

      \draw (-9.141,3.141) node{$\cdot$};
      \draw (-9,3) node{$\cdot$};
      \draw (-8.859,2.859) node{$\cdot$};

      \draw (-6.141,3.141) node{$\cdot$};
      \draw (-6,3) node{$\cdot$};
      \draw (-5.859,2.859) node{$\cdot$};

      \draw (-4.141,2.859) node{$\cdot$};
      \draw (-4,3) node{$\cdot$};
      \draw (-3.859,3.141) node{$\cdot$};

      \draw (-3.141,1.859) node{$\cdot$};
      \draw (-3,2) node{$\cdot$};
      \draw (-2.859,2.141) node{$\cdot$};

      \draw (2.859,-8.859) node{$\cdot$};
      \draw (3,-9) node{$\cdot$};
      \draw (3.141,-9.141) node{$\cdot$};

      \draw (2.859,-5.859) node{$\cdot$};
      \draw (3,-6) node{$\cdot$};
      \draw (3.141,-6.141) node{$\cdot$};

      \draw (1.859,-3.141) node{$\cdot$};
      \draw (2,-3) node{$\cdot$};
      \draw (2.141,-2.859) node{$\cdot$};

      \draw (2.859,-4.141) node{$\cdot$};
      \draw (3,-4) node{$\cdot$};
      \draw (3.141,-3.859) node{$\cdot$};

      \draw (-1.159,-1.159) node{$\cdot$};
      \draw (-1.3,-1.3) node{$\cdot$};
      \draw (-1.441,-1.441) node{$\cdot$};

      \draw[gray,very thin] (-5.1,0.5) -- (0,0.5);
      \draw[gray,very thin] (-5.1,0.5) .. controls (-5.9,0.5) and (-6.1,0.5) .. (-5.254,-0.354);
      \draw[gray,very thin] (-5.254,-0.354) -- (-0.254,-5.354);
      \draw[gray,very thin] (0.5,-5.2) .. controls (0.5,-6.0) and (0.5,-6.2) .. (-0.254,-5.354);
      \draw[gray,very thin] (0.5,0) -- (0.5,-5.2);
      \draw[gray,very thin] (0.5,0) .. controls (0.5,0.3) and (0.3,0.5) .. (0,0.5);

      \draw[gray,very thin] (0.25,0.0) to (1.5,0.6);
      \draw (3.40,0.6) node{position $(a,d)$};

      \draw[gray,very thin] (0.5,-1.1) to (1.5,-0.55);
      \draw (1.77,-0.55) node{$F$};

    \end{tikzpicture} 
    \caption{Assume that $d \geqslant a+2$ and that the entry at
      position $( a,d )$ in the infinite frieze $p$ is equal to $1$.
      Then the restriction of $p$ to the triangle $F$ is the
      fundamental domain of a Conway--Coxeter frieze.} 
\label{fig:p_tiling}
\end{figure}
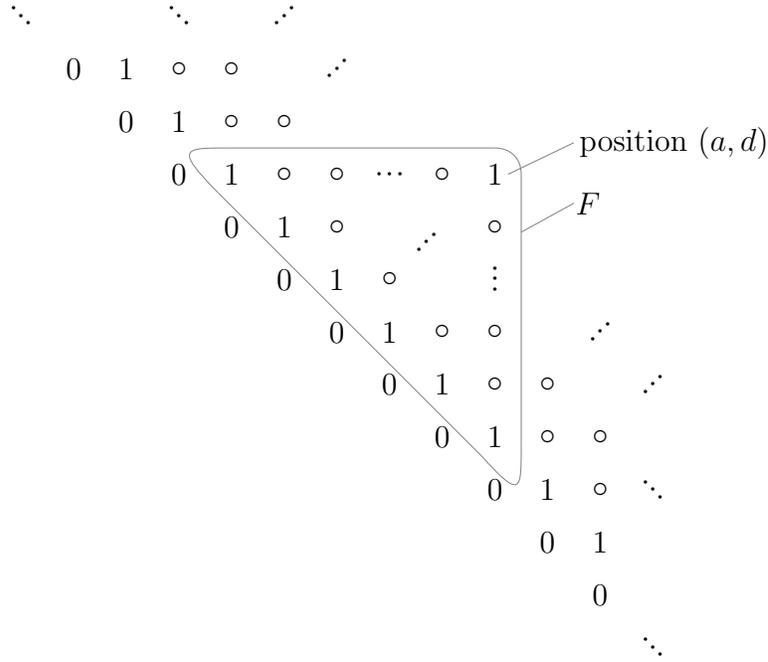
The condition $\{ a^{\I},d^{\I} \} \in \Theta( t )$ means $p( a,d ) =
1$.  In the figure, this entry is at the upper right corner of the
triangle 
\[
  F = \{ (b,c) \,\mid\, a \leqslant b < c \leqslant d \},
\]
and it implies that the restriction $p|_F$ coincides with the
fundamental domain of a Conway--Coxeter frieze $s$.  This follows from
(10) in \cite{CC1} and \cite{CC2}; see also \cite[lem.\ 7.1]{HJ}.

The elements $( b,c )$ of $F$ correspond to the edges and arcs
$\{ b^{\I},c^{\I} \}$ of the polygon with vertices $a^{\I}, \ldots,
d^{\I}$, that is, the polygon $P$.  Hence the frieze $s$ corresponds
to a triangulation $\fS$ of $P$; specifically,
\[
%\label{equ:fS_from_s}
  \fS = \big\{ \{ b^{\I},c^{\I} \} \mbox{ is an arc of $P$}
        \,\big|\, s( b,c ) = 1 \big\}.
\]
This is due to \cite{CC1}, \cite{CC2}, and more details are given in
\cite[sec.\ 2]{BCI}.  Since $s$ and $p$ agree on $F$, the equation
shows that $\fS$ consists of some of the arcs in $\Theta( t )$.
Indeed, $\fS$ is precisely the restriction $\Theta( t )_P$ of $\Theta(
t )$ to $P$.  Hence $\Theta( t )_P$ is a triangulation of $P$, proving
part (i) of the lemma.

The same references show that, conversely, the fundamental region of
the frieze $s$ can be obtained by Conway--Coxeter counting on $\fS$,
namely if $\{ b^{\I},c^{\I} \}$ is an edge or an arc of $P$ then
\[
%\label{equ:s_from_fS}
  s( b,c ) = \fS( b^{\I},c^{\I} ).
\]
This gives part (ii) of the lemma for $b < c$ because $s$ and $p$
agree on $F$ while $\fS = \Theta( t )_P$.  If $b = c$ then part (ii)
of the lemma is trivially true.
\end{proof}

\begin{Lemma}
\label{lem:finite_triangulation1b}
Let $u < x$ be such that $\{ u^{\III},x^{\III} \} \in \Theta( t )$ or
$\{ u^{\III},x^{\III} \}$ is an edge, and let $Q$ denote the finite
polygon below $\{ u^{\III},x^{\III} \}$.
\begin{enumerate}
\setlength\itemsep{4pt}

  \item The restriction $\Theta( t )_Q$ is a triangulation of $Q$.

  \item Conway-Coxeter counting on $\Theta( t )_Q$ agrees with a
    certain part of the infinite frieze $q$ in the following sense: If
    $u \leqslant v \leqslant w \leqslant x$, then
\[
  q( v,w ) = \Theta( t )_Q( v^{ \III },w^{ \III } ).
\]
\end{enumerate}
\end{Lemma}

\begin{proof}
Follows from Lemma \ref{lem:finite_triangulation1} by symmetry.
\end{proof}

\begin{Lemma}
\label{lem:finite_triangulation2}
Let $a \leqslant d$ and $u \leqslant x$ be such that $\{ a^{ \I },x^{
  \III } \}, \{ d^{ \I },u^{ \III } \} \in \Theta( t )$ and let $R$
denote the finite polygon between the arcs $\{ a^{ \I },x^{ \III } \}$
and $\{ d^{ \I },u^{ \III } \}$, see the right half of Figure
\ref{fig:PP}. 
\begin{enumerate}
\setlength\itemsep{4pt}

  \item  The restriction $\Theta( t )_R$ is a
    triangulation of $R$. 

  \item  Conway-Coxeter counting on $\Theta( t )_R$ agrees with certain
    parts of the $\SL2$-tiling $t$ and the infinite friezes $p$, $q$
    in the following sense: If $a \leqslant b \leqslant c \leqslant d$
    and $u \leqslant v \leqslant w \leqslant x$, then
\[
  t( b,v ) = \Theta( t )_R( b^{ \I }, v^{ \III } )
  \;,\;
  p( b,c ) = \Theta( t )_R( b^{ \I },c^{ \I } )
  \;,\;
  q( v,w ) = \Theta( t )_R( v^{ \III },w^{ \III } ).
\]

\end{enumerate}
\end{Lemma}

\begin{proof}
If $\{ a^{ \I },x^{ \III } \} = \{ d^{ \I },u^{ \III } \}$ then $R$ is
a $2$-gon and the lemma is trivial, so suppose that $\{ a^{ \I },x^{
  \III } \}$ and $\{ d^{ \I },u^{ \III } \}$ are distinct.

Then the proof is analogous to that of Lemma
\ref{lem:finite_triangulation1}, except that a more sophisticated
method is needed to obtain a fundamental region of a Conway--Coxeter
frieze.  Specifically, $\{ a^{ \I },x^{ \III } \}$, $\{ d^{ \I },u^{
  \III } \} \in \Theta( t )$ implies $t( a,x ) = t( d,u ) = 1$, and
these two entries of $t$ span a rectangle in the plane to which $t$
can be restricted.  It is shown in \cite[prop.\ 7.2 and fig.\ 15]{HJ}
how to position suitable restrictions of $p$ and $q$ next to the
rectangle in order to obtain a partial $\SL2$-tiling defined on a
triangle $F$.  As in the proof of Lemma
\ref{lem:finite_triangulation1}, this tiling coincides with the
fundamental region of a Conway--Coxeter frieze, and the proof then
proceeds as for Lemma \ref{lem:finite_triangulation1}.
\end{proof}

\section{Saturated vertices and defects associated to an $\SL2$-tiling}
\label{sec:defects}

Recall that $t$ is a fixed $\SL2$-tiling with associated infinite
friezes $p$ and $q$, see Setup \ref{set:blanket}.

\begin{Definition}
[Vertices strictly below and strictly between arcs]
\label{def:below}
Let $\J,\K \in \{ \I,\II,\III,\IV \}$ be intervals of the boundary of
$D_n$.
\begin{enumerate}
\setlength\itemsep{4pt}

  \item If $a \leqslant d-2$ then the vertices $( a+1 )^{\J}$, $\ldots$,
  $( d-1 )^{\J}$ are said to be {\em strictly below the (internal) arc
    $\{ a^{\J},d^{\J} \}$}, see the left half of Figure \ref{fig:PP}.

  \item If $a \leqslant d$ and $u \leqslant x$ are such that $\{ a^{ \J
  },x^{ \K } \}, \{ d^{ \J },u^{ \K } \}$ are distinct arcs, then the
  vertices $( a+1 )^{\J}$, $\ldots$, $( d-1 )^{\J}, ( u+1 )^{\K}$,
  $\ldots$, $( x-1 )^{\K}$ are said to be {\em strictly between the
    (connecting) arcs $\{ a^{ \J },x^{ \K } \}, \{ d^{ \J },u^{ \K }
    \}$}, see the right half of Figure \ref{fig:PP}.

\end{enumerate}
\end{Definition}

\begin{Definition}
[Saturated vertices]
\label{def:saturated}
A vertex of intervals $\I$ or $\III$ is called {\em saturated} if it
is strictly below an internal arc $\{ a^{\I},d^{\I} \}$ or $\{
u^{\III},x^{\III} \}$ in $\Theta( t )$, or strictly between two
connecting arcs $\{ a^{ \I },x^{ \III } \}, \{ d^{ \I },u^{ \III } \}$
in $\Theta( t )$; see Definition \ref{def:below}.

A vertex of intervals $\I$ or $\III$ which is not saturated is called
{\em non-saturated}. 
\end{Definition}

\begin{Remark}
Suppose $\{ a^{ \I },d^{ \I } \} \in \Theta( t )$ and let $P$ denote
the finite polygon below the internal arc $\{ a^{ \I },d^{ \I } \}$.
If $a < b < d$ then
\[
  p( b-1,b+1 ) = 1 + (\mbox{the number of arcs in $\Theta( t )$ ending
  at $b^{ \I }$}).
\]
This follows from Lemmas \ref{lem:finite_triangulation1} and \ref{lem:CC}(vi).
A similar equality holds for $q$.  In general, there is no such equality, and this is captured by the {\em defects} introduced in the next definition.

In due course, the defects will be used to augment the partial
triangulation $\Theta( t )$ to a triangulation $\fT$ which satisfies
\[
  p( a,d ) = \fT( a^{ \I },d^{ \I } )
\]
for all $a \leqslant d$.  In turn, this will permit us to use Lemma
\ref{lem:agree} to prove 
\[
  t( b,v ) = \fT( b^{ \I },v^{ \III } )
\]
for all $b$, $v$, that is, to prove $t = \Phi( \fT )$.
\end{Remark}

\begin{Definition}
\label{def:defects}
The {\em $p$-defect} of an integer $b$ is
\[
  \defect_p( b )
  =
  p( b-1,b+1 )
  - 1
  - (\mbox{the number of arcs in $\Theta( t )$ ending at $b^{\I}$})
\]
and the {\em $q$-defect} of an integer $v$ is
\[
  \defect_q( v )
  =
  q( v-1,v+1 )
  - 1
  - (\mbox{the number of arcs in $\Theta( t )$ ending at $v^{\III}$}).
\]
\end{Definition}

\begin{Lemma}
\label{lem:internal_defect}
Suppose that $\Theta( t )$ has no connecting arcs which end at the
vertex $b_0^{\I}$.  We use the following notation, illustrated by Figure
\ref{fig:bik}, which makes sense because $\Theta( t )$ is locally
finite by Lemma \ref{lem:locally_finite}:
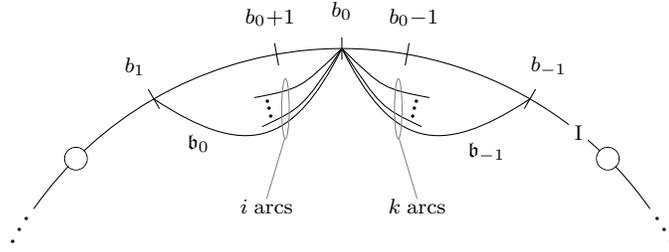
\begin{figure}
  \centering
    \begin{tikzpicture}[scale=5]
%      \draw[step=.25cm,gray,very thin] (-1.4,-1.4) grid (1.4,1.4);

      \draw (35:1cm) arc (35:145:1cm);
      \draw (29:1cm) node{$\cdot$};
      \draw (31:1cm) node{$\cdot$};
      \draw (33:1cm) node{$\cdot$};
      \draw (147:1cm) node{$\cdot$};
      \draw (149:1cm) node{$\cdot$};
      \draw (151:1cm) node{$\cdot$};

      \draw (45:1cm) node[fill=white,circle,inner sep=0.101cm] {} circle (0.03cm);
      \draw (135:1cm) node[fill=white,circle,inner sep=0.101cm] {} circle (0.03cm);

      \draw (60:0.97cm) -- (60:1.03cm);
      \draw (60:1.10cm) node{$\scriptstyle b_{-1}$};
      \draw (80:0.97cm) -- (80:1.03cm);
      \draw (80:1.10cm) node{$\scriptstyle b_0-1$};
      \draw (90:0.97cm) -- (90:1.03cm);
      \draw (90:1.10cm) node{$\scriptstyle b_0$};
      \draw (100:0.97cm) -- (100:1.03cm);
      \draw (100:1.10cm) node{$\scriptstyle b_0+1$};
      \draw (120:0.97cm) -- (120:1.03cm);
      \draw (120:1.10cm) node{$\scriptstyle b_1$};

      \draw (90:1cm) .. controls (80:0.75cm) and (70:0.75cm) .. (60:1cm);
      \draw (90:1cm) .. controls (83:0.85cm) .. (75:0.82cm);
      \draw (90:1cm) .. controls (83:0.9cm) .. (75:0.90cm);
      \draw (77:0.84) node{$\cdot$};
      \draw (77:0.86) node{$\cdot$};
      \draw (77:0.88) node{$\cdot$};

      \draw (90:1cm) .. controls (100:0.75cm) and (110:0.75cm) .. (120:1cm);
      \draw (90:1cm) .. controls (97:0.85cm) .. (105:0.82cm);
      \draw (90:1cm) .. controls (97:0.9cm) .. (105:0.90cm);
      \draw (103:0.84) node{$\cdot$};
      \draw (103:0.86) node{$\cdot$};
      \draw (103:0.88) node{$\cdot$};

      \draw (62:0.82cm) node{$\scriptstyle \fb_{-1}$};
      \draw (117:0.84cm) node{$\scriptstyle \fb_0$};

      \draw[gray,very thin] (0.15,0.84) ellipse (0.01cm and 0.08cm);
      \draw[gray,very thin] (0.152,0.761) -- (0.20,0.60);
      \draw (0.20,0.58) node{${\scriptstyle k\;\mathrm{arcs}}$};

      \draw[gray,very thin] (-0.15,0.84) ellipse (0.01cm and 0.08cm);
      \draw[gray,very thin] (-0.152,0.761) -- (-0.20,0.60);
      \draw (-0.20,0.58) node{${\scriptstyle i\;\mathrm{arcs}}$};

      \draw (51:1cm) node[fill=white,rectangle,inner sep=0.07cm] {$\scriptstyle \I$};

    \end{tikzpicture} 
    \caption{This sketch shows $\Theta( t )$ when it has no connecting
      arcs ending at $b_0^{\I}$.  There are $i$ internal arcs in
      $\Theta( t )$ which go anticlockwise from $b_0^{\I}$ and the
      longest is $\fb_0 = \{ b_0^{\I},b_1^{\I} \}$.  There are $k$
      internal arcs in $\Theta( t )$ which go clockwise from
      $b_0^{\I}$ and the longest is $\fb_{-1} = \{
      b_0^{\I},b_{-1}^{\I} \}$.}
\label{fig:bik}
\end{figure}

Let $\fb_0 = \{ b_0^{\I},b_1^{\I} \}$ be either the longest internal
arc in $\Theta( t )$ going anticlockwise from $b_0^{\I}$, or, if there
are no such arcs, the edge going anticlockwise from $b_0^{\I}$.

Let $\fb_{-1} = \{ b_{-1}^{\I},b_0^{\I} \}$ be either the
longest internal arc in $\Theta( t )$ going clockwise from $b_0^{\I}$,
or, if there are no such arcs, the edge going clockwise from
$b_0^{\I}$.

Then
\[
  \defect_p( b_0 ) = p( b_{-1},b_1 ) - 1.
\]
\end{Lemma}

\begin{proof}
The special cases where $\fb_0$ or $\fb_{-1}$ is an edge are
omitted since they are easy.  We assume that $\fb_0$ and $\fb_{-1}$
are arcs.

We let $i$, respectively $k$, denote the number of internal arcs in
$\Theta( t )$ going anticlockwise, respectively clockwise, from
$b_0^{\I}$, see Figure \ref{fig:bik}.

First, consider the finite polygon $P$ below $\fb_0$.  Lemma
\ref{lem:finite_triangulation1} says that $\Theta( t )$ restricts to a
triangulation $\Theta( t )_P$ of $P$ and that
\[
  p( b_0+1,b_1 ) = \Theta( t )_P\big( ( b_0+1 )^{\I},b_1^{\I} \big) = (*).
\]
Viewed in $P$, the vertices $( b_0+1 )^{\I}$, $b_0^{\I}$, $b_1^{\I}$
%yes, wrong = clockwise order, but that fits better with the previous
%equation 
are consecutive so Lemma \ref{lem:CC}(vi) gives
\[
  (*)
  = 1 + (\mbox{the number of arcs in $\Theta( t )_P$ ending at $b_0^{\I}$})
  = (**).
\]
The arcs in $\Theta( t )_P$ ending at $b_0^{\I}$ are precisely the
arcs in $\Theta( t )$ going anticlockwise from $b_0^{\I}$, except for
$\fb_0$ which is an edge of $P$.  Hence
\[
  (**) = i.
\]
This proves the first of the following equalities, and the second
follows by symmetry.
\begin{align}
\label{equ:connecting_p1b}
  p( b_0+1,b_1 ) & = i, \\
\label{equ:connecting_p2b}
  p( b_{-1},b_0-1 ) & = k.
\end{align}

Secondly, we show two consequences of the Ptolemy relations in
Lemma \ref{lem:Ptolemy}.
\begin{itemize}
\setlength\itemsep{4pt}

  \item Since $\fb_0 = \{ b_0^{\I},b_1^{\I} \}$ is in $\Theta( t )$, we have
\begin{equation}
\label{equ:connecting_p3}
  p( b_0,b_1 ) = 1.
\end{equation}
This gives the first equality in the following computation, 
\begin{align}
\nonumber
  p( b_0-1,b_0+1 )
  & = p( b_0-1,b_0+1 )p( b_0,b_1 ) \\
\nonumber
  & = p( b_0-1,b_0 )p( b_0+1,b_1 )
    + p( b_0-1,b_1 )p( b_0,b_0+1 ) \\
\label{equ:connecting_i}
  & = i + p( b_0-1,b_1 ),
\end{align}
where the second equality is by the Ptolemy relation in Lemma
\ref{lem:Ptolemy}(i) and the third equality uses Equation
\eqref{equ:connecting_p1b}.

  \item  Since $\fb_{-1} = \{ b_{-1}^{\I},b_0^{\I} \}$ is in
    $\Theta( t )$, we have $p( b_{-1},b_0 ) = 1$.  This
    gives the first equality in the following computation,
\begin{align}
\nonumber
  p( b_0-1,b_1 )
  & = p( b_{-1},b_0 )p( b_0-1,b_1 ) \\
\nonumber
  & = p( b_{-1},b_0-1 )p( b_0,b_1 )
    + p( b_{-1},b_1 )p( b_0-1,b_0 ) \\
\label{equ:connecting_v}
  & = k + p( b_{-1},b_1 ),
\end{align}
where the second equality is by the Ptolemy relation in Lemma
\ref{lem:Ptolemy}(i) and the third equality uses Equations
\eqref{equ:connecting_p2b} and \eqref{equ:connecting_p3}.

\end{itemize}

Finally, the previous equations combine as follows.
\[
  p( b_0-1,b_0+1 )
  \stackrel{ \eqref{equ:connecting_i} }{=}
  i + p( b_0-1,b_1 )
  \stackrel{ \eqref{equ:connecting_v} }{=}
  i + k + p( b_{-1},b_1 )
\]
Subtracting $i + k + 1$ from this equation turns the left hand
side into $\defect_p( b_0 )$ because there are a total of $i + k$
arcs in $\Theta ( t )$ which end at $b_0^{\I}$; see Figure
\ref{fig:bik} and Definition \ref{def:defects}.  This proves the
lemma. 
\end{proof}

\begin{Lemma}
\label{lem:connecting_defect}
Suppose that $\Theta( t )$ has at least one connecting arc which ends
at the vertex $b_0^{\I}$.  We use the following notation, illustrated by
Figure \ref{fig:bijk}, which makes sense because $\Theta( t )$ is
locally finite by Lemma \ref{lem:locally_finite}:
\begin{figure}
  \centering
    \begin{tikzpicture}[scale=5]
%      \draw[step=.25cm,gray,very thin] (-1.4,-1.4) grid (1.4,1.4);

      \draw (0,0) circle (1cm);

      \draw (45:1cm) node[fill=white,circle,inner sep=0.101cm] {} circle (0.03cm);
      \draw (135:1cm) node[fill=white,circle,inner sep=0.101cm] {} circle (0.03cm);
      \draw (225:1cm) node[fill=white,circle,inner sep=0.101cm] {} circle (0.03cm);
      \draw (315:1cm) node[fill=white,circle,inner sep=0.101cm] {} circle (0.03cm);

      \draw (60:0.97cm) -- (60:1.03cm);
      \draw (60:1.10cm) node{$\scriptstyle b_{-1}$};
      \draw (80:0.97cm) -- (80:1.03cm);
      \draw (80:1.10cm) node{$\scriptstyle b_0-1$};
      \draw (90:0.97cm) -- (90:1.03cm);
      \draw (90:1.10cm) node{$\scriptstyle b_0$};
      \draw (100:0.97cm) -- (100:1.03cm);
      \draw (100:1.10cm) node{$\scriptstyle b_0+1$};
      \draw (120:0.97cm) -- (120:1.03cm);
      \draw (120:1.10cm) node{$\scriptstyle b_1$};

      \draw (255:0.97cm) -- (255:1.03cm);
      \draw (255:1.10cm) node{$\scriptstyle v_1$};
      \draw (285:0.97cm) -- (285:1.03cm);
      \draw (285:1.10cm) node{$\scriptstyle v_j$};

      \draw (90:1cm) .. controls (80:0.75cm) and (70:0.75cm) .. (60:1cm);
      \draw (90:1cm) .. controls (83:0.85cm) .. (75:0.82cm);
      \draw (90:1cm) .. controls (83:0.9cm) .. (75:0.90cm);
      \draw (77:0.84) node{$\cdot$};
      \draw (77:0.86) node{$\cdot$};
      \draw (77:0.88) node{$\cdot$};

      \draw (90:1cm) .. controls (100:0.75cm) and (110:0.75cm) .. (120:1cm);
      \draw (90:1cm) .. controls (97:0.85cm) .. (105:0.82cm);
      \draw (90:1cm) .. controls (97:0.9cm) .. (105:0.90cm);
      \draw (103:0.84) node{$\cdot$};
      \draw (103:0.86) node{$\cdot$};
      \draw (103:0.88) node{$\cdot$};

      \draw (90:1cm) .. controls (-0.05,0.3) and (-0.05,-0.3) .. (255:1cm);
      \draw (90:1cm) .. controls (-0.01,0.5) .. (-0.04,0);
      \draw (90:1cm) .. controls (0.01,0.5) .. (0.04,0);
      \draw (90:1cm) .. controls (0.05,0.3) and (0.05,-0.3) .. (285:1cm);
      \draw (-0.025,0.05) node{$\cdot$};
      \draw (0,0.05) node{$\cdot$};
      \draw (0.025,0.05) node{$\cdot$};

      \draw (62:0.84cm) node{$\scriptstyle \fb_{-1}$};
      \draw (117:0.86cm) node{$\scriptstyle \fb_0$};

      \draw (-0.17,-0.4) node{$\scriptstyle \fa_1$};
      \draw (0.17,-0.4) node{$\scriptstyle \fa_j$};

      \draw[gray,very thin] (0.15,0.84) ellipse (0.01cm and 0.08cm);
      \draw[gray,very thin] (0.152,0.761) -- (0.20,0.60);
      \draw (0.22,0.58) node{${\scriptstyle k\;\mathrm{arcs}}$};

      \draw[gray,very thin] (-0.15,0.84) ellipse (0.01cm and 0.08cm);
      \draw[gray,very thin] (-0.152,0.761) -- (-0.20,0.60);
      \draw (-0.20,0.58) node{${\scriptstyle i\;\mathrm{arcs}}$};

      \draw[gray,very thin] (0,0.15) ellipse (0.08cm and 0.01cm);
      \draw[gray,very thin] (-0.077,0.147) -- (-0.26,0.05);
      \draw (-0.26,0.025) node{${\scriptstyle j\;\mathrm{arcs}}$};

        \draw (52.5:1cm) node[fill=white,rectangle,inner sep=0.07cm] {$\scriptstyle \I$};
        \draw (180:1cm) node[fill=white,rectangle,inner sep=0.07cm] {$\scriptstyle \II$};
        \draw (233.5:1cm) node[fill=white,rectangle,inner sep=0.07cm] {$\scriptstyle \III$};
        \draw (360:1cm) node[fill=white,rectangle,inner sep=0.07cm] {$\scriptstyle \IV$};

    \end{tikzpicture} 
    \caption{This sketch shows $\Theta( t )$ when it contains $j
      \geqslant 1$ connecting arcs $\fa_1$, $\ldots$, $\fa_j$ ending
      at $b_0^{\I}$.  There are $i$ internal arcs in $\Theta( t )$
      which go anticlockwise from $b_0^{\I}$ and the longest is $\fb_0
      = \{ b_0^{\I},b_1^{\I} \}$.  There are $k$ internal arcs in
      $\Theta( t )$ which go clockwise from $b_0^{\I}$ and the longest
      is $\fb_{-1} = \{ b_0^{\I},b_{-1}^{\I} \}$.}
\label{fig:bijk}
\end{figure}

Let $v_1 < \cdots < v_j$ be such that $\fa_1 = \{ b_0^{\I},v_1^{\III}
\}, \ldots, \fa_j = \{ b_0^{\I},v_j^{\III} \}$ are all the connecting
arcs in $\Theta( t )$ which end at $b_0^{\I}$.

Let $\fb_0 = \{ b_0^{\I},b_1^{\I} \}$ be either the longest internal
arc in $\Theta( t )$ going anticlockwise from $b_0^{\I}$, or, if there
are no such arcs, the edge going anticlockwise from $b_0^{\I}$.

Let $\fb_{-1} = \{ b_{-1}^{\I},b_0^{\I} \}$ be either the longest
internal arc in $\Theta( t )$ going clockwise from $b_0^{\I}$, or, if
there are no such arcs, the edge going clockwise from $b_0^{\I}$.

Then
\[
  \defect_p( b_0 ) = t( b_{ -1 },v_j ) + t( b_1,v_1 ) - 2.
\]
\end{Lemma}

\begin{proof}
The special cases where $\fb_0$ or $\fb_{-1}$ is an edge or where $j =
1$ are omitted since they are easy.  We assume that $\fb_0$ and
$\fb_{-1}$ are arcs and that $j \geqslant 2$.

We let $i$, respectively $k$, denote the number of internal arcs in
$\Theta( t )$ going anticlockwise, respectively clockwise, from
$b_0^{\I}$.  The choice of $v_1 < \cdots < v_j$ means that there are
$j$ connecting arcs in $\Theta( t )$ ending at $b_0^{ \I }$.  See
Figure \ref{fig:bijk}.

First, an argument like the one used to prove Equation
\eqref{equ:connecting_p1b} shows
\begin{equation}
\label{equ:connecting_q}
  q( v_1,v_j ) = j-1.
%\label{equ:connecting_p2}
%  & p( b_{-1},b_0-1 ) = k.
\end{equation}

Secondly, the same arguments as in the proof of Lemma
\ref{lem:internal_defect} show that Equations
\eqref{equ:connecting_p2b}, \eqref{equ:connecting_p3}, and
\eqref{equ:connecting_i} remain valid.  We collect three other
consequences of the Ptolemy relations from Lemma \ref{lem:Ptolemy}.
\begin{itemize}
\setlength\itemsep{4pt}

\item Since $\fa_1 = \{ b_0^{\I},v_1^{\III} \}$ is in $\Theta( t )$,
  we have $t( b_0,v_1 ) = 1$.  This gives the first equality in the
  following computation,
\begin{align}
\nonumber
  p( b_0-1,b_1 )
  & = p( b_0-1,b_1 )t( b_0,v_1 ) \\
\nonumber
  & = p( b_0-1,b_0 )t( b_1,v_1 ) + p( b_0,b_1 )t( b_0-1,v_1 ) \\
\label{equ:connecting_ii}
  & = t( b_1,v_1 ) + t( b_0-1,v_1 ),
\end{align}
where the second equality is by the Ptolemy relation in Lemma
\ref{lem:Ptolemy}(ii) and the third equality uses Equation
\eqref{equ:connecting_p3}.  

  \item Since $\fa_j = \{ b_0^{\I},v_j^{\III} \}$ is in $\Theta( t )$,
  we have $t( b_0,v_j ) = 1$.  This gives the first equality in the
  following computation,
\begin{align}
\nonumber
  t( b_0-1,v_1 )
  & = t( b_0-1,v_1 )t( b_0,v_j ) \\
\nonumber
  & = t( b_0-1,v_j )t( b_0,v_1 ) + p( b_0-1,b_0 )q( v_1,v_j ) \\
\label{equ:connecting_iii}
  & = t( b_0-1,v_j ) + j - 1,
\end{align}
where the second equality is by the Ptolemy relation in Lemma
\ref{lem:Ptolemy}(iii) and the third equality uses $t( b_0,v_1 ) = 1$ and
Equation \eqref{equ:connecting_q}.

  \item  Since $\fb_{-1} = \{ b_{-1}^{\I},b_0^{\I} \}$ is in
    $\Theta( t )$, we have $p( b_{-1},b_0 ) = 1$.  This
    gives the first equality in the following computation,
\begin{align}
\nonumber
  t( b_0-1,v_j )
  & = p( b_{-1},b_0 )t( b_0-1,v_j ) \\
\nonumber
  & = p( b_0-1,b_0 )t( b_{-1},v_j )
    + p( b_{-1},b_0-1 )t( b_0,v_j ) \\
\label{equ:connecting_iv}
  & = t( b_{-1},v_j ) + k,
\end{align}
where the second equality is by the Ptolemy relation in Lemma
\ref{lem:Ptolemy}(ii) and the third equality uses $t( b_0,v_j ) =1$ and
Equation \eqref{equ:connecting_p2b}.

\end{itemize}

Finally, the previous equations combine as follows.
\begin{align*}
  p( b_0-1,b_0+1 )
  & \stackrel{ \eqref{equ:connecting_i} }{=}
    i + p( b_0-1,b_1 ) \\
  & \stackrel{ \eqref{equ:connecting_ii} }{=}
    i + t( b_1,v_1 ) + t( b_0-1,v_1 ) \\
  & \stackrel{ \eqref{equ:connecting_iii} }{=}
    i + t( b_1,v_1 ) + t( b_0-1,v_j ) + j - 1 \\
  & \stackrel{ \eqref{equ:connecting_iv} }{=}
    i + t( b_1,v_1 ) + t( b_{-1},v_j ) + k + j - 1
\end{align*}
Subtracting $i + j + k + 1$ from this equation turns the left hand
side into $\defect_p( b_0 )$ because there are a total of $i + j + k$
arcs in $\Theta ( t )$ which end at $b_0^{\I}$; see Figure
\ref{fig:bijk} and Definition \ref{def:defects}.  This proves the
lemma.
\end{proof}

\begin{Lemma}
\label{lem:defect}
\begin{enumerate}
\setlength\itemsep{4pt}

  \item  $b^{\I}$ is saturated if and only if $\defect_p(
    b ) = 0$. 

  \item  $b^{\I}$ is non-saturated if and only if $\defect_p(
    b ) > 0$. 

  \item $v^{\III}$ is saturated if and only if $\defect_q(
    v ) = 0$.

  \item $v^{\III}$ is non-saturated if and only if $\defect_q(
    v ) > 0$.

\end{enumerate}
\end{Lemma}

\begin{proof}
To prove parts (i) and (ii), it is enough to prove ``only if'' in each part.

Part (i), ``only if'': Let $b^{\I}$ be a saturated
vertex, see Definition \ref{def:saturated}.  There are two cases.

The first case is that $b^{\I}$ is strictly below the internal arc $\{
a^{\I},d^{\I} \} \in \Theta( t )$, whence $a < b < d$.  Let $P$ denote
the finite polygon below $\{ a^{\I},d^{\I} \}$.  Lemma
\ref{lem:finite_triangulation1} says that the restriction $\Theta( t
)_P$ is a triangulation of $P$ and that
\[
  p( b-1,b+1 ) = \Theta( t )_P\big( ( b-1 )^{\I},( b+1 )^{\I} \big) = (*).
\]
Lemma \ref{lem:CC}(vi) gives
\[
  (*) = 1 + (\mbox{the number of arcs in $\Theta( t )_P$
                  ending at $b^{\I}$}) = (**).
\]
However, $\Theta( t )$ is a partial triangulation of $D_n$ so none of its arcs
can cross $\{ a^{\I},d^{\I} \}$.  It follows that
\[
  (**) = 1 + (\mbox{the number of arcs in $\Theta( t )$
                   ending at $b^{\I}$}).
\]
This shows $\defect_p( b ) = 0$.

The second case is that $b^{\I}$ is strictly between the connecting
arcs $\{ a^{ \I },x^{ \III } \}, \{ d^{ \I },u^{ \III } \} \in \Theta(
t )$.  This is handled similarly, replacing Lemma
\ref{lem:finite_triangulation1} by Lemma
\ref{lem:finite_triangulation2}.

Part (ii), ``only if'': Let $b^{\I}$ be a
non-saturated vertex.  There are two cases.

The first case is that there are no connecting arcs in $\Theta( t )$
which end at $b^{\I}$.  Set $b_0 = b$ and apply Lemma
\ref{lem:internal_defect}; in the notation of the lemma, $\Theta( t )$
at $b_0^{\I}$ looks like Figure \ref{fig:bik}.  The lemma gives
\[
  \defect_p( b ) = \defect_p( b_0 ) = p( b_{-1},b_1 ) - 1 = (\dagger).
\]
However, since $b^{\I}$ is non-saturated, the arc $\{
b_{-1}^{\I},b_1^{\I} \}$ cannot be in $\Theta( t )$.  It follows that
$p( b_{-1},b_1 ) > 1$ so $(\dagger) > 0$ as desired.

The second case is that there are connecting arcs in $\Theta( t )$
which end at $b^{\I}$.  Set $b_0 = b$ and apply Lemma
\ref{lem:connecting_defect}; in the notation of the lemma, $\Theta( t
)$ at $b_0^{\I}$ looks like Figure \ref{fig:bijk}.  The lemma gives
\[
  \defect_p( b ) = \defect_p( b_0 )
  = t( b_{-1},v_j ) + t( b_1,v_1 ) - 2 = (\ddagger).
\]
However, since $b^{\I}$ is non-saturated, it cannot be that both arcs
$\{ b_{-1}^{\I},v_j^{\III} \}$ and $\{ b_1^{\I},v_1^{\III} \}$ are in
$\Theta( t )$.  It follows that $t( b_{-1},v_j ) > 1$ or $t( b_1,v_1 )
> 1$ so $(\ddagger) > 0$ as desired.

(iii) and (iv) follow by symmetry.
\end{proof}

\begin{Lemma}
\label{lem:saturated_vertices}
Let $\fT$ be a good triangulation of $D_n$ such that $\Theta( t )
\subseteq \fT$.
\begin{enumerate}
\setlength\itemsep{4pt}

  \item  $b^{\I}$ is saturated 
         $\Rightarrow
          \fT\big( ( b-1 )^{\I},( b+1 )^{\I} \big) = p( b-1,b+1 )$.

  \item  $v^{\III}$ is saturated
         $\Rightarrow
          \fT\big( ( v-1 )^{\III},( v+1 )^{\III} \big) = q( b-1,b+1 )$.

\end{enumerate}
\end{Lemma}

\begin{proof}
(i)  Since $b^{\I}$ is saturated, Lemma \ref{lem:defect}(i) says
$\defect_p( b ) = 0$ which means 
\[
  p( b-1,b+1 ) = 1 + (\mbox{the number of arcs in $\Theta( t )$
                           ending at $b^{\I}$}) = (*)
\]
by Definition \ref{def:defects}.  There are two cases.

The first case is that $b^{\I}$ is strictly below an internal arc $\{
a^{\I},d^{\I} \} \in \Theta( t )$, whence $a < b < d$.  Let $P$ be the
finite polygon below $\{ a^{\I},d^{\I} \}$.  The restriction $\Theta(
t )_P$ is a triangulation of $P$ by Lemma
\ref{lem:finite_triangulation1}.  Since no arc in $\Theta( t )$
crosses $\{ a^{\I},d^{\I} \}$, we have
\[
  (*) = 1 + (\mbox{the number of arcs in $\Theta( t )_P$
                  ending at $b^{\I}$}) = (**).
\]
The inclusion $\Theta( t ) \subseteq \fT$ implies $\Theta( t )_P =
\fT_P$ so we get the first of the following equalities,
\begin{align*}
  (**) & = 1 + (\mbox{the number of arcs in $\fT_P$
                     ending at $b^{\I}$}) \\
       & = 1 + (\mbox{the number of arcs in $\fT$
                     ending at $b^{\I}$}) = (*\!*\!*),
\end{align*}
where the second equality holds since $a < b < d$ and since no arc in
$\fT$ crosses $\{ a^{\I},d^{\I} \} \in \fT$.  Finally,
\[
  (*\!*\!*) = \fT \big( ( b-1 )^{\I},( b+1 )^{\I} \big)
\]
by Lemma \ref{lem:CC}(vi).

The second case is that $b^{\I}$ is strictly between the connecting
arcs $\{ a^{ \I },x^{ \III } \}, \{ d^{ \I },u^{ \III } \} \in \Theta(
t )$.  This is handled similarly, replacing Lemma
\ref{lem:finite_triangulation1} by Lemma
\ref{lem:finite_triangulation2}.

(ii) follows by symmetry.
\end{proof}

\section{Case 1: $\SL2$-tilings with infinitely many entries equal to
  $1$ in each of the first and third quadrants}
\label{sec:Case1}

\begin{Theorem}
\label{thm:Case1}
Let $t$ be an $\SL2$-tiling with infinitely many entries equal to
$1$ in each of the first and third quadrants.

Consider $D_2$, the disc with two accumulation points and intervals
denoted $\I$ and $\III$, see Figure \ref{fig:C2}.

Then $\fT = \Theta( t )$ is a good triangulation of $D_2$ which
satisfies $\Phi( \fT ) = t$; see Figure \ref{fig:ft_in_Case1}.
\end{Theorem}

\begin{proof}
Since $t$ has infinitely many $1$'s in the first and the third
quadrant, Lemma \ref{lem:zig-zag} implies that there are integers
$\cdots < b_{-1} < b_0 < b_1 < \cdots$ and $\cdots < v_{-1} < v_0 <
v_1 < \cdots$ such that $t( b_m,v_{ -m } ) = 1$ for each integer $m$.
There are corresponding (connecting) arcs $\fa_m = \{ b_m^{ \I
},v_{-m}^{ \III } \}$ in $\fT = \Theta( t )$.  By Definition
\ref{def:restriction}, the arcs $\fa_m$ can be viewed as dividing
$D_2$ into finite polygons $R_m$, see Figure \ref{fig:ft_in_Case1}.
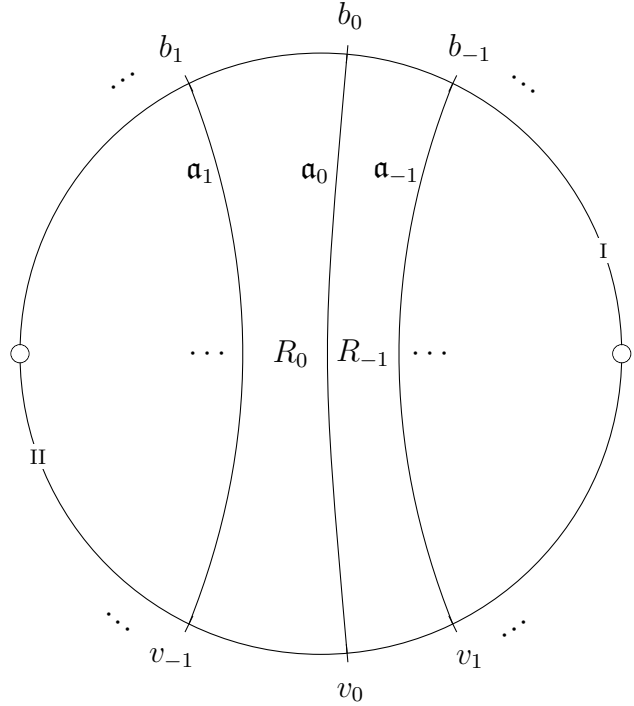
\begin{figure}
  \centering
    \begin{tikzpicture}[scale=4]
%      \draw[step=.25cm,gray,very thin] (-1.4,-1.4) grid (1.4,1.4);

      \draw (0,0) circle (1cm);

      \draw (0:1cm) node[fill=white,circle,inner sep=0.080cm] {} circle (0.03cm);
      \draw (180:1cm) node[fill=white,circle,inner sep=0.080cm] {} circle (0.03cm);

      \draw (51:1.12cm) node{$\cdot$};
      \draw (53:1.12cm) node{$\cdot$};
      \draw (55:1.12cm) node{$\cdot$};
      \draw (64:0.97cm) -- (64:1.03cm);
      \draw (64:1.13cm) node{$b_{-1}$};
      \draw (85:0.97cm) -- (85:1.03cm);
      \draw (85:1.13cm) node{$b_0$};
      \draw (116:0.97cm) -- (116:1.03cm);
      \draw (116:1.13cm) node{$b_1$};
      \draw (124:1.12cm) node{$\cdot$};
      \draw (126:1.12cm) node{$\cdot$};
      \draw (128:1.12cm) node{$\cdot$};

      \draw (-53:1.12cm) node{$\cdot$};
      \draw (-55:1.12cm) node{$\cdot$};
      \draw (-57:1.12cm) node{$\cdot$};
      \draw (-64:0.97cm) -- (-64:1.03cm);
      \draw (-64:1.13cm) node{$v_1$};
      \draw (-85:0.97cm) -- (-85:1.03cm);
      \draw (-85:1.13cm) node{$v_0$};
      \draw (-116:0.97cm) -- (-116:1.03cm);
      \draw (-116:1.13cm) node{$v_{-1}$};
      \draw (-125:1.12cm) node{$\cdot$};
      \draw (-127:1.12cm) node{$\cdot$};
      \draw (-129:1.12cm) node{$\cdot$};

      \draw (-64:1cm) .. controls (0.2,-0.3) and (0.2,0.3) .. (64:1cm);
      \draw (-85:1cm) .. controls (0,0) .. (85:1cm);
      \draw (-116:1cm) .. controls (-0.2,-0.3) and (-0.2,0.3) .. (116:1cm);

      \draw (-0.4,0.6) node{$\fa_1$};
      \draw (-0.02,0.6) node{$\fa_0$};
      \draw (0.25,0.6) node{$\fa_{-1}$};

      \draw (-0.1,0) node{$R_0$};
      \draw (0.14,0) node{$R_{-1}$};

      \draw (-0.37,0) node{$\cdots$};
      \draw (0.37,0) node{$\cdots$};

        \draw (20:1cm) node[fill=white,rectangle,inner sep=0.07cm] {$\scriptstyle \I$};
        \draw (200:1cm) node[fill=white,rectangle,inner sep=0.07cm] {$\scriptstyle \II$};

    \end{tikzpicture} 
    \caption{Outline of the triangulation $\fT = \Theta( t )$ of $D_2$
      corresponding to an $\SL2$-tiling $t$ with infinitely many
      entries equal to $1$ in the first and the third quadrant, giving
      the connecting arcs $\fa_m$.  Between the $\fa_m$ are the finite
      polygons $R_m$.}
\label{fig:ft_in_Case1}
\end{figure}
For each $m$, Lemma \ref{lem:finite_triangulation2} says that the arcs
in $\fT$ between the vertices of $R_m$ form a triangulation
$\fT_{R_m}$ of $R_m$.  It follows that the whole of $\fT$ is a
triangulation of $D_2$.  The triangulation is good because the arcs
$\fa_m$ block both accumulation points, see Definition
\ref{def:block}. 

To show $\Phi( \fT ) = t$ we use Lemma \ref{lem:agree}:

Lemma \ref{lem:zig-zag} implies that there are integers $e < f$ and $g
< h$ such that $t( e,h ) = t( f,g ) = t( f,h ) = 1$.  Hence the
arcs $\{ e^{\I},h^{\III} \}$, $\{ f^{\I},g^{\III} \}$, $\{
f^{\I},h^{\III} \}$ are in $\fT$ whence $\fT( e^{\I},h^{\III} )
= \fT( f^{\I},g^{\III} ) = \fT( f^{\I},h^{\III} ) = 1$.  This
verifies condition (iii)' in Lemma \ref{lem:agree}.

The arcs $\fa_m$ mean that each vertex in intervals $\I$ and $\III$ is
saturated.  If $b^{\I}$ and $v^{\III}$ are given, Lemma
\ref{lem:saturated_vertices} hence confirms conditions (i) and (ii) in
Lemma \ref{lem:agree}.
\end{proof}

\section{Case 2: $\SL2$-tilings with infinitely many entries equal to
  $1$ only in the first or the third quadrant}
\label{sec:Case2}

By symmetry, it is enough to let $t$ be an $\SL2$-tiling with
infinitely many entries equal to $1$ in the first quadrant, but not in
the third quadrant.

\begin{Remark}
Consider $D_3$, the disc with three accumulation points and
intervals denoted $\I$, $\II$, and $\III$ as in the left part of
Figure \ref{fig:C3}.  We will construct a good triangulation $\fT$ of
$D_3$ such that $\Phi( \fT ) = t$.  The overall structure of $\fT$ is
shown in Figure \ref{fig:ft_in_Case2}. 

Note that if $t$ had infinitely many ones in the third quadrant, but
not in the first quadrant, then we would instead use the disc with
three accumulation points and intervals denoted $\I$, $\III$, and
$\IV$ as in the right part of Figure \ref{fig:C3}.  The overall
structure of $\fT$ would be the mirror image in a vertical line of
Figure \ref{fig:ft_in_Case2}.
\end{Remark}

\begin{Description}
[The partial triangulation $\Theta( t )$]
The black arcs in Figure \ref{fig:ft_in_Case2} show the overall
structure of $\Theta( t )$ in $D_3$ which we now describe:
\begin{figure}
  \centering
    \begin{tikzpicture}[scale=5]
%      \draw[step=.25cm,gray,very thin] (-1.4,-1.4) grid (1.4,1.4);

      \draw (0,0) circle (1cm);

      \draw (0:1cm) node[fill=white,circle,inner sep=0.101cm] {} circle (0.03cm);
      \draw (120:1cm) node[fill=white,circle,inner sep=0.101cm] {} circle (0.03cm);
      \draw (240:1cm) node[fill=white,circle,inner sep=0.101cm] {} circle (0.03cm);

      \draw (28:1.12cm) node{$\cdot$};
      \draw (30:1.12cm) node{$\cdot$};
      \draw (32:1.12cm) node{$\cdot$};
      \draw (41:0.97cm) -- (41:1.03cm);
      \draw (41:1.13cm) node{$b_{-2}$};
      \draw (57:0.97cm) -- (57:1.03cm);
      \draw (57:1.13cm) node{$b_{-1}$};
      \draw (73:0.97cm) -- (73:1.03cm);
      \draw (73:1.13cm) node{$b_0$};
      \draw (86:0.97cm) -- (86:1.03cm);
      \draw (86:1.13cm) node{$b_1$};
      \draw (99:0.97cm) -- (99:1.03cm);
      \draw (99:1.13cm) node{$b_2$};
%      \draw (107:1.12cm) node{$\cdot$};
%      \draw (109:1.12cm) node{$\cdot$};
%      \draw (111:1.12cm) node{$\cdot$};

      \draw (-29:1.12cm) node{$\cdot$};
      \draw (-31:1.12cm) node{$\cdot$};
      \draw (-33:1.12cm) node{$\cdot$};
      \draw (-41:0.97cm) -- (-41:1.03cm);
      \draw (-41:1.13cm) node{$v_2$};
      \draw (-57:0.97cm) -- (-57:1.03cm);
      \draw (-57:1.13cm) node{$v_2$};
      \draw (-73:0.97cm) -- (-73:1.03cm);
      \draw (-73:1.13cm) node{$v_0$};
      \draw (-86:0.97cm) -- (-86:1.03cm);
      \draw (-86:1.13cm) node{$v_{-1}$};
      \draw (-99:0.97cm) -- (-99:1.03cm);
      \draw (-99:1.13cm) node{$v_{-2}$};
%      \draw (-107:1.12cm) node{$\cdot$};
%      \draw (-109:1.12cm) node{$\cdot$};
%      \draw (-111:1.12cm) node{$\cdot$};

%      \draw (143:1.12cm) node{$\cdot$};
%      \draw (145:1.12cm) node{$\cdot$};
%      \draw (147:1.12cm) node{$\cdot$};
      \draw (152:0.97cm) -- (152:1.03cm);
      \draw (165:0.97cm) -- (165:1.03cm);
      \draw (165:1.13cm) node{$\beta$};
      \draw (180:0.97cm) -- (180:1.03cm);
      \draw (180:1.13cm) node{$0$};
      \draw (195:0.97cm) -- (195:1.03cm);
      \draw (208:0.97cm) -- (208:1.03cm);
%      \draw (213:1.12cm) node{$\cdot$};
%      \draw (215:1.12cm) node{$\cdot$};
%      \draw (217:1.12cm) node{$\cdot$};

      \draw (-41:1cm) .. controls (0.6,-0.3) and (0.6,0.3) .. (41:1cm);
      \draw (-57:1cm) .. controls (0.4,-0.3) and (0.4,0.3) .. (57:1cm);
      \draw (-73:1cm) .. controls (0.15,-0.3) and (0.15,0.3) .. (73:1cm);

      \draw (73:1cm) .. controls (77:0.8cm) and (82:0.8cm) .. (86:1cm);
      \draw (86:1cm) .. controls (90:0.8cm) and (95:0.8cm) .. (99:1cm);

      \draw (102:0.9cm) node{$\cdot$};
      \draw (104.5:0.9cm) node{$\cdot$};
      \draw (107:0.9cm) node{$\cdot$};

      \draw (-73:1cm) .. controls (-77:0.8cm) and (-82:0.8cm) .. (-86:1cm);
      \draw (-86:1cm) .. controls (-90:0.8cm) and (-95:0.8cm) .. (-99:1cm);

      \draw (-102:0.9cm) node{$\cdot$};
      \draw (-104.5:0.9cm) node{$\cdot$};
      \draw (-107:0.9cm) node{$\cdot$};

      \draw[red] (73:1cm) .. controls (110:0.1cm) and (130:0.1cm) .. (180:1cm);
      \draw[red] (73:1cm) .. controls (110:0.25cm) and (130:0.2cm) .. (165:1cm);
      \draw[red] (86:1cm) .. controls (110:0.25cm) and (140:0.3cm) .. (165:1cm);
      \draw[red] (86:1cm) .. controls (110:0.3cm) and (140:0.4cm) .. (152:1cm);

      \draw[red] (170:0.9cm) node{$\cdot$};
      \draw[red] (172.5:0.9cm) node{$\cdot$};
      \draw[red] (175:0.9cm) node{$\cdot$};

      \draw[red] (156:0.9cm) node{$\cdot$};
      \draw[red] (158.5:0.9cm) node{$\cdot$};
      \draw[red] (161:0.9cm) node{$\cdot$};

      \draw[red] (149:0.9cm) node{$\cdot$};
      \draw[red] (146.5:0.9cm) node{$\cdot$};
      \draw[red] (144:0.9cm) node{$\cdot$};

      \draw[red] (-73:1cm) .. controls (-110:0.1cm) and (-130:0.1cm) .. (-180:1cm);
      \draw[red] (-73:1cm) .. controls (-110:0.25cm) and (-130:0.2cm) .. (-165:1cm);
      \draw[red] (-86:1cm) .. controls (-110:0.25cm) and (-140:0.3cm) .. (-165:1cm);
      \draw[red] (-86:1cm) .. controls (-110:0.3cm) and (-140:0.4cm) .. (-152:1cm);

      \draw[red] (-170:0.9cm) node{$\cdot$};
      \draw[red] (-172.5:0.9cm) node{$\cdot$};
      \draw[red] (-175:0.9cm) node{$\cdot$};

      \draw[red] (-156:0.9cm) node{$\cdot$};
      \draw[red] (-158.5:0.9cm) node{$\cdot$};
      \draw[red] (-161:0.9cm) node{$\cdot$};

      \draw[red] (-149:0.9cm) node{$\cdot$};
      \draw[red] (-146.5:0.9cm) node{$\cdot$};
      \draw[red] (-144:0.9cm) node{$\cdot$};

      \draw (0.145,0) node{$\scriptstyle \fa_0$};
      \draw (0.373,0) node{$\scriptstyle \fa_{-1}$};
      \draw (0.58,0) node{$\scriptstyle \fa_{-2}$};

      \draw (0.75,0) node{$\cdots$};

      \draw (0.13,0.81) node{$\scriptstyle \fb_0$};
      \draw (-0.08,0.83) node{$\scriptstyle \fb_1$};

      \draw (0.13,-0.8) node{$\scriptstyle \fv_0$};
      \draw (-0.06,-0.815) node{$\scriptstyle \fv_{-1}$};

        \draw (12:1cm) node[fill=white,rectangle,inner sep=0.07cm] {$\scriptstyle \I$};
        \draw (130:1cm) node[fill=white,rectangle,inner sep=0.07cm] {$\scriptstyle \II$};
        \draw (247.5:1cm) node[fill=white,rectangle,inner sep=0.07cm] {$\scriptstyle \III$};

    \end{tikzpicture} 
    \caption{Outline of the triangulation $\fT$ of $D_3$ corresponding
      to an $\SL2$-tiling $t$ with infinitely many entries equal to
      $1$ only in the first quadrant.  The arcs in $\Theta( t )$ are
      black.  We add red arcs from the non-saturated vertices to define
      $\fT$.  They connect each non-saturated vertex in intervals $\I$
      and $\III$ to a block of consecutive vertices in interval $\II$.
      The number of red arcs added at vertex $b^{\I}$ is $\defect_p( b
      )$, and the number of red arcs added at vertex $v^{\III}$ is
      $\defect_q( v )$.}
\label{fig:ft_in_Case2}
\end{figure}
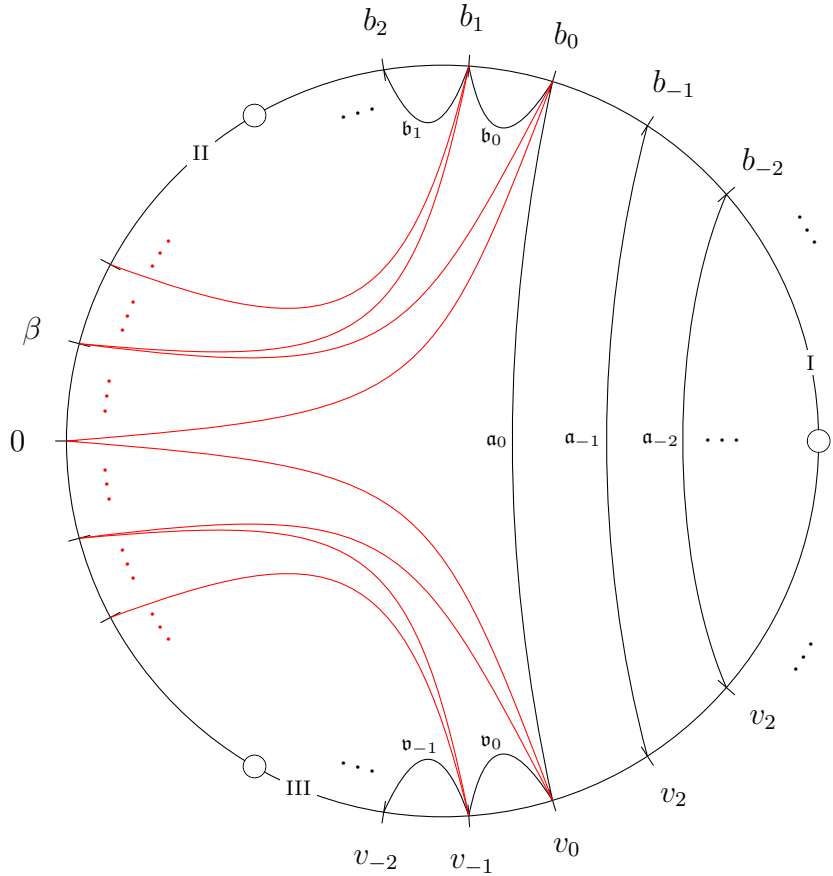

Using Lemma \ref{lem:zig-zag}, we can suppose that among the entries
in $t$ which are equal to $1$, the one which is furthest southwest is
$t( b_0,v_0 ) = 1$.  We can also choose integers $\cdots < b_{-1} <
b_0$ and $v_0 < v_1 < \cdots$ such that $t( b_m,v_{-m} ) = 1$ for $m
\leqslant 0$.  There are corresponding (connecting) arcs $\fa_m = \{
b_m^{\I},v_{-m}^{\III} \}$ in $\Theta( t )$.

The vertices $( b_0-1 )^{\I}$, $( b_0-2 )^{\I}$, $\ldots$ and $( v_0+1
)^{\III}$, $( v_0+2 )^{\III}$, $\ldots$ are saturated because of the
arcs $\fa_m$.  On the other hand, $b_0^{\I}$ is non-saturated: It is
not strictly between two connecting arcs in $\Theta( t )$, nor is it
strictly below an internal arc in $\Theta( t )$ because such an arc
would have to cross $\{ b_0^{\I},v_0^{\III} \} \in \Theta( t )$.

Let $\fb_0 = \{ b_0^{\I},b_1^{\I} \}$ be either the longest internal
arc in $\Theta( t )$ going anticlockwise from $b_0^{\I}$, or, if there
are no such arcs, the edge going anticlockwise from $b_0^{\I}$.  This
makes sense because $\Theta( t )$ is locally finite by Lemma
\ref{lem:locally_finite}.  It is easy to see that $b_1^{\I}$ is also
non-saturated, while the vertices strictly below $\fb_0$ are
saturated.

We can repeat this and thereby get integers $b_0 < b_1 < \cdots$ such
that the non-saturated vertices in interval $\I$ are precisely
$b_0^{\I}$, $b_1^{\I}$, $\ldots$.  A similar treatment provides
integers $\cdots < v_{-1} < v_0$ such that the non-saturated vertices
in interval $\III$ are precisely $\ldots$, $v_{-1}^{\III}$,
$v_0^{\III}$.
\end{Description}

\begin{Construction}
[The triangulation $\fT$]
\label{con:Case2}
We add arcs to $\Theta( t )$ as follows to create a triangulation
$\fT$ of $D_3$; see Figure \ref{fig:ft_in_Case2} where the added arcs
are shown in red:

From the non-saturated vertex $b_0^{\I}$, add $\defect_p(b_0)$ arcs
ending at the consecutive vertices $0^{\II}$, $-1^{\II}$, $\ldots$,
$\beta^{\II}$ in interval $\II$.  Note that $\defect_p(b_0) > 0$ by
Lemma \ref{lem:defect}(ii).  Then, from the non-saturated vertex
$b_1^{\I}$, add $\defect_p(b_1)$ arcs ending at the next block of
consecutive vertices $\beta^{\II}$, $( \beta-1 )^{\II}$, $\ldots$ in
interval $\II$.  Continue in the same way with the non-saturated
vertices $b_2^{\I}$, $b_3^{\I}$, $\ldots$.

Finally, add arcs by a similar recipe from the non-saturated vertices
$v_0^{\III}$, $v_{-1}^{\III}$, $\ldots$, using $\defect_q$ instead of
$\defect_p$.
\end{Construction}

\begin{Theorem}
\label{thm:Case2}
Let $t$ be an $\SL2$-tiling with infinitely many entries equal to $1$
only in the first or the third quadrant.

Then there is a good triangulation $\fT$ of $D_3$ such that $\Phi( \fT
) = t$.
\end{Theorem}

\begin{proof}
As remarked at the start of the section, it is enough by symmetry to
let $t$ be an $\SL2$-tiling with infinitely many entries equal to $1$
in the first quadrant, but not in the third quadrant.  Let $\fT$ be as
in Construction \ref{con:Case2}, see Figure \ref{fig:ft_in_Case2}.

Consider the finite polygons between the arcs $\fa_m$ and below the
arcs $\fb_{-m}$ and $\fv_m$ shown in Figure \ref{fig:ft_in_Case2}, see
Definition \ref{def:restriction}.  In each such polygon, $\Theta( t )$
and hence $\fT$ restricts to a triangulation by Lemmas
\ref{lem:finite_triangulation1} through
\ref{lem:finite_triangulation2}.  The arcs added at the end of
Construction \ref{con:Case2} (red in Figure \ref{fig:ft_in_Case2})
clearly complete $\fT$ to a triangulation of $D_3$.

The arcs $\fa_m$ block the accumulation point between intervals
$\I$ and $\III$, see Definition \ref{def:block}.  The arcs added at
the end of Construction \ref{con:Case2} block the accumulation points
between interval $\II$ and the other intervals.  Hence $\fT$ is a good
triangulation of $D_3$.

To show $\Phi( \fT ) = t$ we use Lemma \ref{lem:agree}:

Lemma \ref{lem:zig-zag} implies that there are integers $e < f$ and $g
< h$ so that $t( e,h ) = t( f,g ) = t( f,h ) = 1$.  Hence the arcs $\{
e^{\I},h^{\III} \}$, $\{ f^{\I},g^{\III} \}$, $\{ f^{\I},h^{\III} \}$
are in $\fT$ whence $\fT( e^{\I},h^{\III} ) = \fT( f^{\I},g^{\III} ) =
\fT( f^{\I},h^{\III} ) = 1$.  This verifies condition (iii)' in Lemma
\ref{lem:agree}.

To verify Lemma \ref{lem:agree}, condition (i), note that if $b$ is
given such that $b^{\I}$ is saturated, then the condition holds by
Lemma \ref{lem:saturated_vertices}(i).  If $b^{\I}$ is non-saturated,
then $b = b_m$ for some $m \geqslant 0$.  Definition \ref{def:defects}
gives
\[
  (\mbox{the number of arcs in $\Theta( t )$ ending at $b^{\I}$})
  =
  p( b-1,b+1 )
  -
  \defect_p( b )
  - 1.
\]
Compared to $\Theta( t )$, the triangulation $\fT$ has an additional
$\defect_p( b )$ arcs ending at $b^{\I}$ by Construction
\ref{con:Case2}, so
\[
  (\mbox{the number of arcs in $\fT$ ending at $b^{\I}$})
  =
  p( b-1,b+1 )  - 1.
\]
On the other hand, $\fT$ is a good triangulation of $D_3$ so
\[
  (\mbox{the number of arcs in $\fT$ ending at $b^{\I}$})
  = \fT\big( (b-1)^{\I},(b+1)^{\I} \big) - 1
\]
by Lemma \ref{lem:CC}(vi).  The last two equations imply $p( b-1,b+1 )
= \fT\big( (b-1)^{\I},(b+1)^{\I} \big)$, verifying Lemma
\ref{lem:agree}, condition (i).

Lemma \ref{lem:agree}, condition (ii) is verified by symmetry.
\end{proof}

\section{Case 3: $\SL2$-tilings with entries equal to $1$ only in a
  proper rectangle}
\label{sec:Case3}

Let $t$ be an $\SL2$-tiling with entries equal to $1$ only in a proper
rectangle.  We do not permit all the entries equal to $1$ to occur in
a single row or a single column; these cases are handled separately in
Section \ref{sec:Case4}.

\begin{Construction}
[The triangulation $\fT$]
\label{con:Case3}
To construct $\fT$, proceed similarly to Construction \ref{con:Case2}.

The difference is that there are now only finitely many connecting
arcs between intervals $\I$ and $\III$.  There will consequently be
non-saturated vertices at {\em both} ends of each of intervals $\I$
and $\III$, so to go from $\Theta( t )$ to a triangulation $\fT$ we
will need {\em two} intervals in addition to $\I$ and $\III$.  Hence
$\fT$ will be a triangulation of $D_4$.

This is shown in Figure \ref{fig:ft_in_Case3} where the black arcs
show the overall structure of $\Theta( t )$ and the red arcs are added
to $\Theta( t )$ in order to obtain $\fT$.  At the vertex $b^{\I}$ we
add $\defect_p( b )$ red arcs.  At the vertex $v^{\III}$ we add
$\defect_q( v )$ red arcs.
\begin{figure}
  \centering
    \begin{tikzpicture}[scale=5]
%      \draw[step=.25cm,gray,very thin] (-1.4,-1.4) grid (1.4,1.4);

      \draw (0,0) circle (1cm);

      \draw (45:1cm) node[fill=white,circle,inner sep=0.101cm] {} circle (0.03cm);
      \draw (135:1cm) node[fill=white,circle,inner sep=0.101cm] {} circle (0.03cm);
      \draw (225:1cm) node[fill=white,circle,inner sep=0.101cm] {} circle (0.03cm);
      \draw (315:1cm) node[fill=white,circle,inner sep=0.101cm] {} circle (0.03cm);

      \draw (100:0.97cm) -- (100:1.03cm);
      \draw (112:0.97cm) -- (112:1.03cm);
      \draw (124:0.97cm) -- (124:1.03cm);
      \draw (80:0.97cm) -- (80:1.03cm);
      \draw (68:0.97cm) -- (68:1.03cm);
      \draw (56:0.97cm) -- (56:1.03cm);

      \draw (-100:0.97cm) -- (-100:1.03cm);
      \draw (-112:0.97cm) -- (-112:1.03cm);
      \draw (-124:0.97cm) -- (-124:1.03cm);
      \draw (280:0.97cm) -- (280:1.03cm);
      \draw (292:0.97cm) -- (292:1.03cm);
      \draw (304:0.97cm) -- (304:1.03cm);

      \draw (152:0.97cm) -- (152:1.03cm);
      \draw (165:0.97cm) -- (165:1.03cm);
      \draw (180:0.97cm) -- (180:1.03cm);
      \draw (180:1.13cm) node{$0$};
      \draw (195:0.97cm) -- (195:1.03cm);
      \draw (208:0.97cm) -- (208:1.03cm);

      \draw (28:0.97cm) -- (28:1.03cm);
      \draw (15:0.97cm) -- (15:1.03cm);
      \draw (0:0.97cm) -- (0:1.03cm);
      \draw (0:1.13cm) node{$0$};
      \draw (-15:0.97cm) -- (-15:1.03cm);
      \draw (-28:0.97cm) -- (-28:1.03cm);

      \draw (-100:1cm) .. controls (-0.075,-0.3) and (-0.075,0.3) .. (100:1cm);
      \draw (280:1cm) .. controls (0.075,-0.3) and (0.075,0.3) .. (80:1cm);

      \draw (100:1cm) .. controls (104:0.8cm) and (108:0.8cm) .. (112:1cm);
      \draw (112:1cm) .. controls (116:0.8cm) and (120:0.8cm) .. (124:1cm);
      \draw (80:1cm) .. controls (76:0.8cm) and (72:0.8cm) .. (68:1cm);
      \draw (68:1cm) .. controls (64:0.8cm) and (60:0.8cm) .. (56:1cm);

      \draw (54:0.9cm) node{$\cdot$};
      \draw (52:0.9cm) node{$\cdot$};
      \draw (50:0.9cm) node{$\cdot$};

      \draw (126:0.9cm) node{$\cdot$};
      \draw (128:0.9cm) node{$\cdot$};
      \draw (130:0.9cm) node{$\cdot$};

      \draw (-100:1cm) .. controls (-104:0.8cm) and (-108:0.8cm) .. (-112:1cm);
      \draw (-112:1cm) .. controls (-116:0.8cm) and (-120:0.8cm) .. (-124:1cm);
      \draw (280:1cm) .. controls (284:0.8cm) and (288:0.8cm) .. (292:1cm);
      \draw (292:1cm) .. controls (296:0.8cm) and (300:0.8cm) .. (304:1cm);

      \draw (306:0.9cm) node{$\cdot$};
      \draw (308:0.9cm) node{$\cdot$};
      \draw (310:0.9cm) node{$\cdot$};

      \draw (-126:0.9cm) node{$\cdot$};
      \draw (-128:0.9cm) node{$\cdot$};
      \draw (-130:0.9cm) node{$\cdot$};

      \draw[red] (100:1cm) .. controls (130:0.15cm) and (150:0.15cm) .. (180:1cm);
      \draw[red] (100:1cm) .. controls (130:0.25cm) and (150:0.2cm) .. (165:1cm);
      \draw[red] (112:1cm) .. controls (130:0.3cm) and (150:0.3cm) .. (165:1cm);
      \draw[red] (112:1cm) .. controls (130:0.3cm) and (150:0.4cm) .. (152:1cm);

      \draw[red] (170:0.9cm) node{$\cdot$};
      \draw[red] (172.5:0.9cm) node{$\cdot$};
      \draw[red] (175:0.9cm) node{$\cdot$};

      \draw[red] (156:0.9cm) node{$\cdot$};
      \draw[red] (158.5:0.9cm) node{$\cdot$};
      \draw[red] (161:0.9cm) node{$\cdot$};

      \draw[red] (149:0.9cm) node{$\cdot$};
      \draw[red] (146.5:0.9cm) node{$\cdot$};
      \draw[red] (144:0.9cm) node{$\cdot$};

      \draw[red] (-100:1cm) .. controls (-130:0.15cm) and (-150:0.15cm) .. (-180:1cm);
      \draw[red] (-100:1cm) .. controls (-130:0.25cm) and (-150:0.2cm) .. (-165:1cm);
      \draw[red] (-112:1cm) .. controls (-130:0.3cm) and (-150:0.3cm) .. (-165:1cm);
      \draw[red] (-112:1cm) .. controls (-130:0.3cm) and (-150:0.4cm) .. (-152:1cm);

      \draw[red] (-170:0.9cm) node{$\cdot$};
      \draw[red] (-172.5:0.9cm) node{$\cdot$};
      \draw[red] (-175:0.9cm) node{$\cdot$};

      \draw[red] (-156:0.9cm) node{$\cdot$};
      \draw[red] (-158.5:0.9cm) node{$\cdot$};
      \draw[red] (-161:0.9cm) node{$\cdot$};

      \draw[red] (-149:0.9cm) node{$\cdot$};
      \draw[red] (-146.5:0.9cm) node{$\cdot$};
      \draw[red] (-144:0.9cm) node{$\cdot$};

      \draw[red] (80:1cm) .. controls (50:0.15cm) and (30:0.15cm) .. (0:1cm);
      \draw[red] (80:1cm) .. controls (50:0.25cm) and (30:0.2cm) .. (15:1cm);
      \draw[red] (68:1cm) .. controls (50:0.3cm) and (30:0.3cm) .. (15:1cm);
      \draw[red] (68:1cm) .. controls (50:0.3cm) and (30:0.4cm) .. (28:1cm);

      \draw[red] (10:0.9cm) node{$\cdot$};
      \draw[red] (7.5:0.9cm) node{$\cdot$};
      \draw[red] (5:0.9cm) node{$\cdot$};

      \draw[red] (24:0.9cm) node{$\cdot$};
      \draw[red] (21.5:0.9cm) node{$\cdot$};
      \draw[red] (19:0.9cm) node{$\cdot$};

      \draw[red] (31:0.9cm) node{$\cdot$};
      \draw[red] (33.5:0.9cm) node{$\cdot$};
      \draw[red] (36:0.9cm) node{$\cdot$};

      \draw[red] (-80:1cm) .. controls (-50:0.15cm) and (-30:0.15cm) .. (0:1cm);
      \draw[red] (-80:1cm) .. controls (-50:0.25cm) and (-30:0.2cm) .. (-15:1cm);
      \draw[red] (-68:1cm) .. controls (-50:0.3cm) and (-30:0.3cm) .. (-15:1cm);
      \draw[red] (-68:1cm) .. controls (-50:0.3cm) and (-30:0.4cm) .. (-28:1cm);

      \draw[red] (-10:0.9cm) node{$\cdot$};
      \draw[red] (-7.5:0.9cm) node{$\cdot$};
      \draw[red] (-5:0.9cm) node{$\cdot$};

      \draw[red] (-24:0.9cm) node{$\cdot$};
      \draw[red] (-21.5:0.9cm) node{$\cdot$};
      \draw[red] (-19:0.9cm) node{$\cdot$};

      \draw[red] (-31:0.9cm) node{$\cdot$};
      \draw[red] (-33.5:0.9cm) node{$\cdot$};
      \draw[red] (-36:0.9cm) node{$\cdot$};

      \draw (50.5:1cm) node[fill=white,rectangle,inner sep=0.07cm] {$\scriptstyle \I$};
      \draw (141:1cm) node[fill=white,rectangle,inner sep=0.07cm] {$\scriptstyle \II$};
      \draw (231:1cm) node[fill=white,rectangle,inner sep=0.07cm] {$\scriptstyle \III$};
      \draw (323:1cm) node[fill=white,rectangle,inner sep=0.07cm] {$\scriptstyle \IV$};

    \end{tikzpicture} 
    \caption{Outline of the triangulation $\fT$ of $D_4$ corresponding
      to an $\SL2$-tiling $t$ with entries equal to $1$ only in a
      proper rectangle.  The arcs in $\Theta( t )$ are black.  We add
      red arcs from the non-saturated vertices to define $\fT$.  They
      connect each non-saturated vertex in intervals $\I$ and $\III$
      to a block of consecutive vertices in interval $\II$ or $\IV$.
      The number of red arcs added at the vertex $b^{\I}$ is
      $\defect_p( b )$, and the number of red arcs added at the vertex
      $v^{\III}$ is $\defect_q( v )$.}
\label{fig:ft_in_Case3}
\end{figure}
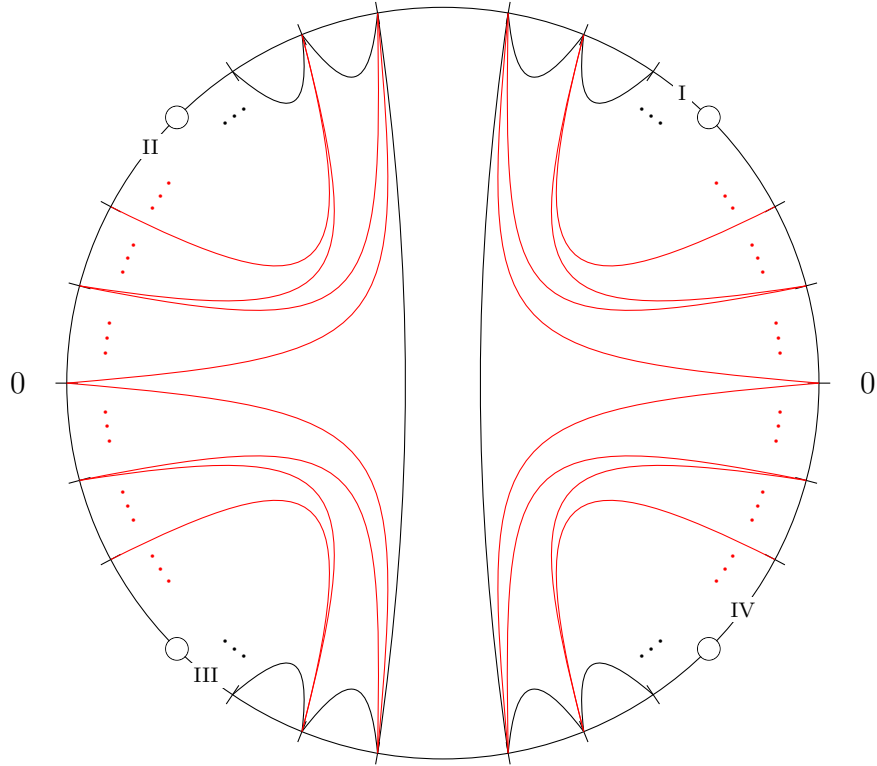
\end{Construction}

\begin{Theorem}
\label{thm:Case3}
Let $t$ be an $\SL2$-tiling with entries equal to $1$ only in a proper
rectangle.  We do not permit all the entries equal to $1$ to occur in
a single row or a single column.

Then there is a good triangulation $\fT$ of $D_4$ such that $\Phi( \fT
) = t$.
\end{Theorem}

\begin{proof}
Using $\fT$ from Construction \ref{con:Case3} (see Figure
\ref{fig:ft_in_Case3}) the proof is similar to the proof of Theorem
\ref{thm:Case2}. 
\end{proof}

\section{Case 4: $\SL2$-tilings with entries equal to $1$ only in a
  single row or column}
\label{sec:Case4}

By symmetry, it is enough to let $t$ be an $\SL2$-tiling with entries
equal to $1$ only in a single row.  We do not permit $t$ to have fewer
than two entries equal to $1$.  The case of a unique entry equal to
$1$ is handled in Section \ref{sec:Case5}.  The case of no entries
equal to $1$ is handled in Section \ref{sec:Case6}.

By Lemma \ref{lem:finiteness_in_rows_and_colums}, only finitely many
entries of $t$ are equal to $1$.  Let them be $t( b_0,v_1 ) = \cdots =
t( b_0,v_j ) = 1$ with $j \geqslant 2$.

\begin{Description}
[The partial triangulation $\Theta( t )$]
\label{des:Case4}
The black arcs in Figure \ref{fig:ft_in_Case4} show the overall
structure of $\Theta( t )$ which we now describe:

The only connecting arcs in $\Theta( t )$ are $\fa_{ \ell } = \{ b_0^{
  \I },v_{ \ell }^{ \III } \}$ for $\ell \in \{ 1,\ldots,j \}$.  The
vertices $( v_1+1 )^{\III}$, $\ldots$, $( v_j-1 )^{\III}$ strictly
between $\fa_1$ and $\fa_j$ are saturated, see Definition
\ref{def:saturated}.

On the other hand, $b_0^{\I}$ is non-saturated: It is not strictly
between two connecting arcs in $\Theta( t )$, nor is it strictly below
an internal arc in $\Theta( t )$ because such an arc would have to
cross $\{ b_0^{\I},v_1^{\III} \} \in \Theta( t )$.  Similarly,
$v_1^{\III}$ and $v_j^{\III}$ are non-saturated.

Let $\fb_0 = \{ b_0^{\I},b_1^{\I} \}$ be either the longest internal
arc in $\Theta( t )$ going anticlockwise from $b_0^{\I}$, or, if there
are no such arcs, the edge going anticlockwise from $b_0^{\I}$.  This
makes sense because $\Theta( t )$ is locally finite by Lemma
\ref{lem:locally_finite}.  It is easy to see that $b_1^{\I}$ is also
non-saturated, while the vertices strictly below $\fb_0$ are
saturated.

We can repeat this to both sides of $b_0$ and thereby get integers
$\cdots < b_{-1} < b_0 < b_1 < \cdots$ such that the non-saturated
vertices in interval $\I$ are precisely the $b_{\ell}^{\I}$.

A similar treatment provides integers $\cdots < v_{-1} < v_0 < v_1$ and $v_j
< v_{ j+1 } < v_{ j+2 } \cdots$ such that the non-saturated vertices in interval
$\III$ are precisely $\ldots$, $v_{-1}^{\III}$, $v_0^{\III}$, $v_1^{\III}$ and
$v_j^{\III}$, $v_{ j+1 }^{\III}$, $v_{ j+2 }^{\III}$, $\ldots$.
\end{Description}

\begin{Lemma}
\label{lem:Case4}
With the notation of Description \ref{des:Case4} and Figure
\ref{fig:ft_in_Case4}, we have
\[
  \defect_p( b_0 ) - \big( t( b_{-1},v_j ) - 1 \big) > 0.
\]
\end{Lemma}

\begin{proof}
We have
\[
  \defect_p( b_0 ) - \big( t( b_{-1},v_j ) - 1 \big)
  = t( b_1,v_1 ) - 1
  > 0
\]
where the equality is by Lemma \ref{lem:connecting_defect} and the
inequality is because $t( b_1,v_1 ) \neq 1$ by assumption.
\end{proof}

\begin{Construction}
[The triangulation $\fT$]
\label{con:Case4}
We add arcs to $\Theta( t )$ as follows to create a triangulation
$\fT$ of $D_4$; see Figure \ref{fig:ft_in_Case4} where the added arcs
are shown in red:

From vertex $b_0^{\I}$, add $t( b_{-1},v_j ) - 1$ arcs ending at the
consecutive vertices $0^{\IV}$, $1^{\IV}$, $\ldots$, $\varphi^{ \IV
}$.  Note that $t( b_{-1},v_j ) - 1 > 0$ since $b_{-1} \neq b_0$.

From vertex $b_{-1}^{\I}$, add $\defect_p( b_{-1} )$ arcs ending at
the next block of consecutive vertices in interval $\IV$.  Note that
$\defect_p( b_{-1} ) > 0$ by Lemma \ref{lem:defect}(ii).  Continue in
the same fashion with vertices $b_{-2}^{\I}$, $b_{-3}^{\I}$, $\ldots$.

Going back to vertex $b_0^{\I}$, add $\defect_p( b_0 ) - \big( t(
b_{-1},v_j ) - 1 \big)$ arcs ending at the consecutive vertices
$0^{\II}$, $-1^{\II}$, $\ldots$, $\beta^{\II}$.  This makes sense
because $\defect_p( b_0 ) - \big( t( b_{-1},v_j ) - 1 \big) > 0$ by
Lemma \ref{lem:Case4}.

From vertex $b_1^{\I}$, add $\defect_p( b_1 )$ arcs ending at the next
block of consecutive vertices in interval $\II$.  Continue in the same
fashion with vertices $b_2^{\I}$, $b_3^{\I}$, $\ldots$.

Finally, add arcs by a similar recipe from vertices $\ldots$,
$v_{-1}^{\III}$, $v_0^{\III}$, $v_1^{\III}$ and $v_j^{\III}$, $v_{ j+1
}^{\III}$, $v_{ j+2 }^{\III}$, $\ldots$ using $\defect_q$ instead of
$\defect_p$.
\begin{figure}
  \centering
    \begin{tikzpicture}[scale=5]
%      \draw[step=.25cm,gray,very thin] (-1.4,-1.4) grid (1.4,1.4);

      \draw (0,0) circle (1cm);

      \draw (45:1cm) node[fill=white,circle,inner sep=0.101cm] {} circle (0.03cm);
      \draw (135:1cm) node[fill=white,circle,inner sep=0.101cm] {} circle (0.03cm);
      \draw (225:1cm) node[fill=white,circle,inner sep=0.101cm] {} circle (0.03cm);
      \draw (315:1cm) node[fill=white,circle,inner sep=0.101cm] {} circle (0.03cm);

      \draw (66:0.97cm) -- (66:1.03cm);
      \draw (66:1.13cm) node{$b_{-2}$};
      \draw (78:0.97cm) -- (78:1.03cm);
      \draw (78:1.13cm) node{$b_{-1}$};
      \draw (90:0.97cm) -- (90:1.03cm);
      \draw (90:1.13cm) node{$b_0$};
      \draw (102:0.97cm) -- (102:1.03cm);
      \draw (102:1.13cm) node{$b_1$};
      \draw (114:0.97cm) -- (114:1.03cm);
      \draw (114:1.13cm) node{$b_2$};

      \draw (-100:0.97cm) -- (-100:1.03cm);
      \draw (-100:1.13cm) node{$v_1$};
      \draw (-112:0.97cm) -- (-112:1.03cm);
      \draw (-112:1.13cm) node{$v_{-1}$};
      \draw (-124:0.97cm) -- (-124:1.03cm);
      \draw (-124:1.13cm) node{$v_{-2}$};
      \draw (280:0.97cm) -- (280:1.03cm);
      \draw (280:1.13cm) node{$v_j$};
      \draw (292:0.97cm) -- (292:1.03cm);
      \draw (292:1.13cm) node{$v_{j+1}$};
      \draw (304:0.97cm) -- (304:1.03cm);
      \draw (304:1.13cm) node{$v_{j+2}$};

      \draw (152:0.97cm) -- (152:1.03cm);
      \draw (165:0.97cm) -- (165:1.03cm);
      \draw (165:1.13cm) node{$\beta$};
      \draw (180:0.97cm) -- (180:1.03cm);
      \draw (180:1.13cm) node{$0$};
      \draw (195:0.97cm) -- (195:1.03cm);
      \draw (208:0.97cm) -- (208:1.03cm);

      \draw (28:0.97cm) -- (28:1.03cm);
      \draw (15:0.97cm) -- (15:1.03cm);
      \draw (15:1.13cm) node{$\varphi$};
      \draw (0:0.97cm) -- (0:1.03cm);
      \draw (0:1.13cm) node{$0$};
      \draw (-15:0.97cm) -- (-15:1.03cm);
      \draw (-28:0.97cm) -- (-28:1.03cm);

      \draw (-100:1cm) .. controls (-0.02,-0.3) and (-0.02,0.3) .. (90:1cm);
      \draw (280:1cm) .. controls (0.02,-0.3) and (0.02,0.3) .. (90:1cm);

      \draw (90:1cm) .. controls (94:0.8cm) and (98:0.8cm) .. (102:1cm);
      \draw (102:1cm) .. controls (106:0.8cm) and (110:0.8cm) .. (114:1cm);
      \draw (90:1cm) .. controls (86:0.8cm) and (82:0.8cm) .. (78:1cm);
      \draw (78:1cm) .. controls (74:0.8cm) and (70:0.8cm) .. (66:1cm);

      \draw (62:0.9cm) node{$\cdot$};
      \draw (60:0.9cm) node{$\cdot$};
      \draw (58:0.9cm) node{$\cdot$};

      \draw (118:0.9cm) node{$\cdot$};
      \draw (120:0.9cm) node{$\cdot$};
      \draw (122:0.9cm) node{$\cdot$};

      \draw (-100:1cm) .. controls (-104:0.8cm) and (-108:0.8cm) .. (-112:1cm);
      \draw (-112:1cm) .. controls (-116:0.8cm) and (-120:0.8cm) .. (-124:1cm);
      \draw (280:1cm) .. controls (284:0.8cm) and (288:0.8cm) .. (292:1cm);
      \draw (292:1cm) .. controls (296:0.8cm) and (300:0.8cm) .. (304:1cm);

      \draw (306:0.9cm) node{$\cdot$};
      \draw (308:0.9cm) node{$\cdot$};
      \draw (310:0.9cm) node{$\cdot$};

      \draw (-126:0.9cm) node{$\cdot$};
      \draw (-128:0.9cm) node{$\cdot$};
      \draw (-130:0.9cm) node{$\cdot$};

      \draw[red] (90:1cm) .. controls (130:0.15cm) and (150:0.15cm) .. (180:1cm);
      \draw[red] (90:1cm) .. controls (130:0.25cm) and (150:0.2cm) .. (165:1cm);
      \draw[red] (102:1cm) .. controls (130:0.3cm) and (150:0.3cm) .. (165:1cm);
      \draw[red] (102:1cm) .. controls (130:0.3cm) and (150:0.4cm) .. (152:1cm);

      \draw[red] (170:0.9cm) node{$\cdot$};
      \draw[red] (172.5:0.9cm) node{$\cdot$};
      \draw[red] (175:0.9cm) node{$\cdot$};

      \draw[red] (156:0.9cm) node{$\cdot$};
      \draw[red] (158.5:0.9cm) node{$\cdot$};
      \draw[red] (161:0.9cm) node{$\cdot$};

      \draw[red] (149:0.9cm) node{$\cdot$};
      \draw[red] (146.5:0.9cm) node{$\cdot$};
      \draw[red] (144:0.9cm) node{$\cdot$};

      \draw[red] (-100:1cm) .. controls (-131:0.15cm) and (-150:0.15cm) .. (-180:1cm);
      \draw[red] (-100:1cm) .. controls (-130:0.25cm) and (-150:0.2cm) .. (-165:1cm);
      \draw[red] (-112:1cm) .. controls (-130:0.3cm) and (-150:0.3cm) .. (-165:1cm);
      \draw[red] (-112:1cm) .. controls (-130:0.3cm) and (-150:0.4cm) .. (-152:1cm);

      \draw[red] (-170:0.9cm) node{$\cdot$};
      \draw[red] (-172.5:0.9cm) node{$\cdot$};
      \draw[red] (-175:0.9cm) node{$\cdot$};

      \draw[red] (-156:0.9cm) node{$\cdot$};
      \draw[red] (-158.5:0.9cm) node{$\cdot$};
      \draw[red] (-161:0.9cm) node{$\cdot$};

      \draw[red] (-149:0.9cm) node{$\cdot$};
      \draw[red] (-146.5:0.9cm) node{$\cdot$};
      \draw[red] (-144:0.9cm) node{$\cdot$};

      \draw[red] (90:1cm) .. controls (50:0.15cm) and (30:0.15cm) .. (0:1cm);
      \draw[red] (90:1cm) .. controls (50:0.25cm) and (30:0.2cm) .. (15:1cm);
      \draw[red] (78:1cm) .. controls (50:0.3cm) and (30:0.3cm) .. (15:1cm);
      \draw[red] (78:1cm) .. controls (50:0.3cm) and (30:0.4cm) .. (28:1cm);

      \draw[red] (10:0.9cm) node{$\cdot$};
      \draw[red] (7.5:0.9cm) node{$\cdot$};
      \draw[red] (5:0.9cm) node{$\cdot$};

      \draw[red] (24:0.9cm) node{$\cdot$};
      \draw[red] (21.5:0.9cm) node{$\cdot$};
      \draw[red] (19:0.9cm) node{$\cdot$};

      \draw[red] (31:0.9cm) node{$\cdot$};
      \draw[red] (33.5:0.9cm) node{$\cdot$};
      \draw[red] (36:0.9cm) node{$\cdot$};

      \draw[red] (-80:1cm) .. controls (-50:0.15cm) and (-30:0.15cm) .. (0:1cm);
      \draw[red] (-80:1cm) .. controls (-50:0.25cm) and (-30:0.2cm) .. (-15:1cm);
      \draw[red] (-68:1cm) .. controls (-50:0.3cm) and (-30:0.3cm) .. (-15:1cm);
      \draw[red] (-68:1cm) .. controls (-50:0.3cm) and (-30:0.4cm) .. (-28:1cm);

      \draw[red] (-10:0.9cm) node{$\cdot$};
      \draw[red] (-7.5:0.9cm) node{$\cdot$};
      \draw[red] (-5:0.9cm) node{$\cdot$};

      \draw[red] (-24:0.9cm) node{$\cdot$};
      \draw[red] (-21.5:0.9cm) node{$\cdot$};
      \draw[red] (-19:0.9cm) node{$\cdot$};

      \draw[red] (-31:0.9cm) node{$\cdot$};
      \draw[red] (-33.5:0.9cm) node{$\cdot$};
      \draw[red] (-36:0.9cm) node{$\cdot$};

      \draw[gray,very thin] (0.25,0.21) ellipse (0.01cm and 0.08cm);
      \draw[gray,very thin] (0.248,0.132) -- (0.23,0.060);
      \draw (0.29,0.030) node{${\scriptscriptstyle t(b_{-1},v_j)-1 \;\mathrm{arcs}}$};

      \draw (-0.085,-0.175) node{$\scriptstyle \fa_1$};
      \draw (0.085,-0.175) node{$\scriptstyle \fa_j$};
      \draw (67:0.835cm) node{$\scriptstyle \fb_{-2}$};
      \draw (80:0.825cm) node{$\scriptstyle \fb_{-1}$};
      \draw (98:0.825cm) node{$\scriptstyle \fb_0$};
      \draw (111:0.835cm) node{$\scriptstyle \fb_1$};

      \draw (-59:0.825cm) node{$\scriptstyle \fv_{j+1}$};
      \draw (-71:0.825cm) node{$\scriptstyle \fv_j$};
      \draw (-108:0.825cm) node{$\scriptstyle \fv_{-1}$};
      \draw (-121:0.835cm) node{$\scriptstyle \fv_{-2}$};

      \draw (50.5:1cm) node[fill=white,rectangle,inner sep=0.07cm] {$\scriptstyle \I$};
      \draw (141:1cm) node[fill=white,rectangle,inner sep=0.07cm] {$\scriptstyle \II$};
      \draw (231:1cm) node[fill=white,rectangle,inner sep=0.07cm] {$\scriptstyle \III$};
      \draw (323:1cm) node[fill=white,rectangle,inner sep=0.07cm] {$\scriptstyle \IV$};

    \end{tikzpicture} 
    \caption{Outline of the triangulation $\fT$ of $D_4$ corresponding
      to an $\SL2$-tiling $t$ with entries equal to $1$ only in a
      single row.  The arcs in $\Theta( t )$ are black.  We add red
      arcs from the non-saturated vertices to define $\fT$.  The number
      of red arcs added is given by the defect at the relevant vertex,
      but at $b_0^{\I}$ there is a choice of how many arcs should go
      to interval $\II$, and how many to $\IV$.  This is resolved by
      letting $t( b_{-1},v_j ) - 1$ arcs go to $\IV$. }
\label{fig:ft_in_Case4}
\end{figure}
\end{Construction}

\begin{Theorem}
\label{thm:Case4}
Let $t$ be an $\SL2$-tiling with at least two entries equal to $1$,
and assume these occur only in a single row or in a single column.

Then there is a good triangulation $\fT$ of $D_4$ such that $\Phi( \fT
) = t$.
\end{Theorem}

\begin{proof}
As remarked at the start of the section, it is enough by symmetry to
let $t$ be an $\SL2$-tiling with entries equal to $1$ only in a single
row.  Let $\fT$ be as in Construction \ref{con:Case4}, see Figure
\ref{fig:ft_in_Case4}. 

Consider the finite polygons between the arcs $\fa_1$ and $\fa_j$ and
below the arcs $\fb_m$ and $\fv_m$ shown in Figure
\ref{fig:ft_in_Case4}, see Definition \ref{def:below}.  In each such
polygon, $\Theta( t )$ and hence $\fT$ restricts to a triangulation by
Lemmas \ref{lem:finite_triangulation1} through
\ref{lem:finite_triangulation2}.  The arcs added at the end of
Construction \ref{con:Case4} (red in Figure \ref{fig:ft_in_Case4})
clearly complete $\fT$ to a triangulation of $D_4$.  These arcs also
block all four accumulation points so $\fT$ is a good triangulation.

To show $\Phi( \fT ) = t$ we use Lemma \ref{lem:agree}.

Since $\fa_1, \fa_j \in \fT$ we have
\begin{equation}
\label{equ:Case4_1}
  t( b_0,v_1 ) = 1 = \fT( b_0^{\I},v_1^{\III})
  \;,\;
  t( b_0,v_j ) = 1 = \fT( b_0^{\I},v_j^{\III}).
\end{equation}
Moreover, the vertices $b_{-1}^{\I}$, $b_0^{\I}$, $v_j^{\III}$,
$0^{\IV}$, $\ldots$, $\varphi^{\IV}$ can be viewed as the vertices of a
finite polygon $P$ inside which $\fT$ restricts to a triangulation
$\fT_P$.  In $P$, the vertices $b_{-1}^{\I}$, $b_0^{\I}$,
$v_j^{\III}$, are consecutive whence
\[
  \fT_P( b_{-1}^{\I},v_j^{\III})
  = 1 + (\mbox{the number of arcs in $\fT_P$ ending at $b_0^{\I}$})
  = t( b_{-1},v_j )
\]
where the first equality is by Lemma \ref{lem:CC}(vi) and the second
equality is by the construction of $\fT$, see Construction
\ref{con:Case4} and Figure \ref{fig:ft_in_Case4}.  It follows that
\begin{equation}
\label{equ:Case4_2}
  t( b_{-1},v_j )
  = \fT_P( b_{-1}^{\I},v_j^{\III})  
  = \fT( b_{-1}^{\I},v_j^{\III}),
\end{equation}
where the second equality is by Remark \ref{rmk:CC}.  Equations
\eqref{equ:Case4_1} and \eqref{equ:Case4_2} verify condition (iii)' in
Lemma \ref{lem:agree} with $e = b_{-1}$, $f = b_0$, $g = v_1$, $h =
v_j$.

Lemma \ref{lem:agree}, conditions (i) and (ii) are verified by the
same method as in the last paragraph of the proof of Theorem
\ref{thm:Case2}.
\end{proof}

\section{Case 5: $\SL2$-tilings with a unique entry equal to $1$}
\label{sec:Case5}

Let $t$ be an $\SL2$-tiling in which $t( b_0,v_1 ) = 1$ is the unique
entry equal to $1$.

\begin{Construction}
[The triangulation $\fT$]
\label{con:Case5}
To construct $\fT$, proceed similarly to Construction \ref{con:Case4}
with a small tweak.

The black arcs in Figure \ref{fig:ft_in_Case5} show the overall
structure of $\Theta( t )$ which can be obtained by the method used in
Description \ref{des:Case4}.  However, there is now only a single
connecting arc $\fa = \{ b_0^{\I},v_1^{\III} \}$.
\begin{figure}
  \centering
    \begin{tikzpicture}[scale=5]
%      \draw[step=.25cm,gray,very thin] (-1.4,-1.4) grid (1.4,1.4);

      \draw (0,0) circle (1cm);

      \draw (45:1cm) node[fill=white,circle,inner sep=0.101cm] {} circle (0.03cm);
      \draw (135:1cm) node[fill=white,circle,inner sep=0.101cm] {} circle (0.03cm);
      \draw (225:1cm) node[fill=white,circle,inner sep=0.101cm] {} circle (0.03cm);
      \draw (315:1cm) node[fill=white,circle,inner sep=0.101cm] {} circle (0.03cm);

      \draw (76:0.97cm) -- (76:1.03cm);
      \draw (76:1.13cm) node{$b_{-2}$};
      \draw (88:0.97cm) -- (88:1.03cm);
      \draw (88:1.13cm) node{$b_{-1}$};
      \draw (100:0.97cm) -- (100:1.03cm);
      \draw (100:1.13cm) node{$b_0$};
      \draw (112:0.97cm) -- (112:1.03cm);
      \draw (112:1.13cm) node{$b_1$};
      \draw (124:0.97cm) -- (124:1.03cm);
      \draw (124:1.13cm) node{$b_2$};

      \draw (-76:0.97cm) -- (-76:1.03cm);
      \draw (-76:1.13cm) node{$v_3$};
      \draw (-88:0.97cm) -- (-88:1.03cm);
      \draw (-88:1.13cm) node{$v_2$};
      \draw (-100:0.97cm) -- (-100:1.03cm);
      \draw (-100:1.13cm) node{$v_1$};
      \draw (-112:0.97cm) -- (-112:1.03cm);
      \draw (-112:1.13cm) node{$v_0$};
      \draw (-124:0.97cm) -- (-124:1.03cm);
      \draw (-124:1.13cm) node{$v_{-1}$};

      \draw (152:0.97cm) -- (152:1.03cm);
      \draw (165:0.97cm) -- (165:1.03cm);
      \draw (180:0.97cm) -- (180:1.03cm);
      \draw (180:1.13cm) node{$0$};
      \draw (195:0.97cm) -- (195:1.03cm);
      \draw (208:0.97cm) -- (208:1.03cm);

      \draw (28:0.97cm) -- (28:1.03cm);
      \draw (15:0.97cm) -- (15:1.03cm);
      \draw (0:0.97cm) -- (0:1.03cm);
      \draw (0:1.13cm) node{$0$};
      \draw (-15:0.97cm) -- (-15:1.03cm);
      \draw (-28:0.97cm) -- (-28:1.03cm);

      \draw (-100:1cm) .. controls (-0.02,-0.3) and (-0.02,0.3) .. (100:1cm);

      \draw (100:1cm) .. controls (104:0.8cm) and (108:0.8cm) .. (112:1cm);
      \draw (112:1cm) .. controls (116:0.8cm) and (120:0.8cm) .. (124:1cm);
      \draw (100:1cm) .. controls (96:0.8cm) and (92:0.8cm) .. (88:1cm);
      \draw (88:1cm) .. controls (84:0.8cm) and (80:0.8cm) .. (76:1cm);

      \draw (72:0.9cm) node{$\cdot$};
      \draw (70:0.9cm) node{$\cdot$};
      \draw (68:0.9cm) node{$\cdot$};

      \draw (128:0.9cm) node{$\cdot$};
      \draw (130:0.9cm) node{$\cdot$};
      \draw (132:0.9cm) node{$\cdot$};

      \draw (-76:1cm) .. controls (-80:0.8cm) and (-84:0.8cm) .. (-88:1cm);
      \draw (-88:1cm) .. controls (-92:0.8cm) and (-96:0.8cm) .. (-100:1cm);
      \draw (-100:1cm) .. controls (-104:0.8cm) and (-108:0.8cm) .. (-112:1cm);
      \draw (-112:1cm) .. controls (-116:0.8cm) and (-120:0.8cm) .. (-124:1cm);

      \draw (286:0.9cm) node{$\cdot$};
      \draw (288:0.9cm) node{$\cdot$};
      \draw (290:0.9cm) node{$\cdot$};

      \draw (-126:0.9cm) node{$\cdot$};
      \draw (-128:0.9cm) node{$\cdot$};
      \draw (-130:0.9cm) node{$\cdot$};

      \draw[red] (100:1cm) .. controls (130:0.15cm) and (150:0.15cm) .. (180:1cm);
      \draw[red] (100:1cm) .. controls (130:0.25cm) and (150:0.2cm) .. (165:1cm);
      \draw[red] (112:1cm) .. controls (130:0.3cm) and (150:0.3cm) .. (165:1cm);
      \draw[red] (112:1cm) .. controls (130:0.3cm) and (150:0.4cm) .. (152:1cm);

      \draw[red] (170:0.9cm) node{$\cdot$};
      \draw[red] (172.5:0.9cm) node{$\cdot$};
      \draw[red] (175:0.9cm) node{$\cdot$};

      \draw[red] (156:0.9cm) node{$\cdot$};
      \draw[red] (158.5:0.9cm) node{$\cdot$};
      \draw[red] (161:0.9cm) node{$\cdot$};

      \draw[red] (149:0.9cm) node{$\cdot$};
      \draw[red] (146.5:0.9cm) node{$\cdot$};
      \draw[red] (144:0.9cm) node{$\cdot$};

      \draw[red] (-100:1cm) .. controls (-130:0.15cm) and (-150:0.15cm) .. (-180:1cm);
      \draw[red] (-100:1cm) .. controls (-130:0.25cm) and (-150:0.2cm) .. (-165:1cm);
      \draw[red] (-112:1cm) .. controls (-130:0.3cm) and (-150:0.3cm) .. (-165:1cm);
      \draw[red] (-112:1cm) .. controls (-130:0.3cm) and (-150:0.4cm) .. (-152:1cm);

      \draw[red] (-170:0.9cm) node{$\cdot$};
      \draw[red] (-172.5:0.9cm) node{$\cdot$};
      \draw[red] (-175:0.9cm) node{$\cdot$};

      \draw[red] (-156:0.9cm) node{$\cdot$};
      \draw[red] (-158.5:0.9cm) node{$\cdot$};
      \draw[red] (-161:0.9cm) node{$\cdot$};

      \draw[red] (-149:0.9cm) node{$\cdot$};
      \draw[red] (-146.5:0.9cm) node{$\cdot$};
      \draw[red] (-144:0.9cm) node{$\cdot$};

      \draw[red] (100:1cm) .. controls (50:0.15cm) and (30:0.15cm) .. (0:1cm);
      \draw[red] (100:1cm) .. controls (50:0.25cm) and (30:0.2cm) .. (15:1cm);
      \draw[red] (88:1cm) .. controls (50:0.3cm) and (30:0.3cm) .. (15:1cm);
      \draw[red] (88:1cm) .. controls (50:0.3cm) and (30:0.4cm) .. (28:1cm);

      \draw[red] (10:0.9cm) node{$\cdot$};
      \draw[red] (7.5:0.9cm) node{$\cdot$};
      \draw[red] (5:0.9cm) node{$\cdot$};

      \draw[red] (24:0.9cm) node{$\cdot$};
      \draw[red] (21.5:0.9cm) node{$\cdot$};
      \draw[red] (19:0.9cm) node{$\cdot$};

      \draw[red] (31:0.9cm) node{$\cdot$};
      \draw[red] (33.5:0.9cm) node{$\cdot$};
      \draw[red] (36:0.9cm) node{$\cdot$};

      \draw[red] (-100:1cm) .. controls (-70:0.15cm) and (-50:0.15cm) .. (0:1cm);
      \draw[red] (-100:1cm) .. controls (-70:0.25cm) and (-50:0.2cm) .. (-15:1cm);
      \draw[red] (-88:1cm) .. controls (-70:0.3cm) and (-50:0.3cm) .. (-15:1cm);
      \draw[red] (-88:1cm) .. controls (-70:0.3cm) and (-50:0.4cm) .. (-28:1cm);

      \draw[red] (-10:0.9cm) node{$\cdot$};
      \draw[red] (-7.5:0.9cm) node{$\cdot$};
      \draw[red] (-5:0.9cm) node{$\cdot$};

      \draw[red] (-25:0.9cm) node{$\cdot$};
      \draw[red] (-22.5:0.9cm) node{$\cdot$};
      \draw[red] (-20:0.9cm) node{$\cdot$};

      \draw[red] (-32:0.9cm) node{$\cdot$};
      \draw[red] (-34.5:0.9cm) node{$\cdot$};
      \draw[red] (-37:0.9cm) node{$\cdot$};

      \draw[gray,very thin] (0.200,0.235) ellipse (0.01cm and 0.08cm);
      \draw[gray,very thin] (0.197,0.157) -- (0.150,0.075);
      \draw (0.200,0.055) node{${\scriptscriptstyle t(b_{-1},v_1)-1 \;\mathrm{arcs}}$};

      \draw[gray,very thin] (0.200,-0.235) ellipse (0.01cm and 0.08cm);
      \draw[gray,very thin] (0.197,-0.157) -- (0.140,-0.105);
      \draw (0.175,-0.075) node{${\scriptscriptstyle t(b_0,v_2)-1 \;\mathrm{arcs}}$};

      \draw (-0.085,0) node{$\scriptstyle \fa$};
      % \draw (77:0.835cm) node{$\scriptstyle \fb_{-2}$};
      % \draw (90:0.82cm) node{$\scriptstyle \fb_{-1}$};
      % \draw (108:0.825cm) node{$\scriptstyle \fb_0$};
      % \draw (121:0.835cm) node{$\scriptstyle \fb_1$};

      % \draw (-79:0.825cm) node{$\scriptstyle \fv_1$};
      % \draw (-91:0.825cm) node{$\scriptstyle \fv_0$};
      % \draw (-108:0.825cm) node{$\scriptstyle \fv_{-1}$};
      % \draw (-121:0.835cm) node{$\scriptstyle \fv_{-2}$};

      \draw (52.5:1cm) node[fill=white,rectangle,inner sep=0.07cm] {$\scriptstyle \I$};
      \draw (142:1cm) node[fill=white,rectangle,inner sep=0.07cm] {$\scriptstyle \II$};
      \draw (231:1cm) node[fill=white,rectangle,inner sep=0.07cm] {$\scriptstyle \III$};
      \draw (323:1cm) node[fill=white,rectangle,inner sep=0.07cm] {$\scriptstyle \IV$};

    \end{tikzpicture} 
    \caption{Outline of the triangulation $\fT$ of $D_4$ corresponding
      to an $\SL2$-tiling $t$ with only a single entry equal to $1$.
      The arcs in $\Theta( t )$ are black.  We add red arcs from the
      non-saturated vertices to define $\fT$.  The number of red arcs
      added is given by the defect at the relevant vertex, but at
      $b_0^{\I}$ there is a choice of how many arcs should go to
      interval $\II$, and how many to $\IV$.  This is resolved by
      letting $t( b_{-1},v_1 ) - 1$ arcs go to $\IV$.
      Similarly at $v_1^{\III}$ we let $t( b_0,v_2 ) -1$
      arcs go to $\II$.}
\label{fig:ft_in_Case5}
\end{figure}
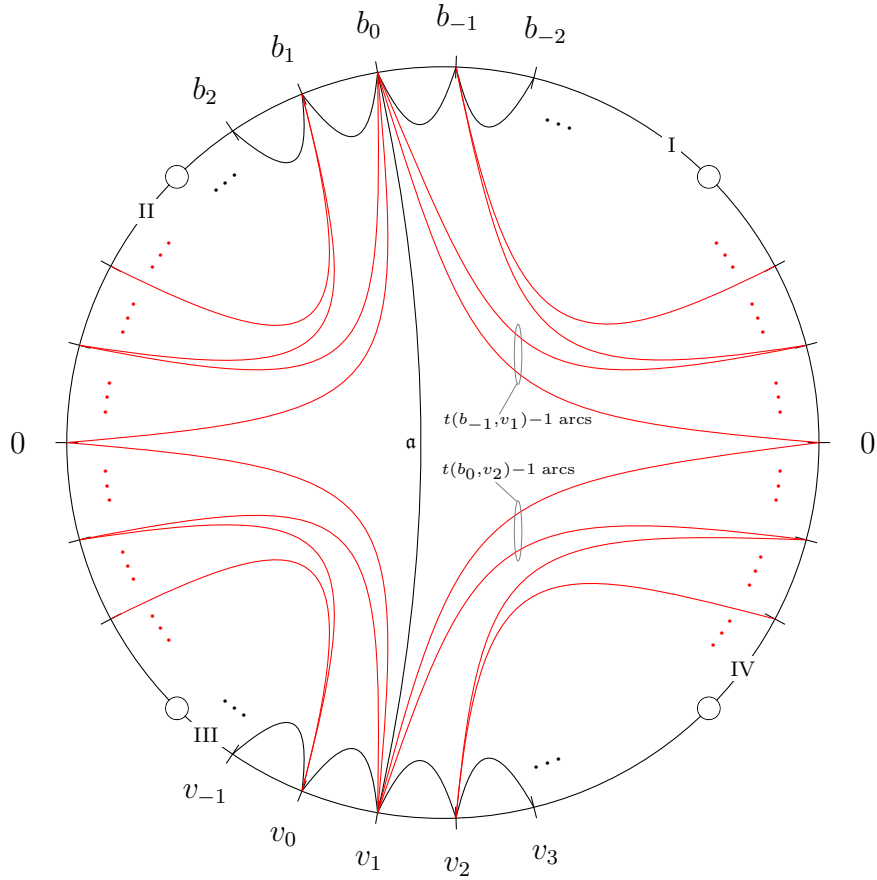

When adding red arcs to obtain the triangulation $\fT$ of $D_4$, the
red arcs go to interval $\II$ or interval $\IV$, depending on which
side of $\fa$ they are on.  We always add as many red arcs at a vertex
as dictated by the defect at that vertex.

At vertices $b_0^{\I}$ and $v_1^{\III}$ only, there are red arcs to
both intervals $\II$ and $\IV$.

From $b_0^{\I}$ there are $t( b_{-1},v_1 ) - 1$ arcs to $\IV$.  Note
that by Lemma \ref{lem:connecting_defect}, this number is strictly
smaller than $\defect_p( b_0 )$, so there will also be at least one arc
from $b_0^{\I}$ to $\II$.

From $v_1^{\III}$ there are $t( b_0,v_2 ) - 1$ arcs to $\IV$.  Again,
this number is strictly smaller than $\defect_q( v_1 )$, so there will
also be at least one arc from $v_1^{\III}$ to $\II$.
\end{Construction}

\begin{Theorem}
\label{thm:Case5}
Let $t$ be an $\SL2$-tiling with a unique entry equal to $1$.

Then there is a good triangulation $\fT$ of $D_4$ such that $\Phi( \fT
) = t$.
\end{Theorem}

\begin{proof}
We suppose $t( b_0,v_1 ) = 1$ and let $\fT$ be as in Construction
\ref{con:Case5}, see Figure \ref{fig:ft_in_Case5}. 

Arguing like the proof of Theorem \ref{thm:Case4}, the choices at the
end of Construction \ref{con:Case5} imply that $\fT$ satisfies
\[
  t( b_{-1},v_1 ) = \fT( b_{-1}^{\I},v_1^{\III} ) 
  \;,\;
  t( b_0,v_2 ) = \fT( b_0^{\I},v_2^{\III} ).
\]
We also have
\[
  t( b_0,v_1 ) = 1 = \fT( b_0^{\I},v_1^{\III} )
\]
so condition (iii) of Lemma \ref{lem:agree} holds with $e = b_{-1}$,
$f = b_0$, $g = v_1$, $h = v_2$.  Now proceed 
like the proof of Theorem \ref{thm:Case4}.
\end{proof}

\section{A lemma on Conway--Coxeter friezes}
\label{sec:CC2}

\begin{Definition}
In this section we will write
$\cS = \biggl\{ \begin{pmatrix} i & j \\ k & \ell \end{pmatrix} \in
\SL2( \BZ ) \bigg| i,j,k,\ell \geqslant 0 \biggr\}$.
\end{Definition}

The following lemma was proved in \cite[thm.\ 6.2]{BG} and
\cite[lem.\ 4.1]{CL}. 

\begin{Lemma}
\label{lem:building_a_matrix}
Each $X$ in $\cS$ can be obtained by starting with the $2 \times 2$
identity matrix $E$ and performing a sequence of operations of the
form: add one of the rows to the other row, or add one of the columns
to the other column.
\end{Lemma}

\begin{Lemma}
\label{lem:a_matrix_exists}
Let $r$ and $m$ be coprime integers with $0 < r < m$.  There
exists
$\begin{pmatrix}
   i & j \\
   k & \ell
 \end{pmatrix}$
in $\cS$ such that $r = i+j$, $m = i+j+k+\ell$.
\end{Lemma}

\begin{proof}
Set $n = m - r$.  Then $r$ and $n$ are coprime so there are integers
$s$, $p$ with $sr - pn = 1$.

We can replace $s$, $p$ with $s + tn$, $p + tr$ so may assume $0
\leqslant p < r$.  It follows that $0 \leqslant pn < rn$, and since
$pn = rs - 1$ this reads $0 \leqslant rs - 1 < rn$, that is $1
\leqslant rs < rn + 1$, that is $1 \leqslant rs \leqslant rn$.  Hence
$1 \leqslant s \leqslant n$.

It is now straightforward to check that
\[
  \begin{pmatrix}
    i & j \\
    k & \ell
  \end{pmatrix}
  =
  \begin{pmatrix}
    r-p & p \\
    n-s & s
  \end{pmatrix}
\]  
can be used in the lemma.
\end{proof}

\begin{Lemma}
\label{lem:CC2}
Let $r$ and $m$ be coprime integers with $0 < r < m$.

There exists a finite polygon $R$ which has two adjacent vertices $\chi$ and $\chi^+ = \beta$, two adjacent vertices $\gamma$ and $\gamma^+ = \varphi$,
and a triangulation $\fS$ such that
\begin{align*}
  r & = \fS( \chi,\gamma ) + \fS( \chi,\varphi ), \\
  m & = \fS( \chi,\gamma ) + \fS( \chi,\varphi )
        + \fS( \beta,\gamma ) + \fS( \beta,\varphi ).
\end{align*}
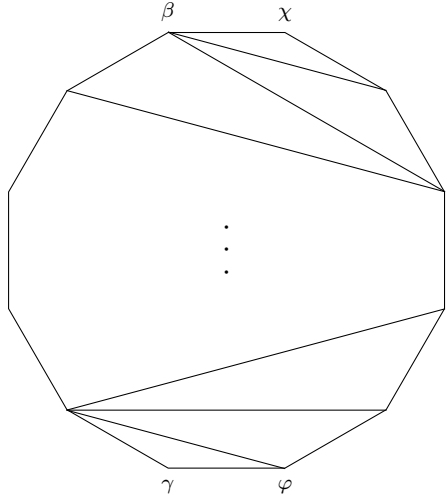
\begin{figure}
  \centering
    \begin{tikzpicture}[auto]
      \node[name=s, shape=regular polygon, regular polygon sides=12, minimum size=6cm, draw] {}; 
      \draw[shift=(s.corner 1)] node[above] {$\scriptstyle \chi$};
      \draw[shift=(s.corner 2)] node[above] {$\scriptstyle \beta$};
      \draw[shift=(s.corner 7)] node[below] {$\scriptstyle \gamma$};
      \draw[shift=(s.corner 8)] node[below] {$\scriptstyle \varphi$};
      \draw (s.corner 2) to (s.corner 12);
      \draw (s.corner 2) to (s.corner 11);
      \draw (s.corner 11) to (s.corner 3);
      \draw (s.corner 6) to (s.corner 8);
      \draw (s.corner 6) to (s.corner 9);
      \draw (s.corner 6) to (s.corner 10);
      \draw (0,0.3) node{$\cdot$};
      \draw (0,0) node{$\cdot$};
      \draw (0,-0.3) node{$\cdot$};
    \end{tikzpicture} 
    \caption{The polygon $R$ with two adjacent vertices $\chi$,
      $\beta$, two adjacent vertices $\gamma$, $\varphi$, and a
      triangulation $\fS$.}
\label{fig:chi_gamma}
\end{figure}
\end{Lemma}

\begin{proof}
By Lemma \ref{lem:a_matrix_exists} there is
\[
  X = 
  \begin{pmatrix}
    i & j \\
    k & \ell
  \end{pmatrix}
\]
in $\cS$ with $r = i+j$, $m = i+j+k+\ell$.  It is hence enough to show
the following:
\renewcommand{\labelenumi}{(\alph{enumi})}
\begin{enumerate}
\setlength\itemsep{4pt}

  \item There exists a finite polygon $R$ which has two adjacent
  vertices $\chi$ and $\chi^+ = \beta$, two adjacent vertices $\gamma$
  and $\gamma^+ = \varphi$, and a triangulation $\fS$ such that $X$
  is equal to
\begin{equation}
\label{equ:CC2}
  Y = 
  \begin{pmatrix}
    \fS( \chi,\gamma )  & \fS( \chi,\varphi ) \\
    \fS( \beta,\gamma ) & \fS( \beta,\varphi )
  \end{pmatrix}.
\end{equation}

\end{enumerate}
\renewcommand{\labelenumi}{(\roman{enumi})}
To show that (a) holds for each $X$ in $\cS$, it is enough to show the
following by Lemma \ref{lem:building_a_matrix}.
\begin{enumerate}
\setlength\itemsep{4pt}

\item  If $X = E$ then (a) holds.

\item  If (a) holds for a matrix $X$ in $\cS$, then it also holds for
  the matrices $X'$ obtained from $X$ by operations of the form: add
  one of the rows to the other row, or add one of the columns to the
  other column.

\end{enumerate}

(i) is true since for $X = E$, we can let $R$ be the $2$-gon
with $\chi = \varphi$ equal to one of the vertices, $\beta = \gamma$
equal to the other, and $\fS$ empty.

(ii): Suppose that (a) holds for $X$ in $\cS$ with the polygon $R$ and
triangulation $\fS$.  Let $\Psi$ denote one of the operations
described in (ii) and perform $\Psi$ on $X$ to obtain a new matrix
$X'$.  To show that (a) holds for $X'$, it is enough to show that
there is a way to change $R$ and $\fS$ to $R'$ and $\fS'$ such that
$\Psi$ is performed on the matrix $Y$ in equation \eqref{equ:CC2}.

We specialise to $\Psi$ being the operation of adding the first row to
the second, since the other operations have similar proofs.
\begin{figure}
  \centering
    \begin{tikzpicture}[auto]
      \node[name=s, shape=regular polygon, regular polygon sides=12, minimum size=6cm, draw] {}; 
      \draw(72:3.25cm) node{$\scriptstyle \chi$};
      \draw(108:3.3cm) node{$\scriptstyle \beta_{\rm old}$};
      \draw(90:3.8cm) node{$\scriptstyle \beta_{\rm new}$};
      \draw[shift=(s.corner 7)] node[below] {$\scriptstyle \gamma$};
      \draw[shift=(s.corner 8)] node[below] {$\scriptstyle \varphi$};

      \draw (s.corner 1) to (90:3.6cm);
      \draw (s.corner 2) to (90:3.6cm);

      \draw (s.corner 2) to (s.corner 12);
      \draw (s.corner 2) to (s.corner 11);
      \draw (s.corner 11) to (s.corner 3);
      \draw (s.corner 6) to (s.corner 8);
      \draw (s.corner 6) to (s.corner 9);
      \draw (s.corner 6) to (s.corner 10);
      \draw (0,0.3) node{$\cdot$};
      \draw (0,0) node{$\cdot$};
      \draw (0,-0.3) node{$\cdot$};
    \end{tikzpicture} 
    \caption{Compared to Figure \ref{fig:chi_gamma}, an ``ear'' has
      been glued to the triangulated polygon $R$, resulting in a new
      triangulated polygon $R'$.}
\label{fig:chi_gamma2}
\end{figure}
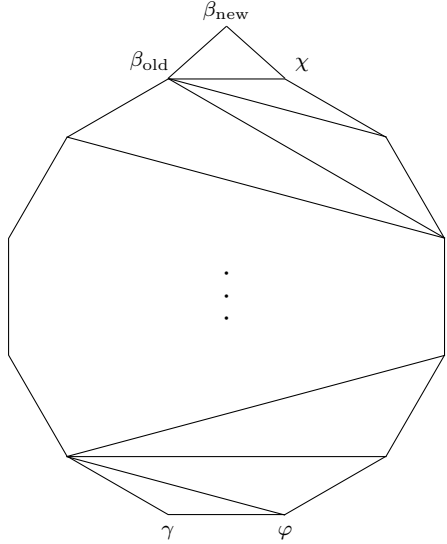
To go from $R$ and $\fS$ to $R'$ and $\fS'$, it turns out that we can
glue an ``ear'' as illustrated by going from Figure
\ref{fig:chi_gamma} to Figure \ref{fig:chi_gamma2}.  That is, $R'$
keeps the vertices of $R$, with $\beta$ renamed $\beta_{\rm old}$, and
acquires a new vertex, $\beta_{\rm new}$, between $\beta_{\rm old}$
and $\chi$.  And $\fS'$ keeps the arcs of $\fS$ and acquires a new arc, $\{ \chi,\beta_{\rm old} \}$.

It is clear from Definition \ref{def:CC} that for an arbitrary pair of
the vertices $\chi$, $\beta_{\rm old}$, $\gamma$, $\varphi$ in Figure
\ref{fig:chi_gamma2}, Conway--Coxeter counting on $\fS$ and $\fS'$ gives the
same result.  Hence the first row of $Y$ is unchanged by going to $R'$
and $\fS'$.

If $\gamma$ is not equal to $\beta_{\rm old}$ or $\chi$ then the
arcs $\{ \beta_{\rm old},\chi \}$ and $\{ \beta_{\rm new},\gamma
\}$ in $R'$ cross, so the Ptolemy formula in Lemma \ref{lem:CC}(v)
gives
\[
  \fS'( \beta_{\rm old},\chi )\fS'( \beta_{\rm new},\gamma )
  = \fS'( \beta_{\rm old},\gamma )\fS'( \beta_{\rm new},\chi )
  + \fS'( \beta_{\rm old},\beta_{\rm new} )\fS'( \chi,\gamma ).
\]
Since $\{ \beta_{\rm old},\chi \}$ is in $\fS'$ and $\{ \beta_{\rm
  new},\chi \}$, $\{ \beta_{\rm old},\beta_{\rm new} \}$ are edges,
the corresponding factors in the equation are equal to $1$.  This
gives the first of the following equalities:
\[
  \fS'( \beta_{\rm new},\gamma )
  = \fS'( \beta_{\rm old},\gamma ) + \fS'( \chi,\gamma )
  = \fS( \beta_{\rm old},\gamma ) + \fS( \chi,\gamma ).
\]
This also holds trivially for $\gamma$ equal to $\beta_{\rm old}$ or
$\chi$, and the same computation works with $\varphi$ instead of
$\gamma$.  Hence the first row of $Y$ is added to the second by going
to $R'$ and $\fS'$.
\end{proof}

\section{An $\SL2$-tiling with no entry equal to $1$ has a unique minimum}
\label{sec:minimum}

The following lemma is obvious.

\begin{Lemma}
\label{lem:no_neighbours}
Suppose that
$\begin{pmatrix}
    i & j \\
    k & \ell
 \end{pmatrix}$
is in $\SL2( \BZ )$ and that each entry is $\geqslant 1$.  If two entries
which are horizontal or vertical neighbours are equal, then they are
equal to $1$.
\end{Lemma}

\begin{Lemma}
\label{lem:SL2_inequalities}
Suppose that
$\begin{pmatrix}
    i & j \\
    k & \ell
 \end{pmatrix}$
is in $\SL2( \BZ )$, has each entry $\geqslant 2$,
and that $j < \ell$, $k < \ell$.  Then $i < j$, $i < k$.
\end{Lemma}

\begin{proof}
Lemma \ref{lem:no_neighbours} implies that we cannot have $i = j$.  If
we had $i > j$ then we would have $i\ell > j\ell$.  But we know $\ell
> k$ whence $j\ell > jk$.  Combining the inequalities would give
$i\ell > j\ell > jk$ so the determinant of the matrix would be $i\ell
- jk \geqslant 2$ which is false.

It follows that $i < j$, and $i < k$ is proved by considering the
transpose. 
\end{proof}

\begin{Lemma}
\label{lem:delete}
\begin{enumerate}
\setlength\itemsep{4pt}

  \item  If $t( b,v )$ is a local maximum in the $b$'th row of $t$ in
    the sense that
\begin{equation}
\label{equ:loc_max_1}
  t( b,v-1 ) < t( b,v ) > t( b,v+1 ),
\end{equation}
then deleting the $v$'th column from $t$ gives a new $\SL2$-tiling.

  \item  If $t( b,v )$ is a local maximum in the $v$'th column of $t$
    in the sense that
\[
  t( b-1,v ) < t( b,v ) > t( b+1,v ),
\]
then deleting the $b$'th row from $t$ gives a new $\SL2$-tiling.

\end{enumerate}
\end{Lemma}

\begin{proof}
(i):  The second Ptolemy relation of Lemma \ref{lem:Ptolemy}(ii) gives
$q( v-1,v+1 )t( c,v ) = q( v-1,v )t( c,v+1 ) + q( v,v+1 )t( c,v-1 )$
for each $c$, that is
\[
  q( v-1,v+1 )t( c,v ) = t( c,v+1 ) + t( c,v-1 ).
\]
Set $c = b$ and combine with the inequalities in part (i) of the
lemma.  It follows that the positive integer $q( v-1,v+1 )$ must be
$1$, so the displayed equation reads
\[
  t( c,v ) = t( c,v+1 ) + t( c,v-1 ) 
\]
for each $c$.  It is elementary from this that deleting from $t$ the
$v$'th column with entries $t( -,v )$ gives a new $\SL2$-tiling.

(ii) follows by symmetry.
\end{proof}

\begin{Lemma}
\label{lem:minimum}
Let $t$ be an $\SL2$-tiling with no entry equal to $1$.  Then $t$ has
a unique minimal entry.
\end{Lemma}

\begin{proof}
If the minimal entry $m \geqslant 2$ occurred twice in $t$, then it
would do so in one of the patterns shown in Figure \ref{fig:m}.
\begin{figure}
\begin{tabular}{cccccccccc}
  $m \cdots m$
  & & & $\begin{matrix} m \\ \vdots \\ m \end{matrix}$
  & & & $\begin{matrix}
       m & \cdots & * \\
       \vdots & & \vdots \\
       * & \cdots & m
     \end{matrix}$
  & & & $\begin{matrix}
       * & \cdots & m \\
       \vdots & & \vdots \\
       m & \cdots & *
     \end{matrix}$ \\[10mm]
  (i) & & & (ii) & & & (iii) & & & (iv)
\end{tabular}
\caption{A minimal entry $m$ appearing twice in an $\SL2$-tiling would
  have to do so in one of these patterns.}
\label{fig:m}
\end{figure}
We treat the cases separately.

(i): Suppose that the $b$'th row of $t$ has at least two entries equal
to $m$.  Pick two such entries which have no entries between them
equal to $m$.  Then either the two $m$'s are neighbours, or each entry
between them is $>m$.

In the latter case, somewhere between the two $m$'s is a local maximum
$t( b,v )$ in the sense of Equation \eqref{equ:loc_max_1}.  By Lemma
\ref{lem:delete}(i), we can delete column number $v$ and get a new
$\SL2$-tiling.  If we iterate this process, then all entries between
the two $m$'s will eventually be deleted, giving an $\SL2$-tiling
where the two $m$'s are neighbours.

However, two neighbouring $m$'s would contradict Lemma
\ref{lem:no_neighbours}.  

(ii): Symmetric to (i), replacing Lemma \ref{lem:delete}(i) with Lemma
\ref{lem:delete}(ii).  

(iii): Suppose that the two entries equal to $m$ are $t( b,v ) = t(
c,w ) = m$ with $b < c$, $v < w$.  Then the Ptolemy relation in Lemma
\ref{lem:Ptolemy}(iii) implies
\[
  m^2 = t( b,w )t( c,v ) + p( b,c )q( v,w ).
\]
However, this contradicts that $t( b,w ), t( c,v ) \geqslant m$ while
$p( b,c ), q( v,w ) \geqslant 1$.

(iv): Suppose that the two entries equal to $m$ are $t( b,v ) = t( c,w
) = m$ with $b > c$, $v < w$.  Repeat as many times as possible the
process of seeking out local maxima among the entries $t( b,v+1 )$,
$\ldots$, $t( b,w-1 )$ and $t( c,v+1 )$, $\ldots$, $t( c,w-1 )$ and
deleting the corresponding columns using Lemma \ref{lem:delete}(i).
Then repeat as many times as possible the process of seeking out local
maxima among the entries $t( c+1,v )$, $\ldots$, $t( b-1,v )$ and $t(
c+1,w )$, $\ldots$, $t( b-1,w )$ and deleting the corresponding rows
using Lemma \ref{lem:delete}(ii).

The resulting $\SL2$-tiling $t'$ still contains the two entries equal
to $m$ which we started with, and they are still minimal.  Since the
local maxima are gone, the entries of $t'$ satisfy the inequalities in
Figure \ref{fig:ineq}. 
\begin{figure}
\[
  \begin{matrix}
    * & > & * & > & \cdots & > & m \\
    \vertgeq &&&&&& \vertleq \\
    \vdots &&&&&& \vdots \\
    \vertgeq &&&&&& \vertleq \\
    * &&&&&& * \\
    \vertgeq &&&&&& \vertleq \\
    m & < & * & < & \cdots & < & *
  \end{matrix}
\]
\caption{If an $\SL2$-tiling $t$ has no entries equal to $1$ but has
  minimal entry $m$ occurring twice in the pattern from Figure
  \ref{fig:m}(iv), then we can achieve the inequalities shown here by
  deleting rows and columns from $t$.}
\label{fig:ineq}
\end{figure}
Note that the inequalities are sharp by Lemma \ref{lem:no_neighbours}
because each entry of $t'$ is $\geqslant 2$.

Starting from the lower right corner of Figure \ref{fig:ineq} and
moving left using Lemma \ref{lem:SL2_inequalities} repeatedly would
give that the two lower rows of Figure \ref{fig:ineq} satisfied the
following inequalities.
\[
  \begin{matrix}
    * & < & * & < & \cdots & < & * \\
    \vertleq &&&&&& \vertleq \\
    m & < & * & < & \cdots & < & *
  \end{matrix}
\]
However, the leftmost inequality contradicts Figure \ref{fig:ineq}.
\end{proof}

\section{Case 6: $\SL2$-tilings with no entry equal to $1$}
\label{sec:Case6}

Let $t$ be an $\SL2$-tiling with no entry equal to $1$ and unique
minimal entry $t( b,v )$, see Lemma \ref{lem:minimum}.

\begin{Notation}
\label{not:Case6}
Let us describe part of what is shown with black arcs in Figure
\ref{fig:ft_in_Case6}: Since $\Theta( t )$ is locally finite by Lemma
\ref{lem:locally_finite}, we can let $a < b < c$ be such that
\begin{itemize}
\setlength\itemsep{4pt}

  \item $\fb_{-1} = \{ b^{\I},a^{\I} \}$ is the longest internal arc in
  $\Theta( t )$ going clockwise from $b^{\I}$, or, if there are no
  such arcs, the edge going clockwise from $b^{\I}$,

  \item $\fb_0 = \{ b^{\I},c^{\I} \}$ is the longest internal arc in
  $\Theta( t )$ going anticlockwise from $b^{\I}$, or, if there are no
  such arcs, the edge going anticlockwise from $b^{\I}$.

\end{itemize}
Likewise, we can let $v < w$ be such that
\begin{itemize}
\setlength\itemsep{4pt}

  \item $\fv_1 = \{ v^{\III},w^{\III} \}$ is the longest internal arc in
  $\Theta( t )$ going anticlockwise from $v^{\III}$, or, if there are
  no such arcs, the edge going anticlockwise from $v^{\III}$.

\end{itemize}
\end{Notation}

\begin{Lemma}
\label{lem:DwR}
Consider the following divisions with remainders:
\begin{align*}
  t( a,v ) = \ell t( b,v ) + r, \;\;\; & 0 \leqslant r < t( b,v ), \\
  t( b,w ) = m t( b,v ) + s, \;\;\; & 0 \leqslant s < t( b,v ).
\end{align*}
Then
\begin{enumerate}
\setlength\itemsep{4pt}

  \item  $0 < \ell < \defect_p( b )$,

  \item  $0 < m < \defect_q( v )$,

  \item  $rs \equiv 1 \bmod t( b,v )$.  Note that since $0 \leqslant
    r,s < t( b,v )$ by definition, it follows that $0 < r,s < t( b,v
    )$ and that $r,s$ are inverses modulo $t( b,v )$.

\end{enumerate}
\end{Lemma}

\begin{proof}
(i): Since $t$ has no entries equal to $1$, there are no connecting
arcs in $\Theta( t )$.  In particular, $\Theta( t )$ has no connecting
arcs ending at $b^{\I}$, so Lemma \ref{lem:internal_defect} can be
applied; see also Figure \ref{fig:bik}.  In the lemma and the figure, we
must set $b_{-1} = a$, $b_0 = b$, $b_1 = c$ to match the notation of
this section.  The lemma gives
\[
%\label{equ:p_a_c_2}
  p( a,c ) = \defect_p( b ) + 1.
\]
The Ptolemy relation in Lemma \ref{lem:Ptolemy}(ii) implies 
\[
  p( a,c )t( b,v ) = p( a,b )t( c,v ) + p( b,c )t( a,v ).
\]
Here $p( a,b ) = p( b,c ) = 1$ since $\{ a^{\I},b^{\I} \}, \{
b^{\I},c^{\I} \} \in \Theta( t )$, so combining the displayed
equations shows $\big( \defect_p( b ) + 1 \big)t( b,v ) = t( c,v ) +
t( a,v )$, that is,
\[
  t( a,v ) = \big( \defect_p( b ) + 1 \big)t( b,v ) - t( c,v )
           < \defect_p( b )t( b,v )
\]
where the inequality holds since $t( b,v )$ is the unique minimal
entry of $t$.  This implies part (i).

(ii):  Follows by symmetry.

(iii):  The Ptolemy relation in Lemma \ref{lem:Ptolemy}(iii) implies 
\[
  t( a,v )t( b,w ) = t( a,w )t( b,v ) + p( a,b )q( v,w ).
\]
Here $p( a,b ) = q( v,w ) = 1$ since $\{ a^{\I},b^{\I} \}, \{
v^{\III},w^{\III} \} \in \Theta( t )$ so
\[
  t( a,v )t( b,w ) \equiv 1 \bmod t( b,v ).
\]
Since $t( a,v ) \equiv r \bmod t( b,v )$ and $t( b,w ) \equiv s \bmod
t( b,v )$ by definition of $r$ and $s$, part (iii) follows.
\end{proof}

\begin{Remark}
\label{rmk:DwR}
Parts (i) and (ii) of the lemma imply $\defect_p( b ) \geqslant 2$
and $\defect_q( v ) \geqslant 2$ so $b^{\I}$ and $v^{\III}$ are
non-saturated vertices by Lemma \ref{lem:defect}.
\end{Remark}

\begin{Description}
[The partial triangulation $\Theta( t )$]
The black arcs in Figure \ref{fig:ft_in_Case6} show the overall
structure of $\Theta( t )$ which we now describe:
\begin{figure}
  \centering
    \begin{tikzpicture}[scale=5]
%      \draw[step=.25cm,gray,very thin] (-1.4,-1.4) grid (1.4,1.4);

      \draw (0,0) circle (1cm);

      \draw (45:1cm) node[fill=white,circle,inner sep=0.101cm] {} circle (0.03cm);
      \draw (135:1cm) node[fill=white,circle,inner sep=0.101cm] {} circle (0.03cm);
      \draw (225:1cm) node[fill=white,circle,inner sep=0.101cm] {} circle (0.03cm);
      \draw (315:1cm) node[fill=white,circle,inner sep=0.101cm] {} circle (0.03cm);

      \draw (66:0.97cm) -- (66:1.03cm);
      \draw (66:1.10cm) node{$\scriptstyle b_{-2}$};
      \draw (78:0.97cm) -- (78:1.03cm);
      \draw (78:1.10cm) node{$\scriptstyle b_{-1} = a$};
      \draw (90:0.97cm) -- (90:1.03cm);
      \draw (90:1.10cm) node{$\scriptstyle b_0 = b$};
      \draw (102:0.97cm) -- (102:1.03cm);
      \draw (102:1.10cm) node{$\scriptstyle b_1 = c$};
      \draw (114:0.97cm) -- (114:1.03cm);
      \draw (114:1.10cm) node{$\scriptstyle b_2$};

      \draw (-66:0.97cm) -- (-66:1.03cm);
      \draw (-66:1.10cm) node{$\scriptstyle v_2$};
      \draw (-78:0.97cm) -- (-78:1.03cm);
      \draw (-78:1.10cm) node{$\scriptstyle v_1 = w$};
      \draw (-90:0.97cm) -- (-90:1.03cm);
      \draw (-90:1.10cm) node{$\scriptstyle v_0 = v$};
      \draw (-102:0.97cm) -- (-102:1.03cm);
      \draw (-102:1.10cm) node{$\scriptstyle v_{-1}$};
      \draw (-114:0.97cm) -- (-114:1.03cm);
      \draw (-114:1.10cm) node{$\scriptstyle v_{-2}$};

      \draw (152:0.97cm) -- (152:1.03cm);
      \draw (160:0.97cm) -- (160:1.03cm);
      \draw (170:0.97cm) -- (170:1.03cm);
      \draw (170:1.10cm) node{$\scriptstyle \beta$};
      \draw (190:0.97cm) -- (190:1.03cm);
      \draw (190:1.10cm) node{$\scriptstyle \gamma$};
      \draw (200:0.97cm) -- (200:1.03cm);
      \draw (208:0.97cm) -- (208:1.03cm);

      \draw (28:0.97cm) -- (28:1.03cm);
      \draw (20:0.97cm) -- (20:1.03cm);
      \draw (20:1.10cm) node{$\scriptstyle \psi$};
      \draw (10:0.97cm) -- (10:1.03cm);
      \draw (10:1.10cm) node{$\scriptstyle \chi$};
      \draw (-10:0.97cm) -- (-10:1.03cm);
      \draw (-10:1.10cm) node{$\scriptstyle \varphi$};
      \draw (-20:0.97cm) -- (-20:1.03cm);
      \draw (-28:0.97cm) -- (-28:1.03cm);

      \draw (90:1cm) .. controls (94:0.8cm) and (98:0.8cm) .. (102:1cm);
      \draw (102:1cm) .. controls (106:0.8cm) and (110:0.8cm) .. (114:1cm);
      \draw (90:1cm) .. controls (86:0.8cm) and (82:0.8cm) .. (78:1cm);
      \draw (78:1cm) .. controls (74:0.8cm) and (70:0.8cm) .. (66:1cm);

      \draw (62:0.9cm) node{$\cdot$};
      \draw (60:0.9cm) node{$\cdot$};
      \draw (58:0.9cm) node{$\cdot$};

      \draw (118:0.9cm) node{$\cdot$};
      \draw (120:0.9cm) node{$\cdot$};
      \draw (122:0.9cm) node{$\cdot$};

      \draw (-66:1cm) .. controls (-70:0.8cm) and (-74:0.8cm) .. (-78:1cm);
      \draw (-78:1cm) .. controls (-82:0.8cm) and (-86:0.8cm) .. (-90:1cm);
      \draw (-90:1cm) .. controls (-94:0.8cm) and (-98:0.8cm) .. (-102:1cm);
      \draw (-102:1cm) .. controls (-106:0.8cm) and (-110:0.8cm) .. (-114:1cm);

      \draw (296:0.9cm) node{$\cdot$};
      \draw (298:0.9cm) node{$\cdot$};
      \draw (300:0.9cm) node{$\cdot$};

      \draw (-116:0.9cm) node{$\cdot$};
      \draw (-118:0.9cm) node{$\cdot$};
      \draw (-120:0.9cm) node{$\cdot$};

      \draw[red] (90:1cm) .. controls (130:0.3cm) and (150:0.3cm) .. (170:1cm);
      \draw[red] (90:1cm) .. controls (130:0.35cm) and (150:0.3cm) .. (160:1cm);
      \draw[red] (102:1cm) .. controls (120:0.4cm) and (140:0.4cm) .. (160:1cm);
      \draw[red] (102:1cm) .. controls (120:0.4cm) and (140:0.5cm) .. (152:1cm);

      \draw[red] (162:0.9cm) node{$\cdot$};
      \draw[red] (164:0.9cm) node{$\cdot$};
      \draw[red] (166:0.9cm) node{$\cdot$};

      \draw[red] (153:0.9cm) node{$\cdot$};
      \draw[red] (155:0.9cm) node{$\cdot$};
      \draw[red] (157:0.9cm) node{$\cdot$};

      \draw[red] (149:0.9cm) node{$\cdot$};
      \draw[red] (146.5:0.9cm) node{$\cdot$};
      \draw[red] (144:0.9cm) node{$\cdot$};

      \draw[red] (-90:1cm) .. controls (-130:0.3cm) and (-150:0.3cm) .. (190:1cm);
      \draw[red] (-90:1cm) .. controls (-130:0.35cm) and (-150:0.3cm) .. (200:1cm);
      \draw[red] (-102:1cm) .. controls (-120:0.4cm) and (-140:0.4cm) .. (200:1cm);
      \draw[red] (-102:1cm) .. controls (-120:0.4cm) and (-140:0.5cm) .. (208:1cm);

      \draw[red] (-162:0.9cm) node{$\cdot$};
      \draw[red] (-164:0.9cm) node{$\cdot$};
      \draw[red] (-166:0.9cm) node{$\cdot$};

      \draw[red] (-153:0.9cm) node{$\cdot$};
      \draw[red] (-155:0.9cm) node{$\cdot$};
      \draw[red] (-157:0.9cm) node{$\cdot$};

      \draw[red] (-149:0.9cm) node{$\cdot$};
      \draw[red] (-146.5:0.9cm) node{$\cdot$};
      \draw[red] (-144:0.9cm) node{$\cdot$};

      \draw[red] (90:1cm) .. controls (50:0.3cm) and (30:0.3cm) .. (10:1cm);
      \draw[red] (90:1cm) .. controls (50:0.35cm) and (30:0.3cm) .. (20:1cm);
      \draw[red] (78:1cm) .. controls (60:0.4cm) and (40:0.4cm) .. (20:1cm);
      \draw[red] (78:1cm) .. controls (60:0.4cm) and (40:0.5cm) .. (28:1cm);

      \draw[red] (18:0.9cm) node{$\cdot$};
      \draw[red] (16:0.9cm) node{$\cdot$};
      \draw[red] (14:0.9cm) node{$\cdot$};

      \draw[red] (27:0.9cm) node{$\cdot$};
      \draw[red] (25:0.9cm) node{$\cdot$};
      \draw[red] (23:0.9cm) node{$\cdot$};

      \draw[red] (31:0.9cm) node{$\cdot$};
      \draw[red] (33.5:0.9cm) node{$\cdot$};
      \draw[red] (36:0.9cm) node{$\cdot$};

      \draw[red] (-90:1cm) .. controls (-70:0.3cm) and (-50:0.3cm) .. (-10:1cm);
      \draw[red] (-90:1cm) .. controls (-70:0.35cm) and (-50:0.3cm) .. (-20:1cm);
      \draw[red] (-78:1cm) .. controls (-60:0.4cm) and (-40:0.4cm) .. (-20:1cm);
      \draw[red] (-78:1cm) .. controls (-60:0.4cm) and (-40:0.5cm) .. (-28:1cm);

      \draw[red] (-18:0.9cm) node{$\cdot$};
      \draw[red] (-16:0.9cm) node{$\cdot$};
      \draw[red] (-14:0.9cm) node{$\cdot$};

      \draw[red] (-27:0.9cm) node{$\cdot$};
      \draw[red] (-25:0.9cm) node{$\cdot$};
      \draw[red] (-23:0.9cm) node{$\cdot$};

      \draw[red] (-31:0.9cm) node{$\cdot$};
      \draw[red] (-33.5:0.9cm) node{$\cdot$};
      \draw[red] (-36:0.9cm) node{$\cdot$};

      \draw[red] (10:1cm) .. controls (0.2,0.075) and (-0.2,0.075) .. (170:1cm);
      \draw[red] (-10:1cm) .. controls (0.2,-0.075) and (-0.2,-0.075) .. (-170:1cm);

      \draw[gray,very thin] (0.200,0.425) ellipse (0.01cm and 0.06cm);
      \draw[gray,very thin] (0.197,0.367) -- (0.150,0.270);
      \draw (0.150,0.255) node{${\scriptscriptstyle \ell \;\mathrm{arcs}}$};

      \draw[gray,very thin] (0.200,-0.390) ellipse (0.01cm and 0.06cm);
      \draw[gray,very thin] (0.197,-0.332) -- (0.150,-0.240);
      \draw (0.150,-0.220) node{${\scriptscriptstyle m \;\mathrm{arcs}}$};

      \draw (67:0.835cm) node{$\scriptstyle \fb_{-2}$};
      \draw (80:0.82cm) node{$\scriptstyle \fb_{-1}$};
      \draw (98:0.825cm) node{$\scriptstyle \fb_0$};
      \draw (111:0.835cm) node{$\scriptstyle \fb_1$};

      \draw (-69:0.825cm) node{$\scriptstyle \fv_2$};
      \draw (-81:0.825cm) node{$\scriptstyle \fv_1$};
      \draw (-98:0.815cm) node{$\scriptstyle \fv_0$};
      \draw (-111:0.825cm) node{$\scriptstyle \fv_{-1}$};

      \draw (52:1cm) node[fill=white,rectangle,inner sep=0.07cm] {$\scriptstyle \I$};
      \draw (142:1cm) node[fill=white,rectangle,inner sep=0.07cm] {$\scriptstyle \II$};
      \draw (232.5:1cm) node[fill=white,rectangle,inner sep=0.07cm] {$\scriptstyle \III$};
      \draw (323:1cm) node[fill=white,rectangle,inner sep=0.07cm] {$\scriptstyle \IV$};

    \end{tikzpicture} 
    \caption{Outline of the triangulation $\fT$ of $D_4$ corresponding
      to an $\SL2$-tiling $t$ with no entry equal to $1$.  The arcs in
      $\Theta( t )$ are black.  We add red arcs from the non-saturated
      vertices to define $\fT$.  The total number of red arcs added is
      given by the defect at the relevant vertex.  The vertices
      $b_0^{\I}$ and $v_0^{\III}$ are the only ones with red arcs both
      to intervals $\II$ and $\IV$.  These vertices are chosen by $t(
      b_0,v_0 )$ being the unique minimal entry in $t$.}
\label{fig:ft_in_Case6}
\end{figure}

We will set $b_0 = b$ and $v_0 = v$.  The vertex $b_0^{\I}$ is
non-saturated by Remark \ref{rmk:DwR}.  Let $\fb_0 = \{
b_0^{\I},b_1^{\I} \}$ be either the longest internal arc in $\Theta( t
)$ going anticlockwise from $b_0^{\I}$, or, if there are no such arcs,
the edge going anticlockwise from $b_0^{\I}$.  This makes sense
because $\Theta( t )$ is locally finite by Lemma
\ref{lem:locally_finite}.  It is easy to see that $b_1^{\I}$ is also
non-saturated, while the vertices strictly below $\fb_0$ are
saturated.

We can repeat this to both sides of $b_0$ and thereby get integers
$\cdots < b_{-1} < b_0 < b_1 < \cdots$ such that the non-saturated
vertices in interval $\I$ are precisely $\ldots$, $b_{-1}^{\I}$,
$b_0^{\I}$, $b_1^{\I}$, $\ldots$.

A similar treatment provides integers $\cdots < v_{-1} < v_0 < v_1
< \cdots$ such that the non-saturated vertices in interval $\III$ are
precisely $\ldots$, $v_{-1}^{\III}$, $v_0^{\III}$, $v_1^{\III}$,
$\ldots$.

Note that we already considered some ``longest arcs'' in Notation
\ref{not:Case6}, and that hence,
\[
  a = b_{-1} \;,\; b = b_0 \;,\; c = b_1 \;,\; v = v_0 \;,\; w = v_1.
\]
\end{Description}

\begin{Construction}
[The triangulation $\fT$]
\label{con:Case6}
We add arcs to $\Theta( t )$ as follows to create a triangulation
$\fT$ of $D_4$; see Figure \ref{fig:ft_in_Case6} where the added arcs
are shown in red:

The vertices $\beta^{\II}$, $\gamma^{\II}$, $\varphi^{\IV}$,
$\chi^{\IV}$ will be explained at the end; for the time being,
consider them fixed and add the arcs $\{ \chi^{\IV},\beta^{\II} \}$
and $\{ \gamma^{\II},\varphi^{\IV} \}$.

Recall the numbers $\ell$, $m$, and $r$ from Lemma \ref{lem:DwR}.

From vertex $b_0^{\I}$, add $\ell$ arcs ending at the consecutive
vertices $\chi^{\IV}$, $( \chi+1 )^{\IV}$, $\ldots$, $\psi^{\IV}$.
Also from vertex $b_0^{\I}$, add $\defect_p( b_0 ) - \ell$ arcs ending
at the consecutive vertices $\beta^{\II}$, $( \beta - 1 )^{\II}$,
$\ldots$.  This makes sense because $\ell$ and $\defect_p( b_0 ) -
\ell$ are both positive by Lemma \ref{lem:DwR}(i).

From vertex $b_1^{\I}$, add $\defect_p( b_1 )$ arcs ending at the next
block of consecutive vertices in interval $\II$.  Continue in the same
fashion with vertices $b_2^{\I}$, $b_3^{\I}$, $\ldots$.

From vertex $b_{-1}^{\I}$, add $\defect_p( b_{-1} )$ arcs ending at the
next block of consecutive vertices in interval $\IV$.  Continue in the
same fashion with vertices $b_{-2}^{\I}$, $b_{-3}^{\I}$, $\ldots$.

Add arcs by a similar recipe from the vertices $\cdots$,
$v_{-1}^{\III}$, $v_0^{\III}$, $v_1^{\III}$, $\cdots$ using $m$,
$\defect_q$, $\gamma^{\II}$, $\varphi^{\IV}$, instead of $\ell$,
$\defect_p$, $\beta^{\II}$, $\chi^{\IV}$.

Now consider the finite polygon $R$ between the arcs $\{
\chi^{\IV},\beta^{\II} \}$ and $\{ \gamma^{\II},\varphi^{\IV} \}$.
Viewed in $R$, each of $\chi^{\IV}$, $\beta^{\II}$ and $\gamma^{\II}$,
$\varphi^{\IV}$ is a pair of adjacent vertices.  By Lemma
\ref{lem:DwR}(iii) we may apply Lemma \ref{lem:CC2}, and thus, if we
space each pair $\beta^{\II}$, $\gamma^{\II}$ and $\varphi^{\IV}$,
$\chi^{\IV}$ suitably, then there is a triangulation $\fS$ of $R$
which satisfies
\begin{align}
\nonumber
  r & = \fS( \chi^{\IV},\gamma^{\II} ) + \fS( \chi^{\IV},\varphi^{\IV} ), \\
\label{equ:rmb}
  t( b,v ) & = \fS( \chi^{\IV},\gamma^{\II} ) + \fS( \chi^{\IV},\varphi^{\IV} )
  + \fS( \beta^{\II},\gamma^{\II} ) + \fS( \beta^{\II},\varphi^{\IV} ).
\end{align}
The final step in constructing $\fT$ is to add to it the arcs in
$\fS$.
\end{Construction}

\begin{Proposition}
\label{pro:Case6}
The $\fT$ of Construction \ref{con:Case6} is a good triangulation of
$D_4$.
\end{Proposition}

\begin{proof}
Consider the finite polygons below the arcs $\fb_j$ and $\fv_j$ shown
in Figure \ref{fig:ft_in_Case6}.  In each such polygon, $\Theta( t )$
and hence $\fT$ restricts to a triangulation by Lemmas
\ref{lem:finite_triangulation1} through
\ref{lem:finite_triangulation2}.  The arcs added in Construction
\ref{con:Case6} (red in Figure \ref{fig:ft_in_Case6}) clearly complete
$\fT$ to a triangulation of $D_4$.  The added arcs also block the
accumulation points of $D_4$ so $\fT$ is a good triangulation.
\end{proof}

The following lemma collects several consequences of the Ptolemy
relation in Lemma \ref{lem:CC}(v) applied to $\fT$.

\begin{Lemma}
\label{lem:Case6_eqs}
The numbers $\ell$ and $m$ from Lemma \ref{lem:DwR} and the triangulation $\fT$ from Construction \ref{con:Case6} satisfy the following.
\begin{enumerate}
\setlength\itemsep{4pt}

  \item  

    \begin{enumerate}
    \setlength\itemsep{4pt}

      \item  $\fT( a^{\I},\beta^{\II} ) = \ell+1$.

      \item  $\fT( w^{\III},\gamma^{\II} ) = m + 1$.

    \end{enumerate}

  \item  $\fT( \chi^{\IV},v^{\III} )\fT( \varphi^{\IV},b^{\I} )
          \equiv 1 \bmod \fT( b^{\I},v^{\III} )$.

  \item  

    \begin{enumerate}
    \setlength\itemsep{4pt}

      \item  $\fT( \chi^{\IV},\gamma^{\II} ) + \fT( \chi^{\IV},\varphi^{\IV} )
              = \fT( \chi^{\IV},v^{\III} )$.

      \item  $\fT( \varphi^{\IV},\beta^{\II} ) + \fT( \varphi^{\IV},\chi^{\IV} )
              = \fT( \varphi^{\IV},b^{\I} )$.

      \item  $\fT( \beta^{\II},\gamma^{\II} ) + \fT( \beta^{\II},\varphi^{\IV} )
              = \fT( \beta^{\II},v^{\III} )$.%Not used outside proof of this lemma

    \end{enumerate}

  \item

    \begin{enumerate}
    \setlength\itemsep{4pt}

      \item  $\fT( \beta^{\II},v^{\III} ) + \fT( \chi^{\IV},v^{\III} )
               = \fT( b^{\I},v^{\III} )$.%Not used outside proof of this lemma

      \item  $\fT( \gamma^{\II},b^{\I} ) + \fT( \varphi^{\IV},b^{\I} )
               = \fT( b^{\I},v^{\III} )$.%Not used outside proof of this lemma

    \end{enumerate}

  \item  $\fT( b^{\I},v^{\III} ) = \fT( \chi^{\IV},\gamma^{\II} ) + \fT( \chi^{\IV},\varphi^{\IV} )
  + \fT( \beta^{\II},\gamma^{\II} ) + \fT( \beta^{\II},\varphi^{\IV} )$.

  \item

    \begin{enumerate}
    \setlength\itemsep{4pt}

      \item  $\big( \fT( a^{\I},\beta^{\II} ) - 1 \big)\fT( b^{\I},v^{\III} ) 
              + \fT( \chi^{\IV},v^{\III} )
              = \fT( a^{\I},v^{\III} )$.

      \item  $\big( \fT( w^{\III},\gamma^{\II} ) - 1 \big)\fT( b^{\I},v^{\III} ) 
              + \fT( \varphi^{\IV},b^{\I} )
              = \fT( b^{\I},w^{\III} )$.

    \end{enumerate}

\end{enumerate}
\end{Lemma}

\begin{proof}
(i)  Figure \ref{fig:ft_in_Case6} shows that the vertex set $\{ a^{\I},
b^{\I}, \beta^{\II}, \chi^{\IV}, \ldots, \psi^{\IV} \}$ is compatible
with $\fT$ in the sense of Definition \ref{def:restriction}.  These
vertices span a finite polygon $P$ and $\fT$ restricts to a
triangulation $\fT_P$ of $P$.  In $P$, the vertices $a^{\I}$,
$b^{\I}$, $\beta^{\II}$ are consecutive, so Lemma \ref{lem:CC}(vi)
gives
\[
  \fT_P( a^{\I},\beta^{\II} ) = 1 + (\mbox{the number of arcs in $\fT_P$
    ending at $b^{\I}$}) = 1 + \ell.
\]
By Remark \ref{rmk:CC} this implies part (i)(a), and part (i)(b)
follows by symmetry.

(ii)  If $\chi = \varphi$ then $\{ \chi^{\IV},v^{\III} \} = \{
\varphi^{\IV},v^{\III} \}$ and $\{ \varphi^{\IV},b^{\I} \} = \{
\chi^{\IV},b^{\I} \}$ are in $\fT$, so $\fT( \chi^{\IV},v^{\III} ) =
\fT( \varphi^{\IV},b^{\I} ) = 1$ and part (ii) holds even without the
congruence.

If $\chi \neq \varphi$ then the arcs $\{ \chi^{\IV},v^{\III} \}$ and
$\{ \varphi^{\IV},b^{\I} \}$ cross so the Ptolemy relation gives
\begin{align*}
  \fT( \chi^{\IV},v^{\III} )\fT( \varphi^{\IV},b^{\I} )
  & = \fT( \chi^{\IV},\varphi^{\IV} )\fT( b^{\I},v^{\III} ) 
    + \fT( \chi^{\IV},b^{\I} )\fT( \varphi^{\IV},v^{\III} ) \\
  & \equiv \fT( \chi^{\IV},b^{\I} )\fT( \varphi^{\IV},v^{\III} )
    \bmod \fT( b^{\I},v^{\III} ).
\end{align*}
This proves part (ii) because $\fT( \chi^{\IV},b^{\I} ) = \fT(
\varphi^{\IV},v^{\III} ) = 1$ since $\{ \chi^{\IV},b^{\I} \}, \{
\varphi^{\IV},v^{\III} \} \in \fT$.

(iii)  If $\chi = \varphi$ then part (iii)(a) claims
\[
  \fT( \varphi^{\IV},\gamma^{\II} ) + \fT( \chi^{\IV},\chi^{\IV} )
  = \fT( \varphi^{\IV},v^{\III} ).
\]
This equation just reads $1 + 0 = 1$ because $\{
\varphi^{\IV},\gamma^{\II} \}, \{ \varphi^{\IV},v^{\III} \} \in \fT$.

If $\chi \neq \varphi$ then the arcs $\{ \chi^{\IV},v^{\III} \}$ and
$\{ \gamma^{\II},\varphi^{\IV} \}$ cross so the Ptolemy relation gives
\[
  \fT( \chi^{\IV},v^{\III} )\fT( \gamma^{\II},\varphi^{\IV} )
  = \fT( \chi^{\IV},\gamma^{\II} )\fT( v^{\III},\varphi^{\IV} )
    + \fT( \chi^{\IV},\varphi^{\IV} )\fT( v^{\III},\gamma^{\II} ).
\]
This proves (iii)(a) because $\fT( \gamma^{\II},\varphi^{\IV} ) = \fT(
v^{\III},\varphi^{\IV} ) = \fT( v^{\III},\gamma^{\II} ) = 1$ since $\{
\gamma^{\II},\varphi^{\IV} \}$, $\{ v^{\III},\varphi^{\IV} \}$, $\{
v^{\III},\gamma^{\II} \}$ $\in \fT$.  Parts (iii)(b) and (iii)(c)
follow by symmetry.

(iv)  The arcs $\{ \beta^{\II},\chi^{\IV} \}$ and $\{ b^{\I},v^{\III}
\}$ cross so the Ptolemy relation gives
\[
  \fT( \beta^{\II},\chi^{\IV} )\fT( b^{\I},v^{\III} )
  = \fT( \beta^{\II},b^{\I} )\fT( \chi^{\IV},v^{\III} )
    + \fT( \beta^{\II},v^{\III} )\fT( \chi^{\IV},b^{\I} ).
\]
This proves (iv)(a) because $\fT( \beta^{\II},\chi^{\IV} ) = \fT(
\beta^{\II},b^{\I} ) = \fT( \chi^{\IV},b^{\I} ) = 1$ since $\{
\beta^{\II},\chi^{\IV} \}$, $\{ \beta^{\II},b^{\I} \}$, $\{
\chi^{\IV},b^{\I} \}$ $\in \fT$.  Part (iv)(b) is follows by symmetry.

(v)  Combine parts (iii)(a), (iii)(c), and (iv)(a).

(vi)  The arcs $\{ a^{\I},\beta^{\II} \}$ and $\{ b^{\I},v^{\III} \}$
cross so the Ptolemy relation gives
\[
  \fT( a^{\I},\beta^{\II} )\fT( b^{\I},v^{\III} )
  = \fT( a^{\I},b^{\I} )\fT( \beta^{\II},v^{\III} )
    + \fT( a^{\I},v^{\III} )\fT( \beta^{\II},b^{\I} ).
\]
We have $\fT( a^{\I},b^{\I} ) = \fT( \beta^{\II},b^{\I} ) = 1$ since
$\{ a^{\I},b^{\I} \}, \{ \beta^{\II},b^{\I} \} \in \fT$, so the
equation reads
\[
  \fT( a^{\I},\beta^{\II} )\fT( b^{\I},v^{\III} )
  = \fT( \beta^{\II},v^{\III} ) + \fT( a^{\I},v^{\III} ).
\]
Combining with part (iv)(a) gives
\[
  \fT( a^{\I},\beta^{\II} )\fT( b^{\I},v^{\III} )
  = \fT( b^{\I},v^{\III} ) - \fT( \chi^{\IV},v^{\III} ) + \fT( a^{\I},v^{\III} )
\]
which can be reorganised into (vi)(a).  Part (vi)(b) follows by
symmetry. 
\end{proof}

\begin{Theorem}
\label{thm:Case6}
Let $t$ be an $\SL2$-tiling with no entry equal to $1$.

Then there is a good triangulation $\fT$ of $D_4$ such that $\Phi( \fT
) = t$.
\end{Theorem}

\begin{proof}
Let $\fT$ be as in Construction \ref{con:Case6}, see Figure
\ref{fig:ft_in_Case6}.  It was proved in Proposition \ref{pro:Case6}
that $\fT$ is a good triangulation of $D_4$.  To show $\Phi( \fT ) =
t$ we use Lemma \ref{lem:agree} in which we first verify condition
(iii).

By construction, in the finite polygon $R$ between the arcs $\{
\chi^{\IV},\beta^{\II} \}$ and $\{ \gamma^{\II},\varphi^{\IV} \}$, the
triangulation $\fT$ agrees with the triangulation $\fS$ featured in
Equations \eqref{equ:rmb} which can hence be rewritten with $\fT$
instead of $\fS$.  Combining with Lemma \ref{lem:Case6_eqs}(iii)(a)
gives
\begin{align}
\label{equ:rm2a}
  r & = \fT( \chi^{\IV},v^{\III} ), \\
\label{equ:rm2b}
  t( b,v ) & = \fT( \chi^{\IV},\gamma^{\II} ) + \fT( \chi^{\IV},\varphi^{\IV} )
  + \fT( \beta^{\II},\gamma^{\II} ) + \fT( \beta^{\II},\varphi^{\IV} ).
\end{align}

Combining Equation \eqref{equ:rm2b} with Lemma \ref{lem:Case6_eqs}(v)
shows
\begin{equation}
\label{equ:last1}
  t( b,v ) = \fT( b^{\I},v^{\III} ).
\end{equation}
Combining this with Lemma \ref{lem:DwR}, Lemma
\ref{lem:Case6_eqs}(i)(a), and Equation \eqref{equ:rm2a} shows
\[
  t( a,v )
  = \ell t( b,v ) + r
  = \big( \fT( a^{\I},\beta^{\II} ) - 1 \big)\fT( b^{\I},v^{\III} )
    + \fT( \chi^{\IV},v^{\III} ) 
  = (*)
\]
and Lemma \ref{lem:Case6_eqs}(vi)(a) gives
\[
  (*) = \fT( a^{\I},v^{\III} ).
\]

Now, on the one hand, Lemma \ref{lem:DwR}(iii) says that $0 < r,s < t(
b,v )$ and that $r$ and $s$ are inverses modulo $t( b,v )$.  On the
other hand, Lemma \ref{lem:Case6_eqs}, (iii)(a), (iii)(b), and (v),
imply that $0 < \fT( \chi^{\IV},v^{\III} ), \fT( \varphi^{\IV},b^{\I}
) < \fT( b^{\I},v^{\III} )$ and Lemma \ref{lem:Case6_eqs}(ii) says
that $\fT( \chi^{\IV},v^{\III} )$ and $\fT( \varphi^{\IV},b^{\I} )$
are inverses modulo $\fT( b^{\I},v^{\III} )$.  Combining with
Equations \eqref{equ:rm2a} and \eqref{equ:last1} shows $s = \fT(
\varphi^{\IV},b^{\I} )$.  We can now proceed as above, combining this
with Lemma \ref{lem:DwR} and Lemma \ref{lem:Case6_eqs}(i)(b) to get
\[
  t( b,w ) = m t( b,v ) + s
  = \big( \fT( w^{\III},\gamma^{\II} ) - 1 \big)\fT( b^{\I},v^{\III} )
    + \fT( \varphi^{\IV},b^{\I} )
  = (**),
\]
and Lemma \ref{lem:Case6_eqs}(vi)(b) says
\[
  (**) = \fT( b^{\I},w^{\III} ).
\]

Combining the last five displayed equations verifies Lemma
\ref{lem:agree}, condition (iii), with $e = a$, $f = b$, $g = v$, $h =
w$.

Finally, Lemma \ref{lem:agree}, conditions (i) and (ii) are verified
by the same method as in the second half of the proof of Theorem
\ref{thm:Case2}.
\end{proof}

\medskip
\noindent
{\bf Acknowledgement.}
This work was carried out while Peter J\o rgensen was visiting the
Leibniz Universit\"{a}t Hannover.  He thanks Thorsten Holm and the
Institut f\"{u}r Algebra, Zahlentheorie und Diskrete Mathematik for
their hospitality.  He also gratefully acknowledges financial support
from Thorsten Holm's grant HO 1880/5-1, which is part of the research
priority programme SPP 1388 {\em Darstellungstheorie} of the Deutsche
Forschungsgemeinschaft (DFG).

\end{document}